\documentclass[12pt]{article}

\usepackage{amsmath, amsfonts, amssymb}
\usepackage{amsthm}
\usepackage{thmtools} 

\usepackage{microtype}
\usepackage{dsfont}
\usepackage{appendix}
\usepackage{enumitem}

\usepackage[natbib=true, 
            backend=bibtex8, 
            style=ext-authoryear, 
            url=false, doi=false, articlein=false, clearlang=true, isbn=false, date=year, giveninits=true, maxcitenames=2, maxbibnames=999, uniquename=init]{biblatex}
\DeclareInnerCiteDelims{cite}{\bibopenparen}{\bibcloseparen}
\DeclareNameAlias{sortname}{family-given}
\usepackage{graphicx}

\usepackage{geometry}
\usepackage{titlesec}

\titleformat{\paragraph}
{\normalfont\normalsize\bfseries}{\theparagraph}{1em}{}
\titlespacing*{\paragraph}
{0pt}{3.25ex plus 1ex minus .2ex}{1.5ex plus .2ex}

\usepackage{hyperref}
\hypersetup{
            colorlinks={true}, 
            linkcolor={blue}, 
            citecolor=blue,
            pdfencoding=auto, 
            psdextra
            }
            
\usepackage[capitalise,compress]{cleveref} %

\numberwithin{equation}{section}

\declaretheorem[name=Definition,
                refname={definition, definitions},
                Refname={Definition, Definitions},
                numberwithin=section]{definition} 

\declaretheorem[name=Assumption, 
                refname={assumption,assumptions},
                Refname={Assumption,Assumptions},
                numberwithin=section]{assumption}

\declaretheorem[name=Lemma,
                refname={lemma, lemmas},
 			Refname={Lemma, Lemmas},
 			numberwithin=section]{lemma}

\declaretheorem[name=Theorem,
               refname={theorem,theorems},
               Refname={Theorem,Theorems},
               numberwithin=section]{theorem}

\declaretheorem[name=Proposition,
                refname={proposition, propositions},
 			refname={Proposition, Propositions},
 			numberwithin=section]{proposition}
               
\declaretheorem[name=Corollary,
                refname={corollary, corollaries},
                Refname={Corollary, Corollaries},
                numberwithin=section]{corollary}

\declaretheorem[name=Remark, 
                refname={remark, remarks},
                Refname={Remark, Remarks},
                numberwithin=section]{remark}

\declaretheorem[name=Example,
                refname={example, examples},
                Refname={Example, Examples},
                numberwithin=section]{example}

\newlist{assumpenum}{enumerate}{1} %
\setlist[assumpenum]{label=(\roman*), ref=\theassumption~(\roman*)}
\crefalias{assumpenumi}{assumption} 

\newlist{lemmaenum}{enumerate}{1} %
\setlist[lemmaenum]{label=(\roman*), ref=\thelemma~(\roman*)}
\crefalias{lemmaenumi}{lemma} 

\newlist{propenum}{enumerate}{1} %
\setlist[propenum]{label=(\roman*), ref=\theproposition~(\roman*)}
\crefalias{propenumi}{proposition} 

\newcommand{\indep}{\rotatebox[origin=c]{90}{$\models$}}

\DeclareMathOperator{\argmin}{\mathop{\mathrm{argmin}}}

\begin{document}

\title{Quantifying Distributional Model Risk in Marginal Problems via Optimal Transport\thanks{We acknowledge valuable feedback from participants of Optimization-Conscious Econometrics Conference II at the University of Chicago, KI+Scale MoDL Retreat at the University of Washington, and Econometrics and Optimal Transport Workshop at the University of Washington. Fan acknowledges support from NSF Infrastructure grant (PIHOT) DMS-2133244.}}

\author{Yanqin Fan,\footnote{Department of Economics, University of Washington. Email: fany88@uw.edu} \ Hyeonseok Park,\footnote{Institute for Advanced Economic Research, Dongbei University of Finance and Economics. Email: hynskpark21@dufe.edu.cn} \ and Gaoqian Xu\footnote{Department of Economics, University of Washington. Email: gx8@uw.edu} }
\date{\today}

\maketitle

\begin{abstract}
This paper studies distributional model risk in marginal problems, where each marginal measure is assumed to lie in a Wasserstein ball centered at a fixed reference measure with a given radius. Theoretically, we establish several fundamental results including strong duality, finiteness of the proposed Wasserstein distributional model risk, and the existence of an optimizer at each radius. In addition, we show continuity of the Wasserstein distributional model risk as a function of the radius.  Using strong duality, we extend the well-known Makarov bounds for the distribution function of the sum of two random variables with given marginals to Wasserstein distributionally robust Markarov bounds. Practically, we illustrate our results on four distinct applications when the sample information comes from multiple data sources and only some marginal reference measures are identified. They are: partial identification of treatment effects; externally valid treatment choice via robust welfare functions; Wasserstein distributionally robust estimation under data combination; and evaluation of the worst aggregate risk measures. 
\end{abstract}

\newpage

\section{Introduction}

\textit{Distributionally robust optimization} (DRO) has emerged as a powerful tool for hedging against model misspecification and distributional shifts. It minimizes \textit{distributional model risk} (DMR) defined as the worst risk over a class of distributions lying in a \textit{distributional uncertainty set}, see \citet{blanchet2019quantifying}. Among many different choices of uncertainty sets, Wasserstein DRO (W-DRO) with distributional uncertainty sets based on optimal transport costs has gained much popularity, see \citet{Kuhn_2019} and \citet{Blanchet2021} for recent reviews. W-DRO  has found successful applications in robust decision making in all disciplines including economics, finance, machine learning, and operations research. Its success is largely credited to the strong duality and other nice properties of the Wasserstein DMR (W-DMR). The objective of this paper is to propose and study W-DMR in \textit{marginal problems where only some marginal measures of a reference measure are given}, see e.g., \citet{Kellerer:1984to,rachev1998mass,villani2009optimal, villani2021topics}, and \citet{ruschendorf1991bounds}.

In practice, \textit{marginal problems} arise from either the lack of complete data or an incomplete model. In insurance and risk management, computing model-free measures of aggregate risks such as Value-at-Risk and Expected Short-Fall is of utmost importance and routinely done. When the exact dependence structure between individual risks is lacking, researchers and policy makers rely on the worst risk measures defined as the maximum value of aggregate risk measures over all joint measures of the individual risks with some fixed marginal measures, see \citet{Embrechts_2010} and \citet{Embrechts_2013};  In causal inference, distributional treatment effects such as the variance and the proportion of participants who benefit from the treatment depend on the joint distribution of the potential outcomes. Even with ideal randomized experiments such as double-blind clinical trials, the joint distribution of potential outcomes is not identified and as a result, only the lower and upper bounds on distributional treatment effects are identified from the sample information, see \citet{FAN_2009}, \citet{Fan_2010_SharpBounds,Fan_2012_CI_quantiles}, \citet{Fan2017}, \citet{Ridder_2007}, \citet{Firpo_2019};  In algorithmic fairness when the sensitive group variable is not observed in the main data set, assessment of unfairness measures must be done using multiple data sets, see \citet{Kallus2022}. Abstracting away from estimation, all these problems involve optimizing the expected value of a functional of multiple random variables with fixed marginals and thus belong to the class of marginal problems for which optimal transport related tools are important.\footnote{When the marginals are univariate, optimal transport problem can be conveniently expressed in terms of copulas. \citet{Fan_2010_SharpBounds}, \citet{Fan_2012_CI_quantiles}, \citet{FAN_2009}, \citet{Fan2017}, \citet{Ridder_2007},  and \citet{Firpo_2019} explicitly use copula tools. } 

The marginal measures in the afore-mentioned applications and general marginal problems are typically empirical measures computed from multiple data sets such as in the evaluation of worst aggregate risk measures or identified under specific assumptions such as randomization or strong ignorability in causal inference. Developing a unified framework for hedging against model misspecification and/or distributional shifts in marginal measures motivates the current paper. 

 Theoretically, this paper makes several contributions to the literature on distributional robustness and the literature on marginal problems. First, it introduces \textit{Wasserstein distributional model risk in marginal problems} (W-DMR-MP), where each marginal measure is assumed to lie in a Wasserstein ball centered at a fixed reference measure with a given radius. We focus on the important case with two marginals and consider both non-overlapping and overlapping marginals. For \textit{non-overlapping marginal measures}, when the radius is zero, the W-DMR-MP reduces to the marginal problems or optimal transport problems studied in \citet{Kellerer:1984to,rachev1998mass,villani2009optimal, villani2021topics}. For \textit{overlapping marginals}, when the radius is zero, the W-DMR-MP reduces to the overlapping marginals problem studied in \citet{ruschendorf1991bounds}; Second, we establish strong duality for our W-DMR with both non-overlapping and overlapping marginals under similar conditions to those for W-DMR, see \citet{zhang2022simple}, \citet{blanchet2019quantifying}, and \citet{gao2022WDRO}. As a first application of our strong duality result for non-overlapping marginals, we extend the well-known Marakov bounds for the distribution function of the sum of two random variables to Wasserstein distributionally robust Makarov bounds; Third, we prove finiteness of the W-DMR-MP and existence of an optimizer at each radius. Based on both results, we show that the identified set of the expected value of a smooth functional of random variables with fixed marginals is a closed interval; Fourth, we show continuity of the W-DMR in marginal problems as a function of the radius. Together these results extend those for W-DMR in \citet{blanchet2019quantifying}, \citet{zhang2022simple}, and \citet{yue2022linear}; Lastly, we extend our formulations and theory to W-DMR with multi-marginals. On a technical note, our proofs build on existing work on W-DMR such as \citet{blanchet2019quantifying}, \citet{zhang2022simple}, and \citet{yue2022linear}. However, an additional challenge due to the presence of multiple marginal measures in our Wasserstein uncertain sets is the verification of the existence of a joint measure with overlapping marginals. We make use of existing results for a given consistent product marginal system in \citet{vorob1962consistent}, \citet{kellerer1964verteilungsfunktionen}, and \citet{shortt1983combinatorial} to address this issue. 
 
 Practically, we demonstrate the flexibility and broad applicability of our W-DMR-MP via four distinct applications when the sample information comes from multiple data sources. First, we consider partial identification of treatment effects when the marginal measures of the potential outcomes lie in their respective Wasserstein balls centered at the measures identified under strong ignorability. The validity of strong ignorability is often questionable when unobservable confounders may be present. We apply our W-DMR-MP to establishing the identified sets of treatment effects which can be used to conducting sensitivity analysis to the selection-on-observables assumption. For average treatment effects, we show that when the cost functions are separable, incorporating covariate information does not help shrink the identified set; on the other hand, for non-separable cost functions such as the Mahalanobis distance, incorporating covariate information may help shrink the identified set; Second, in causal inference when the optimal treatment choice is to be applied to a target population different from the training population, \citet{adjaho2022} introduces robust welfare functions defined by W-DMR to study externally valid treatment choice. The W-DMR-MP we propose allows us to dispense with the assumption of a \textit{known dependence structure} for the reference measure in \citet{adjaho2022}. When shifts in the covariate distribution are allowed, we show that our robust welfare function is upper bounded by the worst robust welfare function of \citet{adjaho2022}; Third, one important application of W-DMR is in distributionally robust estimation and classification. However as \citet{awasthi2022distributionally} points out,\footnote{See \citet{Graham2016} and \citet{Chen_2008} for general data combination problems.} some sensitive variables may not be observed in the same data set as the response variable rendering W-DRO inapplicable. We apply W-DMR-MP to distributionally robust estimation under data combination;\footnote{\Cref{sec:example-logit} provides a detailed comparison of our set up and \citet{awasthi2022distributionally}.} Fourth, applying our W-DMR-MP to the evaluation of the worst aggregate risk measures allows us to dispense with the known marginals assumption in \citet{Embrechts_2010} and \citet{Embrechts_2013}.

The rest of this paper is organized as follows. \Cref{sec:standard-DRO} reviews the W-DMR and strong duality, introduces our W-DMR-MP, and then presents four motivating examples. \Cref{sec:WDRO-main-theory} establishes strong duality and Wasserstein distributionally robust Marakov bounds. \Cref{sec:additional-properties} studies finiteness of W-DMR-MP and existence of optimal solutions. Moreover, we show that the identified set of the expected value of a smooth functional of random variables with fixed marginals is a closed interval. \Cref{sec:continuity} establishes continuity of W-DMR-MP as a function of the radius. \Cref{sec:examples-revisited} revisits the motivating examples in \Cref{sec:standard-DRO}. \Cref{sec:multi-marginals} extends our W-DMR-MP to more than two marginals. The last section offers some concluding remarks. Technical proofs are relegated to a series of appendices.

We close this section by introducing the notation used in the rest of this paper. 
For two sets $A$ and $B$, the relative complement is denoted by $A \setminus B$. Let $\overline{\mathbb{R}} = \mathbb{R} \cup \left\{- \infty, \infty \right \}$, $[d]=\{1,2,...,d\}$,  $\mathbb{R}^d_{+} = \left \{ x \in \mathbb{R}^d : x_i \geq 0, \ \forall i \in [d] \right\}$, and $\mathbb{R}^d_{++} = \left\{ x \in \mathbb{R}^d : x_i > 0, \ \forall i \in [d] \right\}$. For any real numbers $x, y \in \mathbb{R}$, we define $x \wedge y := \min\{x, y\}$ and $x \vee y : = \max\{x, y\}$. The Euclidean inner product of $x$ and $y$ in $\mathbb{R}^d$ is denoted by $\langle x, y \rangle$. For any real matrix $W \in \mathbb{R}^{m \times n}$, let $A^\top$ denote the transpose of $W$. For an extended real function $f$ on $\mathcal{X}$, the positive part $f^+$ and the negative part $f^-$ are defined as $f^{+}(x) = \max \left\{ f(x), 0 \right\}$ and $f^{-}(x) = \max \left\{ -f(x), 0 \right\}$, respectively.

For any Polish space $\mathcal{S}$, let $\mathcal{B}_\mathcal{S}$ be the 
associated Borel $\sigma$-algebra and $\mathcal{P}(\mathcal{S})$ be the collection of probability measures on $\mathcal{S}$. Given a Polish probability space $(\mathcal{S}, \mathcal{B}_{\mathcal{S}}, \nu)$, let $\mathcal{B}_{\mathcal{S}}^\nu$ denote the $\nu$-completion of $\mathcal{B}_{\mathcal{S}}$.  Given a probability space $\left(\Omega, \mathcal{F}, \mathbb{P} \right)$ and a map $T: \Omega \rightarrow \mathcal{S}$, let $T\# \mu$ denote the push forward of $\mathbb{P}$ by $T$, i.e., $(T\# \mathbb{P})(A) = \mathbb{P} \left( T^{-1} (A) \right)$ for all $A \in \mathcal{B}_{\mathcal{S}}$, where $T^{-1}(A) =  \left\{ \omega\in\Omega: T(\omega) \in A \right\}$. The law of a random variable $S: \Omega \rightarrow \mathbb{R}$ is denoted by $\mathrm{Law}(S)$ which is the same as $S\# \mathbb{P}$. For any $\mu, \nu \in \mathcal{P}(\mathcal{S})$, let $\Pi(\mu, \nu)$ denote the set of all couplings (or joint measures) with marginals $\mu$ and $\nu$.

For any $\mathcal{B}_{\mathcal{S}}^{\nu}$-measurable function $f$, let $\int_{\mathcal{S}} f d \nu$ denote the integral of $f$ in the completion of $(\mathcal{S}, \mathcal{B}_{\mathcal{S}}, \nu)$. For a random element $S: \Omega \rightarrow \mathcal{S}$ with $\mathrm{Law}(S) = \nu$, we write $\mathbb{E}_\nu \left[f(S)\right] = \int_{\mathcal{S}} f d \nu$. Given $p \in (0, \infty)$ and a Borel measure $\nu$ on $\mathcal{S}$, let $L^p(\nu) := L^p(\mathcal{S},\mathcal{B}_{\mathcal{S} } , \nu)$ denote the set of all the $\mathcal{B}_{\mathcal{S}}^{\nu}$-measurable functions $f: \mathcal{S}\rightarrow \mathbb{R}$ such that $\| f\|_{L^p(\nu)} := \left( \int_{\mathcal{S}} |f|^p d\nu \right)^{1/p}< \infty$.

\section{W-DMR and Motivating Examples} \label{sec:standard-DRO}

In this section, we first review W-DMR and then introduce W-DMR in marginal problems. Lastly, we present four motivating examples of marginal problems which will be used to illustrate our results in the rest of this paper.

\subsection{A Review of W-DMR and Strong Duality}

W-DMR is defined as the worst model risk over a class of distributions lying in a Wasserstein uncertainty set composed of all probability measures that are a fixed Wasserstein distance away from a given reference measure, see \citet{blanchet2019quantifying}.

Before presenting W-DMR, we review some basic definitions. Let $\mathcal{X}$ be a Polish (metric) space with a metric $\boldsymbol{d}$.

\begin{definition}[Optimal transport cost] \label{Wasserstein distance} 
  Let $\mu, \nu \in \mathcal{P}(\mathcal{X})$ be given probability measures. The {\it optimal transport cost} between $\mu$ and $\nu$ associated with a cost function $c: \mathcal{X}\times \mathcal{X} \rightarrow \mathbb{R}_+ \cup \{\infty\}$ is defined as
\[
 \boldsymbol{K}_c(\mu, \nu) = \inf_{\pi \in \Pi(\mu, \nu)} \int_{\mathcal{X} \times \mathcal{X}} c \, d\pi.
\]
\end{definition}

When the cost function $c$ is lower-semicontinuous, there exists an optimal coupling corresponding to $\boldsymbol{K}_c(\mu, \nu)$. In other words, there exists $\pi^* \in \Pi(\mu, \nu)$ such that $\boldsymbol{K}_c(\mu, \nu)=\int_{\mathcal{X} \times \mathcal{X}} c \, d\pi^*$ \citep[Ref.][Theorem 4.1]{villani2009optimal}. 

\begin{definition}[Wasserstein distance] \label{def:Wp-distance}
Let $p \in [1, \infty)$. The {\it Wasserstein distance} of order $p$ between any two measures $\mu$ and $\nu$ on Polish metric space $(\mathcal{X}, \boldsymbol{d})$ is defined by 
\begin{align*}
\boldsymbol{W}_{ p}(\mu, \nu) &=  \left[ \inf_{\pi \in \Pi(\mu, \nu)} \int_{\mathcal{X} \times \mathcal{X} }  \boldsymbol{d}^p \, d \pi \right]^{1/p}.
\end{align*}
\end{definition}

Throughout this paper, we make the following assumption on the cost function $c$.
\begin{assumption} \label{assumption:cost-function} 
Let $(\mathcal{X}, \mathcal{B}_{\mathcal{X}})$ be a Borel space associated to $\mathcal{X}$. The cost function $c: \mathcal{X} \times \mathcal{X} \to \mathbb{R}_{+} \cup \{\infty\}$ is measurable and satisfies $c(x, y) = 0$ if and only if $x = y$. 
\end{assumption}

\Cref{assumption:cost-function} implies that for $\mu, \nu \in \mathcal{P}(\mathcal{X})$, $\mu = \nu$ if and only if $\boldsymbol{K}_c(\mu, \nu) = 0$. When $c$ is the metric $\boldsymbol{d}$ on $\mathcal{X}$,  $ \boldsymbol{K}_c(\mu, \nu)$ coincides with the Wasserstein distance of order 1 (Kantorovich-Rubinstein distance) between $\mu$ and $\nu$ defined in \Cref{def:Wp-distance}.

For a given function $f: \mathcal{X} \to \mathbb{R}$, \citet{blanchet2019quantifying} define W-DMR as
\begin{align*}
	\mathcal{I}_{\mathrm{DMR}}(\delta) := \sup_{\gamma \in \Sigma_{ \mathrm{DMR} }(\delta)} \int_{
 \mathcal{X}} f \, d \gamma, \mbox{ } \delta\ge 0,
\end{align*}
where $\Sigma_{\mathrm{DMR}}(\delta)$ is the Wasserstein uncertainty set\footnote{By convention, we call all uncertainty sets based on optimal transport costs as Wasserstein uncertainty sets.} centered at a reference measure $\mu \in \mathcal{P}(\mathcal{X})$ with radius $\delta\geq 0$, i.e.,
\begin{align*}
\Sigma_{\mathrm{DMR}}(\delta):=\left\{\gamma \in \mathcal{P}(\mathcal{X}) : \boldsymbol{K}_c(\mu, \gamma) \le \delta \right\}.
\end{align*}
\Cref{assumption:cost-function} allows the cost function $c$ to be asymmetric and take value $\infty$, where the latter corresponds to the case that there is no distributional shift in some marginal measure of $\mu$. 
\begin{remark}
Under \Cref{assumption:cost-function}, $\Sigma_{\mathrm{DMR}}(0)=\{\mu\}$ and 
\[	
\mathcal{I}_{\mathrm{DMR}}(0) = \int_{\mathcal{X}} f \, d \mu.
\]
\end{remark}

It is well-known that under mild conditions, strong duality holds for $\mathcal{I}_{\mathrm{DMR}}(\delta)$ when $\delta > 0$ (c.f., \citet{blanchet2019quantifying,gao2022WDRO,zhang2022simple}). To be self-contained, we restate the strong duality result in \citet{zhang2022simple} for Polish space below.\footnote{The strong duality result in \citet{zhang2022simple} allows for general space $\mathcal{X}$.}
    
\begin{theorem}[{\citet[Theorem 1]{zhang2022simple}}] Let $(\mathcal{X}, \mathcal{B}_{\mathcal{X}}, \mu)$ be a probability space. Let $\delta \in (0, \infty)$ and $f: \mathcal{X} \to \mathbb{R}$ be a measurable function such that $\int_{\mathcal{X}} f \, d \mu > -\infty$. Suppose the cost function satisfies \Cref{assumption:cost-function}. Then, for any $\delta > 0$, 
	\begin{align}
		\mathcal{I}_{\mathrm{DMR}}(\delta) = \inf_{\lambda \in \mathbb{R}_{+}}
		\left\{ \lambda \delta + \int_{\mathcal{X}} \sup_{x' \in \mathcal{X}} [ f(x') - \lambda c(x, x')  ] \, d \mu(x) \right\}, \label{eq:DRO-duality}
	\end{align}
	where $\lambda c(x, x')$ is defined to be $\infty$ when $\lambda = 0$ and $c(x, x') = \infty$.
\end{theorem}
In the rest of this paper, we keep the convention that for any cost function $c$, $\lambda c(x, y)  = \infty$ when $\lambda = 0$ and $c(x, y) = \infty$.

\subsection{W-DMR in Marginal Problems} 

\subsubsection{Non-overlapping Marginals} 

Let $\mathcal{V} := \mathcal{S}_1 \times \mathcal{S}_2$ be the product space of two Polish spaces $\mathcal{S}_1$ and $\mathcal{S}_2$. Let $\mu_1$ and $\mu_2$ be Borel probability measures on $\mathcal{S}_1$ and $\mathcal{S}_2$ respectively. Following \citet{ruschendorf1991bounds} (see also \citet{Embrechts_2010}), we call the Fr\'{e}chet class of all probability measures on $\mathcal{V}$ having marginals $\mu_1$ and $\mu_2$ the Fr\'{e}chet class with non-overlapping marginals denoted as $\mathcal{F}(\mathcal{V}; \mu_{1}, \mu_{2}) :=  \mathcal{F}(\mu_{1}, \mu_{2})$. Note that $\mathcal{F}(\mu_{1}, \mu_{2})=\Pi(\mu_1, \mu_2)$.

Let $g:\mathcal{V}\rightarrow \mathbb{R}$ be a measurable function satisfying the following assumption.

\begin{assumption}\label{assumption:g-bounded-below} 
The function $g:\mathcal{V}\rightarrow \mathbb{R}$ is measurable such that $\int_{\mathcal{V}} g d\gamma_0 > -\infty$ for some $\gamma_0 \in \Pi(\mu_1, \mu_2) \subset \mathcal{P}(\mathcal{V})$.
\end{assumption}

The marginal problem associated with $\mu_1$ and $\mu_2$ is defined as
\begin{align*}
	\mathcal{I}_{\mathrm{M} }(\mu_1,\mu_2):= \sup_{\gamma \in \Pi(\mu_1, \mu_2)} \int_{\mathcal{V}} g \, d \gamma. 
\end{align*}
It is essentially an optimal transport problem, where the $\sup$ operation is replaced with the $\inf$ operation, see \citet{Kellerer:1984to,rachev1998mass,villani2009optimal, villani2021topics}) or \Cref{appendix:duality-OT} for a review of strong duality for $\mathcal{I}_{\mathrm{M} }(\mu_1,\mu_2)$.

 The W-DMR with non-overlapping marginals we propose extends the marginal problem by allowing each marginal measure of $\gamma$ to lie in a fixed Wasserstein distance away from a reference measure. Specifically, for any $\gamma \in \mathcal{P}(\mathcal{V} )$, let $\gamma_{1}$ and $\gamma_{2}$ denote the projection of $\gamma$ on $\mathcal{S}_1$ and $\mathcal{S}_2$, respectively. The W-DMR with non-overlapping marginals is defined as 
\begin{equation} \label{eq:ID-primal}
\mathcal{I}_{\mathrm{D} }(\delta) := \sup_{\gamma \in \Sigma_{\mathrm{D} }(\delta)} \int_{\mathcal{V} } g \, d\gamma,  \quad \delta \in \mathbb{R}_{+}^2,
\end{equation}
where $\Sigma_{\mathrm{D} }(\delta)$ is the uncertainty set given by
\[%
\Sigma_{\mathrm{D} }(\delta) := \Sigma_{\mathrm{D} }(\mu_1, \mu_2, \delta) = \left\{ \gamma \in \mathcal{P}(\mathcal{V}):  \ \boldsymbol{K}_1( \mu_1, \gamma_1) \leq \delta_1, \   \boldsymbol{K}_2( \mu_2, \gamma_2) \leq \delta_2 \right\},
\]
in which $\boldsymbol{K}_1$ and $\boldsymbol{K}_2$ are optimal transport costs associated with cost functions $c_1$ and $c_2$, respectively, and $\delta := (\delta_{1}, \delta_{2}) \in \mathbb{R}_{+}^2$ is the radius of the uncertainty set. Obviously $\Sigma_{\mathrm{D} }(\delta)$ is non-empty for all $\delta \in \mathbb{R}_{+}^2$.

\begin{remark}
(i) Under \Cref{assumption:cost-function} and \Cref{assumption:g-bounded-below}, it holds that $\mathcal{I}_{\mathrm{D} }(\delta) > - \infty$ for all $\delta \in \mathbb{R}_{+}^2$, see \Cref{lemma:ID-concavity};
(ii) Under \Cref{assumption:cost-function}, the uncertainty set $\Sigma_{\mathrm{D} }(0)=\Pi(\mu_{1}, \mu_{2})$ and thus $\mathcal{I}_{\mathrm{D} }(0) =\mathcal{I}_{\mathrm{M} }(\mu_1,\mu_2)$.
\end{remark}

\subsubsection{Overlapping Marginals}

Let $\mathcal{S} := \mathcal{Y}_1 \times \mathcal{Y}_2 \times \mathcal{X}$ be the product space of three Polish spaces $\mathcal{Y}_1$, $\mathcal{Y}_2$, and $\mathcal{X}$. Let $\mathcal{S}_1 := \mathcal{Y}_1 \times \mathcal{X}$ and $\mathcal{S}_2 := \mathcal{Y}_2 \times \mathcal{X}$. Let $\mu_{13} \in \mathcal{P}(\mathcal{S}_1)$ and $\mu_{23} \in \mathcal{P}(\mathcal{S}_2)$ be such that the projection of $\mu_{13}$ and the projection of $\mu_{23}$ on $\mathcal{X}$ are the same. Following \citet{ruschendorf1991bounds} (see also \citet{Embrechts_2010}), we call the Fr\'{e}chet class of all probability measures on $\mathcal{S}$ having marginals $\mu_{13}$ and $\mu_{23}$ the Fr\'{e}chet class with overlapping marginals and denote it as $\mathcal{F}(\mathcal{S}; \mu_{13}, \mu_{23}) :=  \mathcal{F}(\mu_{13}, \mu_{23}).$  Unlike the non-overlapping case, $\mathcal{F}(\mu_{13}, \mu_{23})$ is different from the class of couplings $\Pi(\mu_{13}, \mu_{23})$.

Let $f:\mathcal{S} \rightarrow \mathbb{R}$ be a measurable function satisfying the following assumption. 
\begin{assumption} \label{assumption:f-bounded-below} The function $f:\mathcal{S} \rightarrow \mathbb{R}$ is measurable such that $\int_{\mathcal{S}} f d\nu_0 > - \infty$ for some $\nu_0 \in \mathcal{F}(\mu_{13}, \mu_{23}) \subset \mathcal{P}(\mathcal{S})$.
\end{assumption}

    \citet{ruschendorf1991bounds} studies the following marginal problem with overlapping marginals:
		\begin{align*}
			\mathcal{I}_{\mathrm{M} }(\mu_{13},\mu_{23}):=\sup_{\gamma \in \mathcal{F}(\mu_{13}, \mu_{23})} \int_{\mathcal{S}} f \, d \gamma. 
		\end{align*}
	As shown in \citet{ruschendorf1991bounds}, the marginal problem with overlapping marginals can be computed via the marginal problem with non-overlapping marginals through the following relation:
 \begin{align*}
 \mathcal{I}(0)=\int_{\mathcal{X}}\left[\sup_{\gamma(\cdot|x) \in \Pi(\mu_{1|3}, \mu_{2|3} ) } \int_{\mathcal{Y}_1\times\mathcal{Y}_2}f(y_1,y_2,x) \, d \gamma(y_1,y_2|x)\right]d\gamma_X(x),
 \end{align*}
 where for each fixed $x\in\mathcal{X}$, $\mu_{\ell|3}( \cdot| x)$ denote the conditional measure of $Y_\ell$ given $X=x$, and the inner optimization problem is a marginal problem with non-overlapping marginals. 
 
For any $\gamma \in \mathcal{P}(\mathcal{S} )$, let $\gamma_{13}$ and $\gamma_{23}$ denote the projections of $\gamma$ on $\mathcal{Y}_1 \times \mathcal{X}$ and $\mathcal{Y}_2 \times \mathcal{X}$, respectively. The W-DMR with overlapping marginals is defined as 
 \begin{equation} \label{eq:I-primal}
\mathcal{I}(\delta) := \sup_{\gamma \in \Sigma(\delta)} \int_\mathcal{S} f \, d \gamma,  \quad  \  \delta  \in \mathbb{R}_{+}^2,
 \end{equation}
where $\Sigma(\delta)$ is the uncertainty set given by
\[
\Sigma(\delta):=\Sigma(\mu_{13}, \mu_{23}, \delta) = \left\{   \gamma \in \mathcal{P}(\mathcal{S} ):  \ \boldsymbol{K}_1( \mu_{13}, \gamma_{13}) \leq \delta_1, \   \boldsymbol{K}_2( \mu_{23}, \gamma_{23}) \leq \delta_2    \right\}
\]
in which $\delta := (\delta_{1}, \delta_{2}) \in \mathbb{R}_{+}^2$ is the radius of the uncertainty set, and $\boldsymbol{K}_1$ and $\boldsymbol{K}_2$ are optimal transport costs associated with $c_1$ and $c_2$. We note that $\Sigma(\delta)$ is non-empty for all $\delta \in \mathbb{R}_{+}^2$.

\begin{remark}
(i) \Cref{assumption:cost-function,assumption:f-bounded-below} imply that $\mathcal{I}(\delta)>-\infty$ for all $\delta\ge 0$, see \Cref{lemma:I-concavity};
 (ii) When $\delta=0$, the uncertainty set $\Sigma(0)=\mathcal{F}(\mu_{13}, \mu_{23})$ and 
$ \mathcal{I}(0)= \mathcal{I}_{\mathrm{M} }(\mu_{13},\mu_{23})$.
\end{remark}

\subsection{Motivating Examples}

In this section, we present four distinct examples to demonstrate the wide applicability of the W-DMR in marginal problems. The first example is concerned with partial identification of treatment effect parameters when commonly used assumptions in the literature for point identification fail; the second example is concerned with distributionally robust optimal treatment choice; the third one is an application of W-DMR-MP in distributionally robust estimation under data combination; and the last one concerns measures of aggregate risk.

 For the first two examples, we adopt the potential outcomes framework for a binary treatment. Let $D \in \{0,1\}$ represent an individual's treatment status, and $Y_1\in \mathcal{Y}_1\subset \mathbb{R}$ and $Y_2\in \mathcal{Y}_2\subset \mathbb{R}$ denote the potential outcomes under treatments $D = 0$ and $D = 1$, respectively. 
Let the observed outcome be
\[
Y = D Y_2 + (1-D) Y_1.
\]

To focus on introducing the main ideas, we adopt the selection-on-observables framework stated in \Cref{assumption:selection-on-observables} below.
\begin{assumption} \label{assumption:selection-on-observables} \mbox{}
    \begin{enumerate}
        \item[(i)] \textbf{Conditional Independence}: The potential outcomes are independent of treatment assignment conditional on covariate $X\in\mathcal{X}\subset \mathbb{R}^q$ for $q \geq 1$, i.e.,
        \[
        (Y_1, Y_2) \ \indep \ D  \;|\; X;
        \]

        \item[(ii)] \textbf{Common Support}: For all $x\in \mathcal{X}$, $0<p(x)<1$, where $p(x):=\mathbb{P}(D=1|X=x)$.
    \end{enumerate}
\end{assumption}
Suppose a random sample on $ (Y,X,D)$ is available. Then under \Cref{assumption:selection-on-observables}, the marginal conditional distribution functions of $Y_1, Y_2$ given $X=x$ are point identified:
\[
F_{Y_1|X}(y|x)=\mathbb{P}(Y_1\leq y|X=x)=\mathbb{P}(Y\leq y|X=x,D=0)\] \mbox{and}
\[
F_{Y_2|X}(y|x)=\mathbb{P}(Y_2\leq y|X=x)=\mathbb{P}(Y\leq y|X=x,D=1).
\]
As a result, the probability measures $\mu_{13}$ of $(Y_{1},X)$ and $\mu_{23}$ of $(Y_{2},X)$ are identified as well. 

\subsubsection{Partial Identification of Treatment Effects}
\label{sec:partial-iden-TE}

\Cref{assumption:selection-on-observables} is commonly used to identify treatment effect parameters and optimal treatment choice. However the validity of \Cref{assumption:selection-on-observables} may be questionable when there are unobserved confounders. W-DMR-MP presents a viable approach to studying sensitivity of causal inference to deviations from \Cref{assumption:selection-on-observables} by varying the marginal measures of a joint measure of $(Y_1, Y_2, X)$ in Wasserstein uncertainty sets centered at reference measures consistent with \Cref{assumption:selection-on-observables}.  Specifically, let $f$ be a measurable function of $Y_1, Y_2$. Consider treatment effects of the form: $\theta_o:=\mathbb{E}_o[f(Y_1,Y_2)]$, where $\mathbb{E}_o$ denotes expectation with respect to the true measure. It includes the average treatment effect (ATE) for which $f(Y_1, Y_2)=Y_2 - Y_1$ and the distributional treatment effect such as $\mathbb{P}_o(Y_2 - Y_1\geq 0)$, where $\mathbb{P}_o$ denotes the probability computed under the true measure.

Consider the identified set for $\theta_o$ defined as 
\[
\Theta(\delta) := \left\{  \int_{\mathcal{S} } f(y_1,y_2) \, d \gamma(y_1, y_2, x) : \gamma \in  \Sigma( \delta) \right\},
\]
where
\begin{align*}
    \Sigma(\delta) = \left\{\gamma \in \mathcal{P}(\mathcal{S}) : \boldsymbol{K}_1(\mu_{13}, \gamma_{13}) \le \delta_1, \boldsymbol{K}_2(\mu_{23}, \gamma_{23}) \le \delta_2 \right\},
\end{align*}
in which $\mu_{13}$ and $\mu_{23}$ are the identified measures of $(Y_1, X)$ and $(Y_2, X)$ under \Cref{assumption:selection-on-observables}. Under mild conditions, we show in \Cref{prop:Identified-Sets-Interval} that the identified set $\Theta(\delta)$ is a closed interval given by 
\[
\Theta(\delta)= \left[ \min_{\gamma \in \Sigma(\delta)} \int_\mathcal{S} f(y_1,y_2) \, d \gamma(s), \max_{\gamma \in \Sigma(\delta)} \int_\mathcal{S} f(y_1,y_2) \, d \gamma(s) \right],
\]
where the lower and upper limits of the interval are characterized by the W-DMR-MP.\footnote{Since $\inf_{\gamma \in \Sigma(\delta)} \int_{\mathcal{S}} f(y_1, y_2) d \gamma(s)$ can be rewritten as $ - \sup_{\gamma \in \Sigma(\delta)} \int_{\mathcal{S}} [- f(y_1, y_2)] d \gamma(s)$, we also refer to the lower limit as W-DMR-MP. } When $\delta=0$, \citet{Fan2017} establish a characterization of $\Theta(0)$ via marginal problems with overlapping marginals. 

The identified set $\Theta(\delta)$ can be used to conduct sensitivity analysis to deviations from \Cref{assumption:selection-on-observables}. We note that sensitivity analysis to other commonly used assumptions such as the threshold-crossing model can be done by taking the reference measures as the measures identified under these alternative assumptions, see \citet{FAN_2009}.

\subsubsection{Robust Welfare Function}

In empirical welfare maximization (EWM), an optimal choice/policy is chosen to maximize the expected welfare estimated from a training data set and then applied to a target population, see \citet{Kitagawa_2018}. EWM assumes that the target population and the training data set come from the same underlying probability measure. This may not be valid in important applications. Motivated by designing externally valid treatment policy, \citet{adjaho2022} introduces a robust welfare function which allows the target population to differ from the training population. In this paper, we revisit \citet{adjaho2022}'s robust welfare function and propose a new one based on W-DMR with overlapping marginals. 

\citet{adjaho2022} adopts the following definition of a robust welfare function:
\begin{align*}
	\mathrm{RW}_0(d): =\inf_{\gamma \in \Sigma_0(\delta_0)} \mathbb{E}_{\gamma} [Y_1 (1 - d(X)) + Y_2 d(X)],
\end{align*}
where $d : \mathcal{X} \rightarrow \{0, 1\}$ is a measurable policy function, i.e., $d(X)$ is $0$ or $1$ depending on $X$ and $\Sigma_0(\delta_0)$ is the Wasserstein uncertainty set centered at a joint measure $\mu$ for $(Y_1, Y_2, X)$ consistent with \Cref{assumption:selection-on-observables}, i.e., 
\begin{align*}
	\Sigma_0(\delta_0) := \left\{ \gamma \in \mathcal{P}(\mathcal{S}): \,  \boldsymbol{K}_c(\mu, \gamma) \le \delta_0  \right\},
\end{align*}
where $\boldsymbol{K}_c(\mu, \gamma)$ is the optimal transport cost with cost function $c: \mathcal{S} \times \mathcal{S} \rightarrow \mathbb{R}_{+}\cup \{\infty\}$. 

Noting that \Cref{assumption:selection-on-observables} only identifies the marginal measures $\mu_{13},\mu_{23}$ of the reference measure $\mu$ in $\Sigma_0(\delta_0)$, we define a new robust welfare function as
\begin{equation*}
	\mathrm{RW}(d) := \inf_{\gamma \in \Sigma(\delta)} \mathbb{E}_{\gamma} [ Y_1 (1 - d(X)) + Y_2 d(X)],
\end{equation*}
where $\Sigma(\delta) = \Sigma(\mu_{13}, \mu_{23}, \delta)$ is the uncertainty set for W-DMR with overlapping marginals.

\subsubsection{W-DRO Under Data Combination} \label{sec:example-logit}

An important application of W-DMR is W-DRO. Let $f: \mathcal{Y}_1 \times \mathcal{Y}_2 \times \mathcal{X}\times \Theta \to \mathbb{R}$ be a loss function with an unknown parameter $\theta \in \Theta\subset \mathbb{R}^q$. W-DRO under data combination is defined as 
    	\begin{align}
     \min_{\theta \in \Theta}\sup_{\gamma \in \Sigma(\delta)} \int_{\mathcal{S}} f(y_1, y_2, x;\theta) \, d \gamma(y_1, y_2, x),
	\end{align}
 	where $\Sigma(\delta)$ is the uncertainty set for the overlapping case. For each $\theta \in \Theta$, the inner optimization is a W-DMR with overlapping marginals. In practice, we need to choose the reference measures $\mu_{13}$ and $\mu_{23}$ based on the sample information.
Focusing on logit model, where $\mathcal{Y}_1 = \{+1, -1\}$ is the space for the dependent variable, and $\mathcal{Y}_2$ and $\mathcal{X}$ are feature spaces/covariate space, and
 \begin{align*}
		f(y_1, y_2, x;\theta) = \log(1 + \exp(- y_1 \langle \theta, (y_2, x) \rangle)),
	\end{align*} 
 \citet{awasthi2022distributionally} proposes a method dubbed `Robust Data Join' in which the empirical measures constructed from the two data sets are used as reference measures. Specifically, let $\widehat{\mu}_{13}$ and $\widehat{\mu}_{23}$ denote empirical measures based on two separate data sets. The uncertainty set in \citet{awasthi2022distributionally} takes the following form:
 \[
\Sigma_{\mathrm{RDJ} } (\delta):= \left\{   \gamma \in \mathcal{P}(\mathcal{S} ):  \ \boldsymbol{K}_1( \widehat{\mu}_{13}, \gamma_{13}) \leq \delta_1, \   \boldsymbol{K}_2(\widehat{\mu}_{23}, \gamma_{23}) \leq \delta_2    \right\},
 \]
 where 
\[
c_1((y_1, x), (y_1', x'))  = \| x - x' \|_{p} + \kappa_1 | y_1 - y_1'| \quad \text{and}
\]
\[
c_2((y_2, x), (y_2,  x'))  = \| x - x' \|_{p} + \kappa_2 \| y_2 - y_2' \|_{p'}
\]
with $\kappa_1\geq 1$, $\kappa_2\geq 1$, $p\geq 1$, and $p'\geq 1$.
  
 Note that \citet{awasthi2022distributionally}'s `Robust Data Join' is different from our W-DMR with non-overlapping marginals because the measure of interest $\gamma \in \mathcal{P}(\mathcal{S} )$ has overlapping marginals. It is also different from our W-DMR with overlapping marginals because the reference measures $\widehat{\mu}_{13}$ and $\widehat{\mu}_{23}$ may not have overlapping marginals. Unlike the uncertainty set for W-DMR, $\Sigma_{RDJ}(\delta)$ is empty when $\delta=0$.

\subsubsection{Risk aggregation}
\label{sec:example-risk-aggregation}
	
Let $S_1, S_2$ be random variables representing  individual risks defined on Polish spaces $\mathcal{S}_1, \mathcal{S}_{2}$, respectively. Let $\mu_1, \mu_2$ be probability measures of $S_1,  S_2$. Let $\mathcal{V} = \mathcal{S}_1 \times \mathcal{S}_2$ and $g: \mathcal{V} \to \mathbb{R}$ be a risk aggregating function. Applying W-DMR with non-overlapping marginals to the risk aggregation function $g$, we can compute the worst aggregate risk when the joint measure of the individual risks varies in the uncertainty set $\Sigma_{\mathrm{D} }(\delta)$. This is different from the set-up in \citet{Eckstein2020}, where the following robust risk aggregation problem is studied:
	\begin{align*}
		\mathcal{I}_\Pi(\delta_0):=\sup_{\gamma \in \Sigma_{\Pi}(\delta)} \int_{\mathcal{V}} g \, d \gamma,
	\end{align*}
    where
	\begin{align*}
		\Sigma_{\Pi}(\delta_0) := \left\{\gamma \in \Pi(\mu_1, \mu_2) : \boldsymbol{K}_c(\gamma, \mu) \le \delta_0 \right\},
	\end{align*}
in which $\boldsymbol{K}_c$ is the optimal transport cost associated with a cost function $c: \mathcal{V} \times \mathcal{V} \to \mathbb{R}_{+}$. Since $\gamma \in \Sigma_{\Pi}(\delta_0)$ is a coupling of $(\mu_1, \mu_2)$, we have that $\Sigma_{\Pi}(\delta_0) \subset \Sigma_{\mathrm{D} }(0)$ and thus $\mathcal{I}_\Pi(\delta_0) \le \mathcal{I}_{\mathrm{D} }(0)$.

\section{Strong Duality and Distributionally Robust Makarov Bounds}

\label{sec:WDRO-main-theory}

In this section, we establish strong duality for our W-DMR-MP and apply it to develop Wasserstein distributionally robust Makarov bounds.

\subsection{Non-overlapping Marginals} 

\label{Strong-Duality-Non-overlapping Marginals}

For a measurable function $g: \mathcal{V} \rightarrow \mathbb{R}$ and $\lambda := (\lambda_1, \lambda_2) \in \mathbb{R}^2_{+}$, we define the function $g_\lambda: \mathcal{V}\rightarrow \mathbb{R} \cup \{\infty \}$ as 
\[\tag{2.1}
g_{\lambda}(v) := \sup_{v^\prime \in \mathcal{V}}  \varphi_\lambda(v, v^\prime), 
\]
where $\varphi_\lambda: \mathcal{V} \times \mathcal{V} \rightarrow \mathbb{R} \cup \{ -\infty\}$ is given by
\[
\varphi_\lambda(v, v^\prime) = g\left(s_1^\prime, s_2^\prime\right)  - \lambda_1 c_1 \left(s_1, s_1^\prime\right) -\lambda_2 c_2 \left(s_2, s_2^\prime\right),
\]
with $v:= (s_1, s_2)$ and $v^\prime := (s_1^\prime, s^\prime_2)$. Similarly, define $g_{\lambda_1, 1}: \mathcal{V} \rightarrow \mathbb{R} \cup \{+\infty\}$ and $g_{\lambda_2, 2}: \mathcal{V} \rightarrow \mathbb{R} \cup \{+\infty\}$  as
\begin{align*}
    g_{\lambda_1, 1}(s_1, s_2) & = 
    \sup_{s_1' \in \mathcal{S}_1} \{g(s_1', s_2) - \lambda_1 c_1(s_1, s_1')\} \quad \text{and } \\
    g_{\lambda_2, 2}(s_1, s_2) & = 
    \sup_{s_2' \in \mathcal{S}_2} \{g(s_1, s_2') - \lambda_2 c_2(s_2, s_2')\}. 
\end{align*}

The dual problem $\mathcal{J}_{\mathrm{D}}(\delta)$ corresponding to the primal problem $\mathcal{I}_{\mathrm{D}}(\delta)$ is defined as follows:
\begin{align} \label{eq:ID-dual}
    \mathcal{J}_{\mathrm{D}}(\delta) = 
    \begin{cases}
        \inf_{\lambda \in \mathbb{R}_{+}^2}  \left\{  \langle \lambda, \delta \rangle  +   \sup_{\varpi \in \Pi(\mu_1, \mu_2)} \int_{\mathcal{V} }g_\lambda  d\varpi  \right\} & \text{ if } \delta \in \mathbb{R}^2_{++}, \\
        \inf_{\lambda_1 \in \mathbb{R}_{+}}
        \left\{ \lambda_1 \delta_1 + \sup_{\varpi \in \Pi(\mu_1, \mu_2)} \int_{\mathcal{V}} g_{\lambda_1, 1} d \varpi    \right\} & \text{ if } \delta_1 > 0 \text{ and } \delta_2 = 0, \\
        \inf_{\lambda_2 \in \mathbb{R}_{+}}
        \left\{ \lambda_2 \delta_2 + \sup_{\varpi \in \Pi(\mu_1, \mu_2)} \int_{\mathcal{V}} g_{\lambda_2, 2} d \varpi    \right\} & \text{ if } \delta_1 = 0 \text{ and } \delta_2 > 0.
    \end{cases}
\end{align}

\begin{theorem}\label{thm:ID-duality}
	Suppose that \Cref{assumption:cost-function,assumption:g-bounded-below} hold. Then, $\mathcal{I}_{\mathrm{D}}(\delta) = \mathcal{J}_{\mathrm{D}}(\delta)$ for all $\delta \in \mathbb{R}_{+}^2 \setminus \{0\}$.
	
	\end{theorem}

Unlike the dual for W-DMR, the dual for W-DMR with non-overlapping marginals in \Cref{thm:ID-duality} involves a marginal problem with non-overlapping marginals $\mu_1, \mu_2$ due to the lack of knowledge on the dependence of the joint measure $\mu$. Computational algorithms developed for optimal transport can be used to solve the marginal problem, see \citet{Peyre2018}. For empirical measures $\mu_1, \mu_2$, the marginal problem is a discrete optimal transport problem and there are efficient algorithms to compute it, see \citet{Peyre2018}. For general measures $\mu_1, \mu_2$, strong duality may be employed in the numerical computation of the marginal problem. For instance, consider the case when $\delta > 0$. When $g_\lambda (v)$ is Borel measurable, several strong duality results are available, see e.g., \citet{villani2009optimal, villani2021topics}. For a general function $g$ and cost functions $c_1,c_2$, $g_\lambda (v)$ is not guaranteed to be Borel measurable. However, for Polish spaces, the set $\{v \in \mathcal{V}:g_{\lambda}(v) \ge u\}$ is an analytic set for all $u \in \overline{\mathbb{R}}$ (and $g_{\lambda}$ is universally measurable), since $g$, $c_1$ and $c_2$ are Borel measurable (see \citet[p. 580]{blanchet2019quantifying} and \citet[{Lemma 7.22, Lemma 7.30 (i) and Proposition 7.47}]{Bertsekas1978}). This allows us to apply strong duality for the marginal problem in \citet{Kellerer:1984to} restated in \Cref{thm:duality-OT} to the marginal problem involving $g_\lambda (v)$, see \cref{corollary:ID-dual-full} in \Cref{appendix:duality-OT}.

 Without additional assumptions on the function $g$ and the cost functions, the dual $\mathcal{J}_{\mathrm{D}}(\delta)$ in \Cref{thm:ID-duality} for interior points $\delta \in \mathbb{R}_{++}^2$ and the dual for boundary points may not be the same. To illustrate, plugging in $\delta_2 = 0$ in the dual form for interior points in \Cref{thm:ID-duality}, we obtain 
    \begin{align*}
          \inf_{\lambda_1 \in \mathbb{R}_{+}}   \left[  \lambda_1 \delta_1 + \inf_{\lambda_2 \in \mathbb{R}_{+}}\sup_{\varpi \in \Pi(\mu_1, \mu_2)} \int_{\mathcal{V} }g_\lambda \, d\varpi \right]. 
    \end{align*}
   It is different from the dual $\mathcal{J}_{\mathrm{D}}(\delta_1,0)$ for $\delta_1>0$, since
    \begin{align*}
          \inf_{\lambda_2 \in \mathbb{R}_{+}} \sup_{\varpi \in \Pi(\mu_1, \mu_2)} \int_{\mathcal{V} }g_\lambda \, d\varpi \;
        \neq
        \sup_{\varpi \in \Pi(\mu_1, \mu_2)} \int_{\mathcal{V}} g_{\lambda_1, 1} \, d \varpi. 
    \end{align*}
When the function $g$ and the cost functions satisfy assumptions in \Cref{thm:ID-continuity}, the dual $\mathcal{J}_{\mathrm{D}}(\delta)$ in \Cref{thm:ID-duality} for interior points $\delta \in \mathbb{R}_{++}^2$ and the dual for boundary points are the same so that
 \begin{align*}
    \mathcal{I}_{\mathrm{D} }(\delta) = \inf_{\lambda \in \mathbb{R}_{+}^2}  \left[  \langle \lambda, \delta \rangle  +   \sup_{\varpi \in \Pi(\mu_1, \mu_2)} \int_{\mathcal{V} }g_\lambda \, d\varpi   \right]
    \end{align*}
for all $\delta \in \mathbb{R}_{+}^2$.

 \begin{remark}
   For Polish spaces, \Cref{thm:ID-duality} generalizes the strong duality in \citet{zhang2022simple} restated in Theorem 2.1. Our proof is based on that in \citet{zhang2022simple}. However, due to the presence of two marginal measures in the uncertainty set $\Sigma_D(\delta)$, we need to verify the existence of a joint measure when some of its overlapping marginal measures are fixed, and we rely on existing results for a given consistent product marginal system studied in \citet{vorob1962consistent}, \citet{kellerer1964verteilungsfunktionen}, and \citet{shortt1983combinatorial},
   see \Cref{appendix:CPMS} for a detailed review.
 \end{remark}
	\begin{remark}
	 Similar to \citet{Sinha2017} for W-DMR in marginal problems, we can define an alternative W-DMR through linear penalty terms, i.e., 
        \begin{align*}
			\sup_{\gamma \in \mathcal{P}(\mathcal{V})} \left\{\int_{\mathcal{V}} g d \gamma - \lambda_1 \boldsymbol{K}_1(\mu_1, \gamma_1) - \lambda_2 \boldsymbol{K}_2(\mu_2, \gamma_2): \boldsymbol{K}_\ell(\mu_\ell, \gamma_\ell) < \infty \text{ for } \ell = 1, 2\right\}
		\end{align*}
        with $\lambda_1, \lambda_2 \in \mathbb{R}_{++}$. The proof of \Cref{thm:ID-duality} implies that the dual form of this problem is $\sup_{\varpi \in \Pi(\mu_1, \mu_2)} \int g_\lambda d \varpi$ under the condition in \Cref{thm:ID-duality}. 
	\end{remark}

\subsection{Overlapping Marginals}
\label{Strong-Duality-Overlapping Marginals}
 Let $\phi_\lambda : \mathcal{V} \times \mathcal{S}\rightarrow \mathbb{R} \cup \{- \infty \}$ be
\[
\begin{aligned}
\phi_\lambda(v, s^{\prime}) & := f(s^{\prime})-\lambda_1 c_1 (s_1, s_1^{\prime})-\lambda_2 c_2(s_2, s_2^{\prime}),
\end{aligned}
\]
where $v = (s_1, s_2)$, $s^\prime = (y^\prime_0,y^\prime_1,x^\prime)$, $s^\prime_\ell = (y^\prime_\ell, x^\prime)$ and $s_\ell=(y_\ell, x_\ell)$. 
Define the function $f_\lambda : \mathcal{V}  \rightarrow \overline{\mathbb{R}}$ associated with $f$ as 
\[
f_\lambda(v) :=  \sup_{s^\prime \in \mathcal{S} }  \phi_\lambda (v, s^{\prime}).
\]
Similarly, we define $f_{\lambda, 1}: \mathcal{V} \to \overline{\mathbb{R}}$ and $f_{\lambda, 2}: \mathcal{V} \to \overline{\mathbb{R}}$ as follows:
 \begin{align*}
        f_{\lambda_1, 1}(s_1, s_2) & = 
        \sup_{y_1' \in \mathcal{Y}_1} \{f(y_1', y_2, x_2) - \lambda_1 c_1((y_1, x_1), (y_1', x_2))\} \text{ and } \\
        f_{\lambda_2, 2}(s_1, s_2) & = 
        \sup_{y_2' \in \mathcal{Y}_2} \{f(y_1, y_2', x_1) - \lambda_2 c_2((y_2, x_2),  (y_2', x_1)\},
    \end{align*}
    in which $s_1 = (y_1, x_1)$ and $s_2 = (y_2, x_2)$.
The dual problem $\mathcal{J}(\delta)$ corresponding to the primal problem $\mathcal{I}(\delta)$ is defined as follows:
 \begin{align} \label{eq:I-dual} 
 \mathcal{J}(\delta) = 
 \begin{cases} \inf_{\lambda \in \mathbb{R}_{+}^2}  \left\{  \langle \lambda, \delta \rangle  +   \sup_{\varpi \in  \Pi(\mu_{13}, \mu_{23})} \int_{\mathcal{V} }f_\lambda  d \varpi    \right\} & \text{ if } \delta \in \mathbb{R}^2_{++}, \\
        \inf_{\lambda_1 \in \mathbb{R}_{+}}
        \left\{ \lambda_1 \delta_1 + \sup_{\varpi \in \Pi(\mu_{13}, \mu_{23})} \int_{\mathcal{V}} f_{\lambda_1, 1} d \varpi    \right\} & \text{ if } \delta_1 > 0 \text{ and } \delta_2 = 0, \\
        \inf_{\lambda_2 \in \mathbb{R}_{+}}
        \left\{ \lambda_2 \delta_2 + \sup_{\varpi \in \Pi(\mu_{13}, \mu_{23})}  \int_{\mathcal{V}} f_{\lambda_2, 2} d \varpi   \right\} & \text{ if } \delta_1 = 0 \text{ and } \delta_2 > 0. 
    \end{cases}
    \end{align}

\begin{theorem} \label{thm:I-duality}
Suppose that \Cref{assumption:cost-function,assumption:f-bounded-below} hold.
Then, $\mathcal{I}(\delta) = \mathcal{J}(\delta)$ for all $\delta \in \mathbb{R}_{+}^2 \setminus \{0\}$.
\end{theorem}

An interesting feature of the dual for overlapping marginals is that it involves marginal problems with non-overlapping marginals, i.e., $\sup_{\varpi \in  \Pi(\mu_{13}, \mu_{23})} \int_{\mathcal{V} }f_\lambda (v) d \varpi (v)$, although the uncertainty set in the primal problem involves overlapping marginals. Compared with the non-overlapping marginals case, overlapping marginals in the uncertainty set make the relevant consistent product marginal system in the verification of the existence of a joint measure more complicated, see the proof of \Cref{lemma:AO1}. Nonetheless, the non-overlapping marginals in the dual allow us to apply \Cref{thm:duality-OT} to the marginal problem involving $f_\lambda$, $f_{\lambda, 1}$ and $f_{\lambda, 2}$, see \cref{corollary:I-dual-full} in \Cref{appendix:duality-OT}. 

Under the assumptions in \Cref{thm:I-continuity}, we have
   \begin{align*}
    \mathcal{I}(\delta) = \inf_{\lambda \in \mathbb{R}_{+}^2}  \left[  \langle \lambda, \delta \rangle  +   \sup_{\varpi \in \Pi(\mu_{13}, \mu_{23})} \int_{\mathcal{V} }f_\lambda \, d\varpi   \right]
    \end{align*}
    for all $\delta \in \mathbb{R}_{+}^2$.

\begin{remark}
Similar to the non-overlapping case, we can define an alternative W-DMR with overlapping marginals through linear penalty terms, i.e.,  
    \begin{align*}
        \sup_{\gamma \in \mathcal{P}(\mathcal{S})} \left\{\int_{\mathcal{S}} g \, d \gamma - \lambda_1 \boldsymbol{K}_1(\mu_{13}, \gamma_{13}) - \lambda_2 \boldsymbol{K}_2(\mu_{23}, \gamma_{23}): \boldsymbol{K}_\ell(\mu_{\ell3}, \gamma_{\ell3}) < \infty \text{ for } \ell = 1, 2\right\},
    \end{align*}
    with $\lambda_1, \lambda_2 \in \mathbb{R}_{++}$. The proof of \Cref{thm:I-duality} implies that the dual form of this problem is $\sup_{\varpi \in \Pi(\mu_{13}, \mu_{23})} \int_{\mathcal{V}} f_\lambda \, d \varpi$ under the conditions in \Cref{thm:I-duality}. 
\end{remark}

 \subsection{Wasserstein Distributionally Robust Makarov Bounds}

Let $\mathcal{S}_1=\mathbb{R}$, $\mathcal{S}_2=\mathbb{R}$, $\mu_1\in \mathcal{P}(\mathcal{S}_1)$, and $\mu_2\in \mathcal{P}(\mathcal{S}_2)$. Further, let $Z=S_1+S_2$, where $S_1,S_2$ are random variables whose probability measures are $\mu_1, 
\mu_2$ respectively. For a given $z\in \mathbb{R}$, let $F_{Z}(z)=\mathbb{E}_o[g(S_1, S_2)]$, where $g(s_1, s_2)= \mathds{1}\left\{s_1 + s_2 \le z\right\}$. 

Sharp bounds on the quantile function $F^{-1}_{Z}\left( \cdot\right) $ are established in \citet{Makarov1982}) and referred to as the Makarov bounds. Inverting the Makarov bounds lead to sharp bounds on the distribution function $F _{Z}\left( z\right)$, see \citet{Ruschendorf_1982_RV_Maximum} and \citet{Frank_1987_best_bounds}. They are given by 
\begin{align*}
\inf_{\gamma\in \Pi(\mu_1,\mu_2)} \mathbb{E}_{\gamma}[g(S_1, S_2)] & = \sup_{x\in \mathbb{R}}\max \left\{\mu_1(x)+\mu_2(z-x)-1,0 \right\} \ \text{ and } \\
\sup_{\gamma\in \Pi(\mu_1,\mu_2)} \mathbb{E}_{\gamma}[g(S_1, S_2)] & = 1+\inf_{x\in \mathbb{R} }\min \left\{\mu_1(x)+\mu_2(z-x)-1,0 \right\}. 
\end{align*}%
Since the quantile bounds first established in \citet{Makarov1982}) and the above distribution bounds are equivalent, we also refer to the latter as Makarov bounds. Makarov bounds have been successfully applied in distinct areas. For example, the upper bound
on the quantile of $Z$ is known as the worst VaR of $Z$, see  \citet{Embrechts2003UsingCopulaeBound}, \citet{Embrechts_2005_WorstVaR}; Makarov bounds are also used to study partial identification of distributional treatment effects when the treatment assignment mechanism identifies the marginal measures of the potential outcomes such as in \Cref{assumption:selection-on-observables}, see \citet{Fan_2009_PartialIdentification,Fan_2010_SharpBounds,Fan_2012_CI_quantiles}, \citet{FAN_2009}, \citet{Fan2017}, \citet{Ridder_2007}, and \citet{Firpo_2019}. 

Applying \Cref{thm:I-duality}, we extend Makarov bounds to allow for possible misspecification of the marginal measures and call the resulting bounds Wasserstein distributionally robust Makarov bounds.

\begin{corollary}[Wasserstein distributionally robust Makarov bounds] \label{corollary:ID-duality-marakov}
Suppose that $g(s_1, s_2) = \mathds{1}(s_1 + s_2 \le z)$ and $c_{\ell}(s_\ell, s_\ell') = |s_\ell - s_\ell'|^2$ for $\ell = 1, 2$. 
For all $\delta\in\mathbb{R}_{+}^2$,
\begin{align*}
&\sup_{\gamma\in \Sigma_{\mathrm{D}}(\delta)} \mathbb{E}_{\gamma}[g(S_1, S_2)]\\
= &\inf_{\lambda\in\mathbb{R}_{+}^2} \Bigg(\langle\lambda, \delta\rangle +\sup_{ \varpi\in \Pi(\mu_1,\mu_2) }
\Bigg[\int_{\{s_1 + s_2 > z\}}   \left[1 - \frac{ \lambda_1\lambda_2 (s_1 + s_2 - z)^2}{\lambda_1 + \lambda_2} \right]^+ d\varpi(s_1, s_2)\\
& \quad  + \mathbb{E}_{\varpi} \Big[\mathds{1}\left\{S_1 + S_2 \le z\right\} \Big]\Bigg];
\end{align*}
\begin{align*}
&\inf _{\gamma \in \Sigma_{\mathrm{D}}(\delta)} \mathbb{E}_\gamma\left[g\left(S_1, S_2\right)\right] \\
= & \sup_{\lambda\in\mathbb{R}_{+}^2} \Bigg[- \langle\lambda, \delta\rangle +\inf_{ \varpi\in \Pi(\mu_1,\mu_2) }
\Bigg\{- \int_{\{s_1 + s_2 \leq z\} }   \left[1 - \frac{ \lambda_1\lambda_2 (s_1 + s_2 - z)^2}{\lambda_1 + \lambda_2} \right]^+ d\varpi(s_1, s_2)\\
& \quad  \qquad \quad    +  \mathbb{E}_{\varpi} \Big[\mathds{1}\left\{S_1 + S_2 \le z\right\} \Big] \Bigg].
\end{align*}
\end{corollary}
We note that $g_{\lambda}(v)$ is bounded and continuous in $v$, and convex in $\lambda$, and $\Pi(\mu_1, \mu_2)$ is compact.  Applying \citet[Theorem 2]{Fan_1953_MinimaxThms}'s minimax theorem, we can interchange the order of $\inf$ and $\sup$ in the dual in the above corollary and get
\begin{align*}
&\sup_{\gamma\in \Sigma_{\mathrm{D}}(\delta)} \mathbb{E}_{\gamma}[g(S_1, S_2)]\\
 = &\sup_{ \varpi\in \Pi(\mu_1,\mu_2) }\Bigg[ 
  \inf_{\lambda\in\mathbb{R}_{+}^2} \left(\langle\lambda, \delta\rangle  +\int_{\{s_1 + s_2 > z\}}   \left[1 - \frac{ \lambda_1\lambda_2 (s_1 + s_2 - z)^2}{\lambda_1 + \lambda_2} \right]^+ d\varpi(s_1, s_2)\right) \\
& \quad   + \mathbb{E}_{\varpi} \Big[\mathds{1}\left\{S_1 + S_2 \le z\right\} \Big]\Bigg].
\end{align*} 
This expression is very insightful, where the inner infimum term characterizes possible deviations of the true marginal measures from the reference measures.

\section{Finiteness of the W-DMR-MP and Existence of Optimizers} \label{sec:additional-properties}

In this section, we assume that all the reference measures belong to appropriate Wasserstein spaces and prove finitness of the W-DMR-MP and existence of an optimizer. 

\begin{definition}[Wasserstein space]
The Wasserstein space of order $p \geq 1$ on a Polish space $\mathcal{X}$ with metric $ \boldsymbol{d}$ is defined as
\[
\mathcal{P}_p( \mathcal{X} ) = \left\{  \mu \in \mathcal{P}( \mathcal{X} ): \int_{\mathcal{X} } \boldsymbol{d}(x_0, x)^p d \mu(x)   < \infty \right\},
\]
where $x_0 \in \mathcal{X}$ is arbitrary.
\end{definition}

\begin{assumption} \label{assumption:finite-moments} \leavevmode 
	\begin{assumpenum}
		\item \label{assumption:finite-moments-ID} In the non-overlapping case, we assume that $\mu_{1} \in \mathcal{P}_{p_1}(\mathcal{S}_1)$ and $\mu_{2} \in \mathcal{P}_{p_2}(\mathcal{S}_2)$ for some $p_1 \ge 1$ and $p_2 \ge 1$;
		\item \label{assumption:finite-moments-I} In the overlapping case, we assume that  $\mu_{13} \in \mathcal{P}_{p_1}(\mathcal{S}_1)$ and $\mu_{23} \in \mathcal{P}_{p_2}(\mathcal{S}_2)$ for some $p_1 \ge 1$ and $p_2 \ge 1$. 
	\end{assumpenum}
\end{assumption}

\begin{assumption} \label{assumption:metric-cost}
	The cost function $c_\ell: \mathcal{S}_\ell \times \mathcal{S}_\ell \to \mathbb{R} \cup \{\infty\}$ is of the form $c_\ell(s_\ell, s_\ell^\prime ) = \boldsymbol{d}_{\mathcal{S}_\ell }  (s_\ell, s_\ell^\prime )^{p_\ell}$, where $(\mathcal{S}_\ell, \boldsymbol{d}_{\mathcal{S}_\ell }) $ is a Polish space and $p_\ell \geq 1$ for $\ell=1, 2$.
\end{assumption}

\subsection{Finiteness of the W-DMR-MP}

For non-overlapping case, we establish the following result.
\begin{theorem} \label{thm:ID-finite}
	Suppose that \Cref{assumption:g-bounded-below,assumption:finite-moments-ID,assumption:metric-cost} hold. Then for all $\delta\in \mathbb{R}_{++}^2$,  $\mathcal{I}_{\mathrm{D} }(\delta) < \infty$ if and only if there exist $v^\star := (s_1^\star, s_2^\star) \in \mathcal{V}$ and a constant $M > 0$ such that for all $(s_1, s_2)\in \mathcal{V}$, 
    \begin{equation} \label{eq:GC-I} 
		g(s_1, s_2) \leq M \left[   1  + \boldsymbol{d}_{\mathcal{S}_1} (s_1^\star,s_1 )^{p_1}+ \boldsymbol{d}_{ \mathcal{S}_2 }(s_2^\star,s_2)^{p_2} \right], 
	\end{equation}
	where $p_1$ and $p_2$ are defined in \Cref{assumption:finite-moments-ID}.
\end{theorem}

The inequality in \Cref{eq:GC-I} is a growth condition on the function $g$. It extends the growth condition in \citet{yue2022linear} for W-DMR to our W-DMR with non-overlapping narginals. 

For the overlapping case, the following result holds.

\begin{theorem} \label{thm:I-finite} 
	Suppose that \Cref{assumption:f-bounded-below,assumption:finite-moments-I,assumption:metric-cost} hold. Then for all $\delta\in \mathbb{R}_{++}^2$,  $\mathcal{I}(\delta) < \infty$ if and only if there exist $ (s_1^\star, s_2^\star) \in \mathcal{S}_1 \times \mathcal{S}_2$ and a constant $M > 0$ such that 
	\begin{equation} \label{eq:GC-II} 
		f(s) \leq M \left[   1  +\boldsymbol{d}_{\mathcal{S}_1} (s_1^\star, s_1)^{p_1}+\boldsymbol{d}_{ \mathcal{S}_2 }(s_2^\star, s_2)^{p_2}  \right],
	\end{equation}
	for all $ s \in \mathcal{S}$, where $s := (y_1, y_2, x), s_\ell := (y_\ell, x)$ and $s_\ell^\star := (y_\ell^\star, x^\star)$ for $\ell = 1, 2$, and $p_1$ and $p_2$ are defined in \Cref{assumption:finite-moments-I}.
\end{theorem}

The growth condition \eqref{eq:GC-II} on the function $f$ extends the growth condition in \citet{yue2022linear} for W-DMR. When
\[
\boldsymbol{d}_{\mathcal{S}_\ell} (  (y_\ell, x), (y^{\prime}_\ell, x^\prime) ) =  \boldsymbol{d}_{\mathcal{Y}_\ell} (  y_\ell,y^{\prime}_\ell )  +   \boldsymbol{d}_{\mathcal{X}} (  x, x^{\prime} ),
\]
condition \eqref{eq:GC-II}  is satisfied if and only if there exist $s^\star := (y_1^\star, y_2^\star, x^\star)$ and a constant $M > 0$ such that 
\[
f(s) \leq M \left[   1  +  \boldsymbol{d}_{\mathcal{Y}_1} (y_1, y_1^\star )^{p_1}+ 
  \boldsymbol{d}_{\mathcal{Y}_2} (y_2, y_2^\star )^{p_2}+
  \boldsymbol{d}_{ \mathcal{X} }(x, x^\star )^{p_1 \wedge p_2}  \right],
\]
for all $s = (y_1,y_2,x) \in \mathcal{S}$.

\begin{remark}
    The conditions in \Cref{thm:ID-finite,thm:I-finite} are sufficient conditions for $\mathcal{I}_{\mathrm{D}}(\delta)$ and $\mathcal{I}(\delta)$ to be finite for all $\delta \in \mathbb{R}_{+}^2$ including boundary points because $\mathcal{I}_{\mathrm{D}}(\delta)$ and $\mathcal{I}(\delta)$ are non-decreasing.
\end{remark}

\subsection{Existence of Optimizers}

\begin{definition}\label{def:proper}
	A metric space $(\mathcal{X}, \boldsymbol{d})$ is said to be proper if for any $r>0$ and $x_0 \in \mathcal{X}$, the closed ball $\overline{B}(x_0, r) := \{ x \in \mathcal{X}: \boldsymbol{d}(x,x_0) \leq r \}$ is compact.
\end{definition}
	
Examples of proper metric spaces include finite dimensional Banach spaces and complete  Riemannian manifolds, see \citet{yue2022linear}.

\begin{assumption} \label{assumption:proper} $(\mathcal{S}_1, \boldsymbol{d}_{\mathcal{S}_1})$ and $(\mathcal{S}_2, \boldsymbol{d}_{\mathcal{S}_2})$ are proper. 
\end{assumption}

\Cref{assumption:finite-moments,assumption:metric-cost,assumption:proper} imply that $\Sigma_{\mathrm{D}}(\delta)$ and $\Sigma(\delta)$ are weakly compact, see \Cref{prop:non-overlapping-compactness,prop:overlapping-compactness} in Appendix C. Given weak compactness of the uncertainty sets $\Sigma_{\mathrm{D}}(\delta)$ and $\Sigma(\delta)$, it is sufficient to show that the mapping: $\gamma \to \int g d\gamma$ is upper semi-continuous over $\gamma \in \Sigma_{\mathrm{D} }(\delta)$ for the non-overlapping case, and the mapping: $\gamma \to \int f d\gamma$ is upper semi-continuous over $\gamma \in \Sigma(\delta)$ for the overlapping case. In \Cref{thm:ID-existence,thm:I-existence} below, we provide conditions for $g$ and $f$  ensuring upper semi-continuity of each map and thus the existence of optimal solutions for $\mathcal{I}_{\mathrm{D}}(\delta)$ and $\mathcal{I}(\delta)$. 

\begin{theorem}\label{thm:ID-existence}
	Suppose that \Cref{assumption:g-bounded-below,assumption:finite-moments-ID,assumption:metric-cost,assumption:proper} hold.
	Further, assume that $g$ is upper-semicontinuous, and there exist a constant $M>0$,  $v^\star := (s_1^\star, s_2^\star) \in \mathcal{V}$ and $p_\ell^\prime \in (0, p_\ell)$ for $\ell=1,2$, such that 
 \begin{align} \label{eq:GC-existence-i}
     g(v) \leq  M\left[1+  \boldsymbol{d}_{\mathcal{S}_1} (s^\star_1, s_1)^{p_1^{\prime}} +  \boldsymbol{d}_{\mathcal{S}_2} (s^\star_2,s_2)^{p_2^{\prime}}  \right],
 \end{align}
	for all $v := (s_1, s_2) \in \mathcal{V}$.  Then an optimal solution of \eqref{eq:ID-primal} exists for all $\delta \in \mathbb{R}_{+}^2$.
\end{theorem}

\begin{theorem} \label{thm:I-existence}
    Suppose that \Cref{assumption:f-bounded-below,assumption:finite-moments-I,assumption:metric-cost,assumption:proper} hold.
	Further, assume that $f$ is upper-semicontinuous, and there exist $ (s_1^\star, s_2^\star) \in \mathcal{S}_1 \times \mathcal{S}_2$, a constant $M>0$, $p_\ell^{\prime} \in (0, p_\ell)$ for $\ell = 1, 2$, such that 
    \begin{align}  \label{eq:GC-existence-ii}
        f(s) \leq M \left[   1  +\boldsymbol{d}_{\mathcal{S}_1} (s_1^\star, s_1)^{p_1^{\prime}}+\boldsymbol{d}_{ \mathcal{S}_2 }(s_2^\star,s_2 )^{p_2^{\prime}}  \right],
    \end{align}
	for all $ s \in \mathcal{S}$ where $s := (y_1, y_2, x), s_\ell := (y_\ell, x)$ and $s_\ell^\star := (y_\ell^\star, x_\ell^\star)$ for $\ell = 1, 2$. Then an optimal solution of \eqref{eq:I-primal} exists for all $\delta \in \mathbb{R}_{+}^2$. 
\end{theorem}

\subsection{Characterization of Identified Sets}
\label{sec:identified-sets-interval}

In some applications, such as the partial identification of treatment effects introduced in \Cref{sec:partial-iden-TE}, the identified sets of $\theta_{\mathrm{D}o} := \mathbb{E}_{o}[g(S_1, S_2)]$ and $\theta_o := \mathbb{E}_{o}[f(S)]$ are of interest, where $S$ is a random variable whose probability measure belongs to $\Sigma(\delta)$, and $S_1$ and $S_2$ are random variables whose joint probability measure belongs to $\Sigma_{D}(\delta)$. They are:
\begin{align*}
    \Theta_{\mathrm{D}}(\delta) & := 
    \left\{ \int_{\mathcal{S}_1 \times \mathcal{S}_2} g\,  d \gamma : \gamma \in \Sigma_{\mathrm{D}}(\delta) \right\} \quad \text{and} \quad
    \Theta(\delta)  := 
    \left\{ \int_{\mathcal{S}} f \, d \gamma : \gamma \in \Sigma(\delta) \right\}.
\end{align*}

By applying finiteness and existence results, we show below that under mild conditions, the identified sets  $\Theta_{\mathrm{D}}(\delta)$ and $\Theta(\delta)$ are both closed intervals.

\begin{proposition}\label{prop:Identified-Sets-Interval}
\leavevmode
\begin{enumerate}[label=(\roman*)]
    \item Suppose \Cref{assumption:finite-moments-ID,assumption:metric-cost,assumption:proper} hold. In addition, $g$ is continuous, and $|g|$ satisfies Condition \eqref{eq:GC-existence-i}. 
    Then, for $\delta \in \mathbb{R}_{+}^2$, we have \[
\Theta_{\mathrm{D}}(\delta)=\left[\min _{\gamma \in \Sigma_{\mathrm{D}}(\delta)} \int_{\mathcal{S}_1 \times \mathcal{S}_2} g \, d \gamma, \max _{\gamma \in \Sigma_{\mathrm{D}}(\delta)} \int_{\mathcal{S}_1 \times \mathcal{S}_2} g \, d \gamma \right],
\]
where both the lower and upper bounds are finite.

   \item  Suppose \Cref{assumption:finite-moments-I,assumption:metric-cost,assumption:proper} hold. In addition, $f$ is continuous and $|f|$ satisfies Condition \eqref{eq:GC-existence-ii}. Then for $\delta \in \mathbb{R}_{+}^2$, we have \[
\Theta(\delta)=\left[\min _{\gamma \in \Sigma(\delta)} \int_{\mathcal{S}} f \, d \gamma, \max_{\gamma \in \Sigma(\delta)} \int_{\mathcal{S}} f \, d \gamma \right],
\]
where both the lower and upper bounds are finite.
\end{enumerate}

\end{proposition}
The strong duality in Section 3 can be used to evaluate the lower and upper bounds.

\section{Continuity of the DMR-MP Functions} \label{sec:continuity}

In this section, we establish continuity of the W-DMR-MP functions $\mathcal{I}_{\mathrm{D} }(\delta)$ and $\mathcal{I}(\delta)$ for all $\delta\in\mathbb{R}_+^2$ under similar conditions to those in \citet{zhang2022simple}. Compared with \citet{zhang2022simple}, our analysis is more involved, because the boundary in our case includes not only the origin $(0,0)$ but also $(\delta_1,0)$ and $(0,\delta_2)$ for all $\delta_1 > 0$ and $\delta_2 > 0$. 

\subsection{Non-overlapping Marginals } 
\label{sec:continuity-nonoverlapping}

\Cref{lemma:ID-concavity} implies that under \Cref{assumption:cost-function,assumption:g-bounded-below}, $\mathcal{I}_{\mathrm{D} }(\delta)$ is a concave function for $\delta \in \mathbb{R}_{+}^2$ and hence is continuous on $\mathbb{R}_{++}^2$. We provide the main assumption for the continuity of $\mathcal{I}_{\mathrm{D}}(\delta)$ on $\mathbb{R}_{+}^2$ in this subsection.

\begin{assumption}\label{assumption:g-Psi-growth}
Let $\Psi: \mathbb{R}_+^2 \rightarrow \mathbb{R}_+$ be a continuous, non-decreasing, and concave function with $\Psi(0,0) =0$. Suppose the function $g: \mathcal{V} \rightarrow \mathbb{R}$ satisfies
\begin{align}
	g(v) - g(v^\prime) \leq \Psi \left(  c_1(s_1, s_1^\prime),  c_2(s_2, s_2^\prime) \right), \label{eq:g-Psi-growth}
\end{align} 
for all $v=(s_1, s_2) \in \mathcal{V}$ and  $v^\prime= (s^\prime_1, s^\prime_2) \in \mathcal{V}$. 
\end{assumption}

The function $\Psi$ in \Cref{assumption:g-Psi-growth} plays the role of the modulus of continuity of $g$. To illustrate, consider the following example. 

\begin{example}
Suppose \cref{assumption:metric-cost} holds, i.e., $c_\ell(s_\ell, s_\ell^\prime) = \boldsymbol{d}_{\mathcal{S}_\ell}(s_\ell, s_\ell^\prime)^{p_\ell}$ for some $p_\ell \geq 1$, $\ell=1,2.$
\begin{enumerate}[label=(\roman*)]
    \item Define a product metric $\boldsymbol{d}_{\mathcal{V}}$ on $\mathcal{V} = \mathcal{S}_1 \times \mathcal{S}_2$ as 
\[
\boldsymbol{d}_{\mathcal{V}} ((s_1,s_2), (s_1^\prime, s_2^\prime)) =    \boldsymbol{d}_{\mathcal{S}_1} (s_1,s_1^\prime)  + \boldsymbol{d}_{\mathcal{S}_2} (s_2, s_2^\prime).
\]
Let $\Psi (x,y) = x^{1/p_1} + y^{1/p_2}$. Then, $\boldsymbol{d}_{\mathcal{V}} ((s_1,s_2), (s_1^\prime, s_2^\prime)) =  \Psi \left(   c_1 (s_{1}, s_{1}^{\prime}) , c_2(s_{2}, s_{2}^{\prime})   \right)$. On the metric space $(\mathcal{V}, \boldsymbol{d}_{\mathcal{V} })$, the function $g$ is continuous and has $\omega: x\mapsto x$ as modulus of continuity. Moreover, \Cref{assumption:g-Psi-growth} implies the growth condition in \eqref{thm:ID-existence}.

\item Suppose $p_1=p_2$. Define a product metric $\boldsymbol{d}_{\mathcal{V}}$ on $\mathcal{V} = \mathcal{S}_1 \times \mathcal{S}_2$ as 
\[
\boldsymbol{d}_{\mathcal{V}} ((s_1,s_2), (s_1^\prime, s_2^\prime)) =  \left[   \boldsymbol{d}_{\mathcal{S}_1} (s_1,s_1^\prime)^p  + \boldsymbol{d}_{\mathcal{S}_2} (s_2, s_2^\prime)^p \right]^{1/p}.
\]
Let $\Psi (x,y) = (x+y)^{1/p}$.  Then, $\boldsymbol{d}_{\mathcal{V}} (  (s_1,s_2), (s_1^\prime, s_2^\prime)  ) =  \Psi \left(   c_1 (s_{1}, s_{1}^{\prime}) , c_2(s_{2}, s_{2}^{\prime})   \right)$. On the metric space $(\mathcal{V}, \boldsymbol{d}_{\mathcal{V} })$, the function $g$ is continuous and has $\omega: x\mapsto x$ as modulus of continuity.  \Cref{assumption:g-Psi-growth} also implies the growth condition in \eqref{thm:ID-existence}.

\item Suppose $p_1 \neq p_2$.  Define a product metric $\boldsymbol{d}_{\mathcal{V}}$ on $\mathcal{V} = \mathcal{S}_1 \times \mathcal{S}_2$ as 
\[
\boldsymbol{d}_{\mathcal{V}} ((s_1,s_2), (s_1^\prime, s_2^\prime))  =     \boldsymbol{d}_{\mathcal{S}_1} (s_1,s_1^\prime)  \vee \boldsymbol{d}_{\mathcal{S}_2} (s_2, s_2^\prime).
\]
Then, \Cref{assumption:g-Psi-growth}  implies 
\[
g(v)-g (v^{\prime}) \leq   \Psi\left(  \boldsymbol{d}_{\mathcal{V}} (  v, v^\prime) , \boldsymbol{d}_{\mathcal{V}} (v, v^\prime)   \right)  =  \omega(  \boldsymbol{d}_{\mathcal{V}} (v, v^\prime)    ) .
\]
where $\omega: x \mapsto \Psi(x,x)$ is a concave function. On the metric space $(\mathcal{V},\boldsymbol{d}_{\mathcal{V}})$, the function $g$ is continuous and has $\omega: x \mapsto \Psi(x,x)$ as modulus of continuity.
\end{enumerate}
\end{example}

\begin{theorem} \label{thm:ID-continuity}
    Suppose \Cref{assumption:cost-function,assumption:g-bounded-below,assumption:g-Psi-growth} hold and $\mathcal{I}_{\mathrm{D}}(\delta) < \infty$ for some $\delta > 0$. 
    Then, the function $\mathcal{I}_{\mathrm{D} }(\delta)$ is continuous on $\mathbb{R}_{+}^2$.
\end{theorem}
Two implications follow. First, under \Cref{assumption:cost-function} and \Cref{assumption:g-bounded-below},  
\[\mathcal{I}_{\mathrm{D} }(0) = \sup_{\gamma \in \Pi(\mu_{1}, \mu_{2})} \int_{\mathcal{V} } g\, d\gamma.
  \]
Continuity facilitates sensitivity analysis as $\delta$ approaches zero; Second, 
under the assumptions in \Cref{thm:ID-continuity}, we have
 \begin{align*}
    \mathcal{I}_{\mathrm{D} }(\delta) = \inf_{\lambda \in \mathbb{R}_{+}^2}  \left[  \langle \lambda, \delta \rangle  +   \sup_{\varpi \in \Pi(\mu_1, \mu_2)} \int_{\mathcal{V} }g_\lambda \, d\varpi   \right]
    \end{align*}
for all $\delta \in \mathbb{R}_{+}^2$. As a result, the dual $\mathcal{J}_{\mathrm{D} }(\delta)$ in \eqref{eq:ID-dual} is continuous for all $\delta\in\mathbb{R}_+^2$.

\subsection{Overlapping Marginals}
\label{sec:continuity-overlapping}

\Cref{lemma:I-concavity} implies that under \Cref{assumption:cost-function,assumption:f-bounded-below}, $\mathcal{I}(\delta)$ is a concave function for $\delta \in \mathbb{R}_{+}^2$ and hence is continuous on $\mathbb{R}_{++}^2$. We provide the main assumption for the continuity of $\mathcal{I}(\delta)$ on $\mathbb{R}_{+}^2$ below.

To simplify the technical analysis, we maintain \Cref{assumption:metric-cost} in this section. Since the metrics in $\mathcal{Y}_1$ and $\mathcal{Y}_2$ are not specified, we introduce an auxiliary function $\rho_\ell$ from $\mathcal{Y}_\ell \times \mathcal{Y}_\ell$ to $\mathbb{R}_+$ induced by the cost function $c_\ell$, $\ell = 1, 2$.

\begin{assumption}\label{assumption:quasi-metric}
	For $\ell=1,2$, there exists a function $\rho_\ell$ from $\mathcal{Y}_\ell \times \mathcal{Y}_\ell$ to $\mathbb{R}_+$ such that  
	\begin{assumpenum}
	\item \label{assumption:quasi-metric-a} $\rho_\ell$ is symmetric, i.e., $\rho_\ell(y_\ell, y_\ell^\prime) =\rho_\ell(y_\ell^\prime, y_\ell)$ for all $y_\ell, y_\ell^\prime \in \mathcal{Y}_\ell$;
 
	\item \label{assumption:quasi-metric-ii}  there is $q_\ell\in [1, p_\ell]$ such that $\rho_\ell(y_\ell, y_\ell^\prime) \leq \boldsymbol{d}_{\mathcal{S}_\ell } (s_\ell, s_\ell^\prime)^{q_\ell}$ for all $s_\ell \equiv (y_\ell, x)  \in \mathcal{S}_\ell$ and $s^\prime_\ell \equiv (y^\prime_\ell, x^\prime)  \in \mathcal{S}_\ell$; 
 
	\item \label{assumption:quasi-metric-iii} there is a constant $N > 0$ such that  $\rho_\ell(y_\ell, y_\ell^\prime) \leq  N  \left[  \rho_\ell(y_\ell, y_\ell^\star ) + \rho_\ell(  y_\ell^\star , y_\ell^\prime) \right]$ for all $y_\ell, y_\ell^\prime, y_\ell^\star \in \mathcal{Y}_\ell$. 
	\end{assumpenum}
	\end{assumption}

We now introduce the main assumption on $f$.
\begin{assumption}\label{assumption:f-Psi-growth}
	For $\ell = 1,2$, let $\Psi_{\ell} : \mathbb{R}^2_+ \rightarrow \mathbb{R}_+$ be continuous, non-decreasing, and concave satisfying $\Psi_\ell(0,0) = 0$. Suppose for all $s = (y_1, y_2, x)$ and $s^\prime = (y^\prime_1, y^\prime_2, x^\prime)$, it holds that
	\[
	f(y_1, y_2, x) -  f(y_1^\prime, y_2^\prime, x^\prime)  \leq   \Psi_1 \left (  c_1(s_1, s_1^\prime), \rho_2(y_2,  y_2^\prime) \right),
	\]
	and 
	\[
	f(y_1, y_2, x) -  f(y_1^\prime, y_2^\prime, x^\prime)  \leq   \Psi_2 \left ( \rho_1(y_1, y_1^\prime), c_2(s_2, s_2^\prime)  \right).
	\]
	\end{assumption}

Like \Cref{assumption:g-Psi-growth}, \Cref{assumption:f-Psi-growth} depends on the cost functions $c_1,c_2$. It also depends on the auxiliary functions  $\rho_1,\rho_2$. The functions $\Psi_1, \Psi_2$ play the role of the modulus of continuity. 

\begin{example}[$p_j$-product metric]\label{p_prod} 
Let $(\mathcal{Y}_1,\boldsymbol{d}_{\mathcal{Y}_1} ), (\mathcal{Y}_2, \boldsymbol{d}_{\mathcal{Y}_2})$, and  $(\mathcal{X}, \boldsymbol{d}_{\mathcal{X} })$ be Polish (metric) spaces. For $p_\ell \geq 1$, define the $p_\ell$-product metric on $\mathcal{S}_\ell$ as
	 \[\boldsymbol{d}_{\mathcal{S}_\ell} (s_\ell, s_\ell^\prime) =   \left[ \boldsymbol{d}_{\mathcal{Y}_\ell}(y_\ell , y_\ell^\prime)^{p_\ell} +   \boldsymbol{d}_{\mathcal{X} }(x, x^\prime)^{p_\ell} \right]^{1/p_\ell}. \]
Let
  \[
\rho_\ell( y_\ell , y_\ell^\prime ):= \inf_{ x_\ell, x_\ell^\prime \in \mathcal{X} }  \boldsymbol{d}_{\mathcal{S}_\ell} \left( (y_\ell, x_\ell) , (y_\ell^{\prime}, x_\ell^{\prime}) \right)^{p_\ell}.
\]
It is easy to show that $ \rho_\ell( y_\ell , y_\ell^\prime ) =\boldsymbol{d}_{\mathcal{Y}_\ell}(y_\ell, y_\ell^\prime)^{p_\ell}$ and \Cref{assumption:quasi-metric} is satisfied with $N=2^{p_\ell}$. Moreover, \Cref{assumption:f-Psi-growth} reduces to 
\[
\begin{aligned}
    f(y_1, y_2, x) -  f(y_1^\prime, y_2^\prime, x^\prime)  &\leq   \Psi_1 \left(  \boldsymbol{d}_{\mathcal{S}_1}\left(s_1, s_1^{\prime}\right)^{p_1},   \boldsymbol{d}_{\mathcal{Y}_2}\left(y_2, y_2^{\prime}\right)^{p_2}  \right) \quad \text{and } \\
    f(y_1, y_2, x) -  f(y_1^\prime, y_2^\prime, x^\prime) & \leq   \Psi_2 \left(\boldsymbol{d}_{\mathcal{Y}_1}  \left(y_1, y_1^{\prime} \right)^{p_1},  \boldsymbol{d}_{\mathcal{S}_2}\left(s_2, s_2^{\prime}\right)^{p_2}  \right).
\end{aligned}	
\]
When  $p_1 = p_2 = p$, \Cref{assumption:f-Psi-growth} may be reduced to a simpler form. To see this, define two functions $\psi_1$ and $\psi_2$ from $\mathbb{R}^3$ to $\mathbb{R}^2$ as $\psi_1: (z_1,z_2, z) \mapsto (  z_1 + z ,  z_2)$ and  $\psi_2: (z_1,z_2, z) \mapsto ( z_1,  z_2 + z  )$. We can see that
	\[
	\Psi_1 \left(  \boldsymbol{d}_{\mathcal{S}_1}(s_1, s_1^\prime)^p, \rho_2(y_1, y_1^\prime)^{p}   \right)  = \Psi_1 \circ   \psi_1  \left(  \boldsymbol{d}_{\mathcal{Y}_1}(y_1, y_1^\prime)^{p} , \boldsymbol{d}_{\mathcal{Y}_2} (y_2, y_2^\prime )^p  ,    \boldsymbol{d}_{\mathcal{X}}(x, x^\prime)^p  \right), 
	\]
	\[
	\Psi_2 \left( \rho_1(y_1, y_1^\prime)^{p} ,  \boldsymbol{d}_{\mathcal{S}_2}(s_2, s_2^\prime)^p   \right)  = \Psi_2 \circ   \psi_2  \left(  \boldsymbol{d}_{\mathcal{Y}_1}(y_1, y_1^\prime)^{p} ,   \boldsymbol{d}_{\mathcal{Y}_2}(y_2, y_2^\prime )^p  ,    \boldsymbol{d}_{\mathcal{X}}(x, x^\prime)^p  \right).
	\]
	Since $\psi_j$ is linear, $\Phi_j  = \Psi_j \circ \psi_j$ is still continuous, non-decreasing and concave.  \Cref{assumption:f-Psi-growth} is reduced to the following condition:
	\[
	f(y_1, y_2, x) -  f(y_1^\prime, y_2^\prime, x^\prime)  \leq   \Phi_j  \left(  \boldsymbol{d}_{\mathcal{Y}_1}(y_1, y_1^\prime)^{p} ,    \boldsymbol{d}_{\mathcal{Y}_2} (y_2, y_2^\prime )^p  ,    \boldsymbol{d}_{\mathcal{X}}(x, x^\prime)^p  \right)
	\]
	for all $(y_1,y_2, x) \in \mathcal{S}$ and $(y_1^\prime, y_2^\prime, x^\prime) \in \mathcal{S}$. 
\end{example}

\begin{theorem} \label{thm:I-continuity}
    Suppose \Cref{assumption:f-bounded-below,assumption:finite-moments-I,assumption:metric-cost,assumption:quasi-metric,assumption:f-Psi-growth} hold, and $\mathcal{I}(\delta) < \infty$ for some $\delta > 0$.  Then 
    the function $\mathcal{I}(\delta)$ is continuous on $\mathbb{R}_{+}^2$.
\end{theorem}

Like the non-overlapping case, two implications follow. First, under \Cref{assumption:cost-function} and \Cref{assumption:g-bounded-below},  
\[
 \mathcal{I}(0)= \sup_{\gamma \in \mathcal{F}(\mu_{13}, \mu_{23}) } \int_\mathcal{S} f\,  d \gamma.
  \]
Continuity facilitates sensitivity analysis as $\delta$ approaches zero; Second, 
under the assumptions in \Cref{thm:I-continuity}, we have
   \begin{align*}
    \mathcal{I}(\delta) = \inf_{\lambda \in \mathbb{R}_{+}^2}  \left[  \langle \lambda, \delta \rangle  +   \sup_{\varpi \in \Pi(\mu_{13}, \mu_{23})} \int_{\mathcal{V} }f_\lambda \, d\varpi   \right]
    \end{align*}
    for all $\delta \in \mathbb{R}_{+}^2$. As a result, the dual $\mathcal{J}(\delta)$ in \eqref{eq:I-dual} is continuous for all $\delta\in\mathbb{R}_+^2$.

\section{Motivating Examples Revisited} \label{sec:examples-revisited}

In this section, we apply the results in Sections 3-5 to the examples introduced in Section 2. 

\subsection{Partial Identification of Treatment Effects} \label{sec:ATE-revisited}

In addition to characterizing $\Theta(\delta)$ introduced in Section 2, we also study the identified set for $\theta_{Do}= \mathbb{E}_o[f(Y_1,Y_2)]$ without using the covariate information:
\[
\Theta_{\mathrm{D}}(\delta) := \left\{  \int_{\mathcal{Y}_1 \times \mathcal{Y}_2 } f(y_1, y_2) \, d \gamma(y_1 ,y_2) : \gamma \in  \Sigma_{\mathrm{D} }(\delta) \right\},
\]
where 
\begin{align*}
    \Sigma_{\mathrm{D} }(\delta) = 
    \left\{\gamma \in \mathcal{P}(\mathcal{Y}_1 \times \mathcal{Y}_2): \boldsymbol{K}_{Y_1}(\mu_{Y_1}, \gamma_{1}) \le \delta_1,  \boldsymbol{K}_{Y_1}( \mu_{Y_2},\gamma_{2}) \le \delta_2 \right\}
\end{align*}
in which $\boldsymbol{K}_{Y_1}$ and $\boldsymbol{K}_{Y_2}$ are the optimal transport costs associated with cost functions $c_{Y_1}$ and $c_{Y_2}$, respectively. 

\subsubsection{Characterization of the Identified Sets}

When $f$ is continuous and conditions in \Cref{prop:Identified-Sets-Interval} are satisfied, the identified sets $\Theta_{\mathrm{D}}(\delta)$ and $\Theta(\delta)$ are both closed intervals with upper limits given by W-DMR for non-overlapping and overlapping marginals respectively. This allows us to apply our duality results in Section 3 to evaluate and compare $\Theta_{\mathrm{D}}(\delta)$ and $\Theta(\delta)$.

Let $\mathcal{I}_{\mathrm{D} }(\delta)$ and $\mathcal{I}(\delta)$ denote the upper bounds of $\Theta_{\mathrm{D}}(\delta)$ and $\Theta(\delta)$, respectively, where 
\[
	\mathcal{I}_{\mathrm{D} }(\delta) 
	= \sup_{\gamma \in \Sigma_{\mathrm{D} }(\delta)} \int_{\mathcal{Y}_1 \times \mathcal{Y}_2} f(y_1, y_2) \, d \gamma(y_1, y_2) \text{  and  } \mathcal{I}(\delta) 
	 = \sup_{\gamma \in \Sigma(\delta)} \int_{\mathcal{S}} f(y_1, y_2) \, d \gamma(y_1, y_2, x).
\]
\Cref{prop:Identified-Sets-Interval} establishes robust versions of existing results on the identified sets of treatment effects under \Cref{assumption:selection-on-observables}, see \citet{Fan2017}. Sensitivity to deviations from \Cref{assumption:selection-on-observables} can be examined via $\Theta_{\mathrm{D}}(\delta)$ and $\Theta(\delta)$ by varying $\delta$. For example, when $f$ satisfies assumptions in \Cref{thm:ID-continuity,thm:I-continuity}, $\mathcal{I}(\delta) $ and $\mathcal{I}_{\mathrm{D} }(\delta)$ are continuous on $\mathbb{R}_+^2$. As a result, 
\[\lim_{\delta\rightarrow 0 }\mathcal{I}(\delta)=\mathcal{I}(0) \quad \text{and} \quad \lim_{\delta\rightarrow 0 }\mathcal{I}_{\mathrm{D} }(\delta)=\mathcal{I}_{\mathrm{D} }(0).
\]

For a general function $f$, the lower and upper limits of the identified sets $\Theta_{\mathrm{D}}(\delta)$ and $\Theta(\delta)$ need to be computed numerically. When $f$ is additively separable, we show that duality results in Section 3 simplify the evaluation of $\Theta_D(\delta)$ and $\Theta(\delta)$. Since the lower bounds of $\Theta_D(\delta)$ and $\Theta(\delta)$ can be computed in a similar way by applying duality to $-f(y_1, y_2)$, we omit details for the lower bounds.

\begin{assumption}\label{assumption: diff_f}
	Let $f: (y_1,y_2, x) \mapsto f_1(y_1) + f_2(y_2)$ from  $\mathcal{S}$ to $\mathbb{R}$, where $f_\ell \in L^1(\mu_{\ell 3})$ for $\ell=1,2$.
	\end{assumption}
To avoid tedious notation, we also treat $f$ as a function from $\mathcal{Y}_1 \times \mathcal{Y}_2$ to $\mathbb{R}$. Under \Cref{assumption:cost-function,assumption: diff_f}, it is easy to show that 
\begin{align*}
 \mathcal{I}_{\mathrm{D}}(\delta) 
 & = \sup_{\gamma_1:\boldsymbol{K}_{Y_{1}}(\mu_{Y_1}, \gamma_1) \le \delta_1 }\int_{\mathcal{Y}_1}f_1 \, d\gamma_1 + \sup_{\gamma_2:\boldsymbol{K}_{Y_{2}}(\mu_{Y_2}, \gamma_{2}) \le \delta_2}\int_{\mathcal{Y}_2}f_2 \,  d\gamma_2 \\
 & = \inf_{\lambda_1 \geq 0 } \left[ \lambda_1 \delta_1 +\int_{\mathcal{Y}_1}(f_1)_{\lambda_1}   d \mu_1  \right] + \inf_{\lambda_2 \geq 0 } \left[ \lambda_2 \delta_2 +\int_{\mathcal{Y}_2}  (f_2 )_{\lambda_2}   d \mu_2  \right],
 \end{align*}
where $(f_\ell)_{\lambda_\ell}: \mathcal{Y}_\ell \rightarrow \mathbb{R}$ is given by 
	\[
	(f_\ell)_{\lambda_\ell}(y_\ell) = \sup_{ y_\ell^\prime \in \mathcal{Y}_\ell  }   \left\{   f_\ell(y_\ell^{\prime})-\lambda_\ell c_{Y_\ell}(y_{\ell}, y_{\ell}')\right\}.
	\]
That is, when $f$ is an additively separable function, the W-DMR for non-overlapping marginals is the sum of two W-DMRs associated with the marginals regardless of the cost functions.

Depending on the cost functions, the W-DMR for overlapping marginals may be different from the sum of two W-DMRs associated with the marginals.

\begin{definition}[Ref. {\citet{Chen_2022}}] 
	We say that a function $f: \mathcal{X} \times \mathcal{Y} \to \mathbb{R}$ is separable if each $x$ and $y$ can be optimized regardless of the other variable. In other words,  
	 	\begin{align*}
	 		\argmin_{x, y} f(x, y) = \left (\argmin_{x \in \mathcal{X}} f(x, y'),  \argmin_{y\in \mathcal{Y}} f(x', y) \right )
	 	\end{align*} 
	    for any $x' \in \mathcal{X}$  and $y' \in \mathcal{Y}$.  
\end{definition}
\begin{assumption}\label{assumption: separable_cost}
	For $\ell = 1, 2$, the cost function $c_\ell((y_\ell, x_\ell), (y_\ell', x_\ell'))$ is separable with respect to $(y_\ell, y_\ell')$ and $(x_\ell, x_\ell')$.
\end{assumption}
	
    \begin{example}
        Let $a_\ell: \mathcal{Y}_{\ell} \times \mathcal{Y}_{\ell} \to \mathbb{R}_{+} \cup \{\infty\}$ and $b_\ell: \mathcal{X} \times \mathcal{Y} \to \mathbb{R}_{+} \cup \{\infty\}$ satisfy \Cref{assumption:cost-function}. Let $s = (y, x)$ and $s' = (y', x')$.
        Then $c(s, s') = a(y, y') + b(x, x')$ is separable with respect to $(x, x')$ and $(y, y')$. Also, both  $c(s, s')=(a(y, y') +1)(b(x, x')+1)-1$ and $c(s, s') = \left[a(y, y')^p + b(x, x')^p\right]^{1/p}$ for $p \ge 1$ are separable with respect to $(x, x')$ and $(y, y')$ even though they are not additively separable.
    \end{example}

	\begin{proposition}  \label{prop:I-ID-equivalence-separable}
 For $\ell=1,2$, let $c_\ell: (\mathcal{Y}_\ell \times \mathcal{X}) \times  (\mathcal{Y}_\ell \times \mathcal{X}) \rightarrow \mathbb{R}_+$ denote the cost function for $\Theta(\delta)$. Suppose that $c_\ell$ satisfies \Cref{assumption:cost-function} and the marginal measure of  $\mu_{\ell 3}$ on $\mathcal{Y}_\ell$ coincides with $\mu_\ell$, i.e., $\mu_{\ell,3} = \mathrm{Law}(Y_\ell, X)$ with $\mu_\ell = \mathrm{Law}(Y_\ell)$. Under \Cref{assumption: diff_f,assumption: separable_cost}, one has $\mathcal{I}(\delta) = \mathcal{I}_{\mathrm{D} }(\delta)$, where $\mathcal{I}_{\mathrm{D} }(\delta)$ is based on the cost function $c_{Y_\ell}$ on $\mathcal{Y}_\ell \times \mathcal{Y}_\ell$ given by
	\[
	c_{Y_{\ell}}\left(y_{\ell}, y_{\ell}^{\prime}\right)=\inf _{x_{\ell}, x_{\ell}^{\prime} \in \mathcal{X}} c_{\ell}\left(\left(y_{\ell}, x_{\ell}\right),\left(y_{\ell}^{\prime}, x_{\ell}^{\prime}\right)\right).
	\]
 \end{proposition}
It is easy to verify that $c_{Y_\ell}(y_\ell, y_\ell^\prime) = 0$ if and only if $y_\ell = y_\ell^\prime$.

This proposition implies that for separable cost functions, the W-DMR for overlapping marginals equals the W-DMR for non-overlapping marginals with cost function $c_{Y_\ell}(y_\ell, y_\ell^\prime)$. As a result, the covariate information does not help shrink the identified set.

\subsubsection{Average Treatment Effect} 

Suppose $f(y_1, y_2) = y_2 - y_1$ and $c_{\ell}((y, x), (y_{\ell}, x_{\ell})) = |y - y'|^{2} + \| x_\ell - x_\ell' \|^2$ for $\ell = 1, 2$. Let $\tau_{ATE}=\mathbb{E}[Y_2 - Y_1]$. Then \Cref{prop:I-ID-equivalence-separable} implies that the upper bound on $\tau_{ATE}$ is given by
	\begin{align*}
		\mathcal{I}(\delta) 
		= \mathcal{I}_{\mathrm{D} }(\delta) 
		& = 
		\mathbb{E}[Y_2] - \mathbb{E}[Y_1] + \sqrt{\delta_1} + \sqrt{\delta_2}.
	\end{align*}
 
In the rest of this section, we demonstrate that when \Cref{assumption: separable_cost} is violated, the W-DMR for overlapping marginals may be smaller than the W-DMR for non-overlapping marginals and, as a result, $\Theta(\delta)$ is a proper subset of $\Theta_D(\delta)$. 

Consider the squared Mahalanobis distance with respect to a positive definite matrix. That is,
\begin{align*}
c_{\ell} (s_\ell, s_\ell') = (s_\ell - s_\ell')^{\top} V_\ell^{-1} (s_\ell - s_\ell'),
\end{align*}
where $V_\ell = \begin{pmatrix}
    V_{\ell, YY} & V_{\ell, YX} \\
    V_{\ell, XY} & V_{\ell, XX}
\end{pmatrix}$ is a positive definite matrix. It is easy to show that 
\begin{align*}
	c_{Y_\ell}(y_\ell , y_\ell') & =  \min_{x_\ell, x_\ell' \in \mathcal{X}_\ell'} c_{\ell} (s_\ell, s_\ell')  = (y_\ell - y_\ell')^{\top} V_{\ell, YY}^{-1} (y_\ell - y_\ell'),
\end{align*}
where $s_\ell = (y_\ell, x_\ell)$ and $s_{\ell}' = (y_\ell', x_\ell')$. 

\begin{proposition} \label{prop:ATE-Mahalanobis}
    Let $\mathcal{I}$ be the primal of the overlapping W-DMR problem under
    \begin{align*}
        c_{\ell} (s_\ell, s_\ell') = (s_\ell - s_\ell')^{\top} V_\ell^{-1} (s_\ell - s_\ell').
        \end{align*}
Let $\mathcal{I}_{\mathrm{D}}$ be the primal of the non-overlapping W-DMR problem under
$  c_{Y_\ell}(y_\ell ,y_\ell') $. Assume that $\mathbb{E} \|X\|_2^2 < \infty$, $\mathbb{E}|Y_1|^2 < \infty$, and $\mathbb{E}|Y_2|^2 < \infty$. Then, $ \mathcal{I}(\delta) \le \mathcal{I}_{\mathrm{D}}(\delta)$ for all $\delta > 0$.
\end{proposition}

\begin{proposition} \label{prop:ATE-Mahalanobis-Dualform}
Suppose that all the conditions in \Cref{prop:ATE-Mahalanobis} hold. Then,
\begin{propenum}
    \item \label{prop:ATE-Mahalanobis-Dualform_i} for all $\delta \in \mathbb{R}^2_{+}$, 
    \begin{align*}
        \mathcal{I}_{\mathrm{D}}(\delta) & =
        \mathbb{E}[Y_2] - \mathbb{E}[Y_1] + V_{1, YY}^{1/2} \; \delta_1^{1/2} + V_{2, YY}^{1/2}\;  \delta_2^{1/2}, \\
       \mathcal{I}(\delta) & = 
       \mathbb{E}[Y_2] - \mathbb{E}[Y_1]
       +
       \inf_{\lambda \in \mathbb{R}_{++}^2}
       \Bigg\{
        \lambda_1 \delta_1 + \lambda_2 \delta_2  +
       \frac{1}{4\lambda_{1}} \left(V_{1} / V_{1, XX} \right) + \frac{1}{4\lambda_{2}}  \left(V_{2} / V_{2, XX} \right)  \\
       & \qquad \qquad \qquad  \qquad  \qquad \qquad \quad  + 
       \frac{1}{4}
       V_o^{\top} 
       \left(\lambda_1 V_{1, XX}^{-1} + \lambda_2 V_{2, XX}^{-1}\right)^{-1} V_o
       \Bigg\},
    \end{align*}
    where $V_{\ell} / V_{\ell, XX} : = V_{\ell, YY} - V_{\ell, YX} V_{\ell, XX}^{-1} V_{\ell, XY}$ is the Schur complement of $V_{\ell, XX}$ in $V_{\ell}$ for $\ell = 1, 2$, and $V_o = V_{2, XX}^{-1} V_{2, XY} -  V_{1, XX}^{-1} V_{0, XY}$;

    \item \label{prop:ATE-Mahalanobis-Dualform_ii} $\mathcal{I}_{\mathrm{D}}(\delta) = \mathcal{I}(\delta)$ for all $\delta \in \mathbb{R}_{+}^2$ if and only if $V_{1, XY} = V_{2, XY} = 0$; 

    \item \label{prop:ATE-Mahalanobis-Dualform_iii} $\mathcal{I}_{\mathrm{D}}(\delta)$ and $\mathcal{I}(\delta)$ are continuous on $\mathbb{R}^2_+$.
\end{propenum}
\end{proposition}

\Cref{prop:ATE-Mahalanobis} and \Cref{prop:ATE-Mahalanobis-Dualform} imply that  for non-separable Mahalanobis cost functions, the information in covariates may help shrink the identified set since $\mathcal{I}_{\mathrm{D}}(\delta) < \mathcal{I}(\delta)$ for some $\delta$ under mild conditions. \Cref{prop:ATE-Mahalanobis-Dualform} also implies that  (i) $\mathcal{I}(0) = \mathcal{I}_{\mathrm{D}}(0) = \mathbb{E}[Y_2] - \mathbb{E}[Y_1]$; 
(ii) $\mathcal{I}(\delta_1, 0) = \mathcal{I}_{\mathrm{D}}(\delta_1, 0)$ and $\mathcal{I}(0, \delta_2) = \mathcal{I}_{\mathrm{D}}(0, \delta_2)$ for all $\delta_1 \ge 0$ and $\delta_2 \ge 0$.

\subsection{Comparison of Robust Welfare Functions} \label{sec:RobustWelfare-revisited}

Recall that
\begin{align*}
	\mathrm{RW}_0(d) & := \inf_{\gamma \in \Sigma_0(\delta)} \mathbb{E}[ Y_1 (1 - d(X)) + Y_2 d(X)] \quad \text{and } \\
	\mathrm{RW}(d) & := \inf_{\gamma \in \Sigma(\delta)} \mathbb{E}[ Y_1 (1 - d(X)) + Y_2 d(X)],
\end{align*}
where 
\begin{align*}
	\Sigma_0(\delta_0)& = \left\{ \gamma \in \mathcal{P}(\mathcal{S}):  \boldsymbol{K}(\mu, \gamma) \le \delta_0  \right\} \quad \text{and}  \\
 \Sigma(\delta) & = \left \{\gamma \in \mathcal{P}(\mathcal{S}):  \boldsymbol{K}_\ell(\mu_{\ell,3}, \gamma_{\ell,3}) \le \delta_\ell, \ \forall  \ell=1,2\right\}.
\end{align*}

Consider the following cost function $c_{\ell}$ for $\ell = 1, 2$:
\begin{align*}
	c_{\ell}(s_\ell, s_\ell') = c_{Y_\ell}(y_\ell, y_\ell') + b(x, x'),
\end{align*}
where $s_\ell = (y_\ell, x_\ell)$, $s_\ell' = (y_\ell', x_\ell')$, and $ c_{Y_1}(y_1, y_1')$ and $c_{Y_2}(y_2, y_2')$ are cost functions for $Y_1$ and $Y_2$, respectively, and $b(x, x')$ is some function on the space $\mathcal{X}$ satisfying \Cref{assumption:cost-function}. When $b(x, x') = \infty \mathds{1}\{x \ne x'\}$, $\mathbb{P}(X = X') = 1$ for any probability measure in uncertainty set. 

 \citet{adjaho2022} establishes strong duality for $\mathrm{RW}_0(d)$ under several cost functions. For comparison purposes, we restate the following Proposition in \citet{adjaho2022} which allows distributional shifts in covariate $X$. 

\begin{proposition} \label{prop:RW_AC} (Proposition 4.1 in \citet{adjaho2022})
Suppose $Y_1$ and $Y_2$ are unbounded and $\mathbb{E}\left\| X \right\|_2^2$ is finite. Let the cost function $c: \mathcal{S} \times \mathcal{S} \rightarrow \mathbb{R}_{+}$ be given by 
	\begin{align*}
		c(s, s') =  |y_1 - y_1'| + |y_2 - y_2'| + \| x' - x \|_2,
	\end{align*}
    for $s = (y_1, y_2, x)$ and $s' = (y_1', y_2', x')$. Then
	\begin{align*}
		\mathrm{RW}_0(d) =  \sup_{\eta \ge 1} \left\{ \mathbb{E}_{\mu}\left[\max\{Y_2 + \eta h_1(X), Y_1 + \eta h_0(X) \}\right] - \eta \delta_0 \right\}, \quad \text{where }
	\end{align*}
	\begin{align*}
		h_0(x) = \inf_{u \in \mathcal{X}: d(u) = 0} \| x - u \|_2 \text{ and } h_1(x) = \inf_{u \in \mathcal{X}: d(u) = 1} \| x - u \|_2.
	\end{align*}
\end{proposition}

This proposition implies that $\mathrm{RW}_0(d)$ depends on the choice of the reference measure $\mu$. Since only the marginals $\mu_{13}$ and $\mu_{23}$ are identified under \Cref{assumption:selection-on-observables},
\citet{adjaho2022} suggest three possible choices for $\mu$ by imposing specific dependence structures on $\mu$:
	\begin{itemize}
		\item $Y_1$ and $Y_2$ are perfectly positively dependent conditional on $X = x$;
		
		\item $Y_1$ and $Y_2$ are conditionally independent given $X = x$;
		
		\item $Y_1$ and $Y_2$ are perfectly negatively dependent conditional on $X = x$.
	\end{itemize} 
Section 4.3.1 in \citet{adjaho2022} shows that their robust welfare function $\mathrm{RW}_0(d)$ is minimized when  $Y_1$ and $Y_2$ are perfectly negatively dependent conditional on $X = x$.

The following proposition evaluates $\mathrm{RW}(d)$ via the duality result in Section 3 and compares it with $\mathrm{RW}_0(d)$.
\begin{proposition} \label{prop:RW-L1-with-X-L2}
	Consider
	\begin{align*}
		c_{\ell}(s_\ell, s_\ell') & =  | y_\ell - y_\ell' | + \| x_\ell - x_\ell' \|_2. 
	\end{align*}
Assume that $Y$ is unbounded and $\mathbb{E}|Y_1|$, $\mathbb{E}|Y_2|$, and $\mathbb{E}\| X \|_2^2$ are finite. Then, 
\begin{propenum}
     \item \label{prop:RW-L1-with-X-L2_i} the robust welfare function $\mathrm{RW}(d)$ based on $\Sigma(\delta)$ has the following dual reformulation:
	\begin{align*}
		\mathrm{RW}(d) =
		\sup_{\lambda \ge 1}
		\left[
		\inf_{\pi \in \Pi(\mu_{13}, \mu_{23})}  \int_{\mathcal{V}} \min\{ y_2 + \varphi_{\lambda, 1} (x_1, x_2), y_1 +  \varphi_{\lambda, 0} (x_1, x_2) \} d \pi(v) - \langle \lambda, \delta \rangle
		\right],
	\end{align*} 
	where $v = (y_1, x_1, y_2, x_2)$, and 
	\begin{align*}
		\varphi_{\lambda, 0} (x_1, x_2) & = \min_{x': d(x') = 0} \bigg( \lambda_1 \| x_1 - x' \|_2 + \lambda_2 \| x_2 - x' \|_2 \bigg), \\
		\varphi_{\lambda, 1} (x_1, x_2) & =  \min_{x': d(x') = 1} \bigg( \lambda_1 \| x_1 - x' \|_2 + \lambda_2 \| x_2 - x' \|_2 \bigg);
  \end{align*}	

 \item \label{prop:RW-L1-with-X-L2_ii} When $\delta_0 = \delta_1 = \delta_2$, $\mathrm{RW}(d)\le \mathrm{RW}_0^*(d)$, where $\mathrm{RW}_0^*(d)$ is the robust welfare function $\mathrm{RW}_0(d)$ based on the reference measure $\pi^*=\int\max\{\mu_{1|3}+\mu_{2|3}-1,0 \}d\mu_3$.

 \end{propenum}

\end{proposition}

Part (ii) of the above proposition implies that $\mathrm{RW}(d) \le \mathrm{RW}_0(d)$ for any reference measure $\mu \in \mathcal{F}(\mu_{13}, \mu_{23})$.

\subsection{W-DRO for Logit Model Under Data Combination}

We revisit the logit model in \Cref{sec:example-logit} and make the following assumption.
\begin{assumption} \label{assumption:Data-Logit}
    (i) Let $(Y_1, Y_2, X)$ follow some unknown measure $\mu$. Let $D$ denote a binary random variable  independent of $(Y_1, Y_2, X)$ such that we observe $(Y_1, X)$ when $D = 0$, and $(Y_2, X)$ when $D = 1$; (ii) Let $\{Y_{1i}, X_{1i}\}_{i=1}^{n_1}$ be the data set from $(Y_1, X)$, and $\{Y_{2i}, X_{2i}\}_{i=1}^{n_2}$ be the data set from $(Y_2, X)$. 
\end{assumption}
Under this assumption, $X|D=1$ has the same distribution as $X|D=0$ and the empirical distributions of the two data sets are consistent estimators of the population reference measures for $(Y_1, X)$ and $(Y_2, X)$. 

Suppose \Cref{assumption:cost-function,assumption:f-bounded-below} hold. Then
\Cref{thm:I-duality} implies that for all $\delta > 0$, 
	\begin{align*}
		\mathcal{I}(\delta) 
		=
		\inf_{\lambda \in \mathbb{R}_{+}^2}  \left[  \langle \lambda, \delta \rangle  +   \sup_{\varpi \in  \Pi(\mu_{13}, \mu_{23})} \int_{\mathcal{V} }f_{\theta, \lambda} \, d \varpi   \right],
	\end{align*}
	where
	\begin{align*}
		f_{\theta, \lambda} (v) 
        & = \sup_{y_1', y_2', x'}
		\left[f(y_1', y_2', y; \theta) - \lambda_1 c_1((y_1, x_1), (y_1', x')) - \lambda_2 c_2((y_2, x_2, y_2', x'))  \right]
	\end{align*}
with $v = (y_1, x_1, y_2, x_2)$. 
    
Let $\widehat{\mu}_{13}$ and $\widehat{\mu}_{23}$ denote the empirical measures based on the two data sets. The dual form of $\mathcal{I}(\delta)$ can be estimated by
    \begin{align*}
        \widehat{\mathcal{I}}(\delta) 
        & :=
        \inf_{\lambda \in \mathbb{R}_{+}^2}  \left[  \langle \lambda, \delta \rangle  +   \sup_{\varpi \in  \Pi(\widehat{\mu}_{13}, \widehat{\mu}_{23})} \int_{\mathcal{V} }f_{\theta, \lambda} \, d \varpi  \right] .
    \end{align*}
A direct consequence of \citet[Proposition 2.1]{Kellerer:1984to} is that 
 \begin{align*}
        \widehat{\mathcal{I}}(\delta) 
        & = 
        \inf_{\lambda \in \mathbb{R}_{++}^2, \{\varphi_i\}_{i=1}^{n_1}, \{\varphi_j\}_{j=1}^{n_2}}
        \left[
            \langle \lambda, \delta\rangle + \frac{1}{n_1} \sum_{i=1}^{n_1} \varphi_i + \frac{1}{n_2} \sum_{j=1}^{n_2} \varphi_j
        \right] \\
        & \quad \text{ such that } f_{\theta, \lambda} (s_{1i}, s_{2j}) \le \varphi_i + \varphi_j' \text{ for any } i \in [n_1] \text{ and } j \in [n_2],
    \end{align*}
where the last expression reduces to the dual in  \citet{awasthi2022distributionally} for the cost functions  
\[
c_1((y_1, x), (y_1', x'))  = \| x - x' \|_{p} + \kappa_1 | y_1 - y_1'| \quad \text{and}
\]
\[
c_2((y_2, x), (y_2,  x'))  = \| x - x' \|_{p} + \kappa_2 \| y_2 - y_2' \|_{p'}.
\]

\section{W-DMR with Multi-marginals} \label{sec:multi-marginals}

Sections 2-6 present a detailed study of W-DMR with two marginals. In this section, we briefly introduce W-DMR with more than two marginals or multi-marginals and discuss strong duality for non-overlapping and overlapping marginals.\footnote{For multi-marginals, the collection of given marginals can be more complicated than the non-overlapping and overlapping marginals (see \citet{ruschendorf1991bounds}, \citet{Embrechts_2010} and \citet{Doan2015}), we leave a complete treatment of the W-DMR with multi-marginals in future work.} Applications include extension of risk aggregation in \Cref{sec:example-risk-aggregation} to any finite number of individual risks and robust treatment choice in \Cref{sec:example-risk-aggregation} to multi-valued treatment.

\subsection{Non-overlapping Marginals}

Let $\mathcal{V}:=\prod_{\ell \in [L]} \mathcal{S}_\ell$ for Polish spaces $\mathcal{S}_\ell$ for $\ell \in [L]$, and $\mu_\ell$ be a probability measure on $(\mathcal{S}_{\ell}, \mathcal{B}_{\mathcal{S}_\ell})$. Let $\Pi(\mu_1, \dotsc, \mu_L)$ be the set of all possible couplings of $\mu_1, \dotsc, \mu_L$. Further, let $g:\mathcal{V}\rightarrow \mathbb{R}$ be a measurable function satisfying the following assumption.
\begin{assumption}\label{assumption:g-bounded-below-mulitple-outcomes} 
The function $g:\mathcal{V}\rightarrow \mathbb{R}$ is a measurable function such that $\int_{\mathcal{V}} g d\gamma_0> -\infty$ for some $\gamma_0 \in \Pi(\mu_1, \dotsc, \mu_L) \subset \mathcal{P}(\mathcal{V})$. 
\end{assumption}

For any $\gamma \in \mathcal{P}(\mathcal{V})$, let $\gamma_{\ell}$ denote the projection of $\gamma$ on $\mathcal{S}_{\ell}$ for $\ell \in [L]$. 
The W-DMR with non-overlapping multi-marginals is formulated as
\begin{align*}
	\mathcal{I}_{\mathrm{D}} (\delta) = \sup_{\gamma \in \Sigma_{\mathrm{D} }(\delta)} \int_{\mathcal{V}} g d \gamma,
\end{align*}
where $\Sigma_{\mathrm{D} }(\delta)$ is the uncertainty set defined as
\begin{align*}
	\Sigma_{\mathrm{D} }(\delta) = \{\gamma \in \mathcal{P}(\mathcal{V}): \boldsymbol{K}_\ell (\mu_\ell, \gamma_\ell) \le \delta_\ell, \ \forall \ell \in [L]\}
\end{align*}
in which $\delta = (\delta_1, \dotsc, \delta_L) \in \mathbb{R}_{+}^L$ is the radius of the uncertainty set.

For a generic vector $v \in \mathbb{R}^L$ and $A \subset [L]$, we write $v_A = (v_{A, 1}, \dotsc, v_{A, L}) \in \mathbb{R}^{L}$ as follows:
\begin{align*}
    v_{A, \ell}
    & =
    \begin{cases}
        v_{\ell} & \text{ if } \ell \in A, \\
        0 & \text{ if } \ell \notin A.
    \end{cases}
\end{align*}
We also define $\tilde{c}_\ell: \mathcal{S}_{\ell} \times \mathcal{S}_{\ell} \to \mathbb{R}_{+} \cup \{\infty\}$ as:
\begin{align*}
           \tilde{c}_{\ell}(s_{\ell}, s_{\ell}') 
           =
           \begin{cases}
               c_{\ell}(s_{\ell}, s_{\ell}') & \text{ if } \ell \in A, \\
               \infty \mathds{1}\{s_{\ell} \ne s_{\ell}'\} & \text{ if } \ell \notin A.
           \end{cases}
\end{align*}
For a function $g: \mathcal{V} \rightarrow \mathbb{R}$ and $\lambda := (\lambda_1, \dotsc, \lambda_L) \in \mathbb{R}^L_{+}$, we define the function $g_{\lambda, A}: \mathcal{V} \rightarrow \mathbb{R} \cup \{\infty \}$ as 
\begin{align*}
     g_{\lambda, A}(v) =
           \sup_{v' \in \mathcal{V}}
           \left\{g(v') - \sum_{\ell=1}^{L} \lambda_\ell \tilde{c}_{\ell}\left\{ s_{\ell}, s_{\ell}' \right\} \right\}
\end{align*}
with $v:= (s_1, \dotsc, s_L)$ and $v^\prime := (s_1^\prime, \dotsc, s_L^\prime)$.

\begin{theorem}[Non-overlapping case] \label{thm:ID-duality-multi-marginals}
Suppose that \Cref{assumption:cost-function,assumption:g-bounded-below-mulitple-outcomes} hold. 
Then, for any $\delta \in \mathbb{R}_{++}^L$ and $A \subset [L]$, we have
       \begin{align*}
           \mathcal{I}_{\mathrm{D} }(\delta_{A} ) =
           \inf_{\lambda \in \mathbb{R}_{+}^L}
		\left[ \langle \lambda, \delta_{A} \rangle +  \sup_{\pi \in \Pi(\mu_1, \dotsc, \mu_L)}   \int_{\mathcal{V}} g_{\lambda, A} \, d \pi      \right].
       \end{align*}
\end{theorem}

In practice, the dual in \Cref{thm:ID-duality-multi-marginals} involves the computation of the multi-marginal problem, $\sup_{\pi \in \Pi(\mu_1, \dotsc, \mu_L)}   \int_{\mathcal{V}} g_{\lambda} \, d \pi$, see \citet{Pass_2010,Pass_2012,Pass_2015,von_Lindheim_2022,Nenna_2022,mehta2023efficient} for detailed studies of properties and computation of multi-marginal problems for specific functions $g_{\lambda}$. For general possibly non Borel-measurable $g_{\lambda}$, the strong duality in \citet{Kellerer:1984to} could be applied. The established result is stated in \Cref{corollary:ID-dual-full-multi-marginals} in \Cref{appendix:duality-OT}.

\subsection{Overlapping Marginals}

Let $\mathcal{S} := \left(\prod_{\ell \in [L]} \mathcal{Y}_\ell \right) \times \mathcal{X}$, where $\mathcal{Y}_\ell$ for $\ell \in [L]$ and $\mathcal{X}$ are Polish spaces. Let $\mathcal{S}_\ell := \mathcal{Y}_\ell \times \mathcal{X}$ for $\ell \in [L]$. Let $\mu_{\ell,L+1} \in \mathcal{P}(\mathcal{S}_\ell)$ for $\ell \in [L]$ be such that the projections of $\mu_{\ell, L+1}$ on $\mathcal{X}$ are the same for  $\ell \in [L]$. We call the Fr\'{e}chet class of all probability measures on $\mathcal{S}$ having marginals $\left(\mu_{1,L+1} \right)_{\ell \in [L]}$ the Fr\'{e}chet class with overlapping marginals and denote it as $\mathcal{F}\left(\mathcal{S};  (\mu_{\ell,L+1})_{\ell \in [L]}  \right) :=  \mathcal{F}\left(  \left(\mu_{\ell,L+1} \right)_{\ell \in [L]}  \right)$. This class is the star-like system of marginals in \citet{ruschendorf1991bounds} and \citet{Embrechts_2010}, see also 
 \citet{Doan2015}.

Moreover, let $f:\mathcal{S} \rightarrow \mathbb{R}$ be a measurable function satisfying the following assumption. 
\begin{assumption} \label{assumption:f-multi-marginals} The function $f:\mathcal{S} \rightarrow \mathbb{R}$ is a measurable function  such that $\int_{\mathcal{S}} f \, d\nu_0  > - \infty$ for some $\nu_0 \in \Pi(\mu_{1,L+1},..., \mu_{L,L+1}) \subset \mathcal{P}(\mathcal{S})$.
\end{assumption}

For any $\gamma \in \mathcal{P}(\mathcal{S})$, let $\gamma_{\ell, L+1}$ denote the projection of $\gamma$ on $\mathcal{Y}_{\ell} \times \mathcal{X}$ for $\ell \in [L]$. 
Similar to the two marginals case, the W-DMR with overlapping multi-marginals is defined as
\begin{align*}
	\mathcal{I} (\delta) = \sup_{\gamma \in \Sigma(\delta)} \int_{\mathcal{S}} f \, d \gamma,
\end{align*}
where $\Sigma(\delta)$ is the uncertainty set defined as
\begin{align*}
	\Sigma(\delta) = \{\gamma \in \mathcal{P}(\mathcal{S}):  \boldsymbol{K}_\ell(\mu_{\ell, L+1}, \gamma_{\ell, L+1}) \le \delta_\ell \text{ for } \ell \in [L]\},
\end{align*}
in which $\delta = (\delta_1, \dotsc, \delta_L) \in \mathbb{R}_{+}^L$ is the radius of the uncertainty set.

For a function $f: \mathcal{V} \rightarrow \mathbb{R}$, $\lambda := (\lambda_1, \dotsc, \lambda_L) \in \mathbb{R}^L_{+}$, and $A \subset [L]$, we define the function $f_{\lambda, A}: \mathcal{V} \to \overline{\mathbb{R}}$ as follows:
\begin{align*}
           f_{\lambda, A}(v) =
           \sup_{s' \in \mathcal{S}}
           \left\{f(s') - \sum_{\ell=1}^{L} \lambda_\ell \tilde{c}_{\ell}(s_{\ell}, s_{\ell}')\right\}, 
       \end{align*}
       where $v = (s_1, \dotsc, s_L)$, $s^\prime = (y^\prime_1, \dotsc, y^\prime_L, x^\prime)$, $s^\prime_\ell = (y^\prime_\ell, x^\prime)$ and $s_\ell=(y_\ell, x_\ell)$, and
       \begin{align*}
           \tilde{c}_{\ell}(s_{\ell}, s_{\ell}^\prime) 
           =
           \begin{cases}
               c_{\ell}(s_{\ell}, s_{\ell}') & \text{ if } \ell \in A, \\
               \infty \mathds{1}\left\{s_{\ell} \ne s_{\ell}'\right\} & \text{ if } \ell \notin A.
           \end{cases}
       \end{align*}

\begin{theorem}[Overlapping case] \label{thm:I-duality-multi-marginals}
     Suppose that \Cref{assumption:cost-function,assumption:f-multi-marginals} hold. 
     Then, for any $\delta \in \mathbb{R}^L_{++}$ and $A \subset [L]$, we have
       \begin{align*}
           \mathcal{I}(\delta_{A} ) =
           \inf_{\lambda \in \mathbb{R}_{+}^L}
		\left[ \langle \lambda, \delta_{A} \rangle +  \sup_{\pi \in \Pi(\mu_{1, L+1},\dotsc, \mu_{L,L+1})}   \int_{\mathcal{V}} f_{\lambda, A} \, d \pi      \right].
       \end{align*}
\end{theorem}

Similar to the nonoverlapping case, strong duality holds for the inner multi-marginal problem under additional conditions. The result is stated in \cref{corollary:I-dual-full-multi-marginals} of \Cref{appendix:duality-OT}.

\subsection{Treatment Choice for Multi-valued Treatment}

We apply strong duality to multi-valued treatment in \citet{Kido2022}. Let $d: \mathcal{X} \rightarrow [L]$ be a policy function or treatment rule on $\mathcal{X}$ and $Y_\ell \in \mathbb{R}$ denote the potential outcome under the treatment $\ell$ for $\ell \in [L]$. Consider the policy function defined as 
\begin{align*}
	Y(d) := \sum_{\ell=1}^L  Y_\ell \times \mathds{1}\{d(X) = \ell\} .
\end{align*}
\citet{Kido2022} introduces the following robust welfare function.
\begin{align*}
    	\mathrm{RW}_{C}(d) = \sup_{\gamma \in\Sigma_{\mathrm{M}}(\delta_0)} \mathbb{E}_{\gamma} \left[ \sum_{\ell=1}^L Y_\ell \mathds{1}\{d(X) = \ell\} \right],
    \end{align*}
where the uncertainty set $\Sigma_{\mathrm{M}}(\delta_0)$ is based on the conditional distribution of $(Y_\ell)_{\ell \in [L]}$ given $X$:
\begin{align*}
	\Sigma_{\mathrm{M}}(\delta_0)
	:= 
    \left\{\gamma \in \mathcal{P}(\mathcal{S}): \boldsymbol{K}(\mu_{(Y_1, \dotsc, Y_L)|X=x}, \gamma_{(Y_1, \dotsc, Y_L)|X=x}) \le \delta_0 \text{ for all } x, \; \mu_{X} = \gamma_X\right\}, 
\end{align*}
in which the cost function $c$ associated with $\boldsymbol{K}$ is
\begin{align*}
	c((y_1, \dotsc, y_L), (y_1', \dotsc, y_L')    ) = \sum_{\ell=1}^L |y_\ell - y_\ell'|.
\end{align*}
Note that the uncertainty set $\Sigma_{\mathrm{M}}(\delta_0)$ does not allow any potential shift\footnote{\citet{Kido2022} mentions the possibility of allowing for covariate shift by incorporating uncertainty sets in \citep[e.g.,][]{Mo2020,Zhao2019} for the distribution of the covariate in future work.
} in $X$.
When $Y_1, \dotsc, Y_L$ are unbounded, \citet{Kido2022} shows that 
	\begin{align*}
		\mathrm{RW}_{C}(d) 
		& = \sum_{\ell=1}^L \mathbb{E}_{(Y_\ell, X) \sim \mu_{\ell, L+1}} \left[ (Y_\ell  - \delta_0)  I(D(X) = \ell)  \right] \\
		& = \mathbb{E}_{X} \left[ \sum_{\ell=1}^L \left(\mathbb{E}[Y_\ell \mid X] - \delta_0 \right) I(D(X) = \ell)  \right].
	\end{align*}
We apply W-DMR for overlapping marginals with the following cost function:
\begin{align*} 
	c_{\ell}(s_\ell, s_\ell') = | y_\ell - y_\ell'| + \| x_{\ell} - x_\ell' \|_{2} ,
\end{align*}
and define a robust welfare function as
\begin{align*}
	\mathrm{RW}(d) = \sup_{\gamma \in \Sigma(\delta)} \mathbb{E}_{\gamma} \left[ \sum_{\ell=1}^L Y_\ell I(d(X) = \ell) \right] .
\end{align*}

\begin{proposition} \label{prop:RW-L1-with-X-L2-multi-marginals}
	For $\ell \in [L]$, let 
	\begin{align*}
		c_{\ell}(s_\ell, s_\ell') = | y_\ell - y_\ell'| + \| x_{\ell} - x_\ell' \|_{2}.
	\end{align*}
Assume that $Y_\ell$ is unbounded, $\mathbb{E}[\|X\|_2^2] < \infty$ and $\mathbb{E}[|Y_\ell|] < \infty$. Then
	\begin{align*}
		\mathrm{RW}(d) 
		& =
		\sup_{\lambda \ge 1}
		\left\{
		\inf_{\pi \in \Pi(\mu_{1,L+1}, \dotsc, \mu_{L, L+1})} \int_{\mathcal{V}} \min_{\ell \in [L]} \{y_\ell + \phi_{\lambda, \ell}(x_1, \dotsc, x_L)  \} d \pi(s) - \langle \lambda, \delta \rangle \right\},
	\end{align*}
	where 
	\begin{align*}
		\varphi_{\lambda, \ell}(x_1, \dotsc, x_L) = 
		\min_{x', d(x') = \ell} \sum_{\ell=1}^{L} \lambda_\ell \| x_\ell - x' \|_{2}.
	\end{align*}
\end{proposition}
\Cref{prop:RW-L1-with-X-L2-multi-marginals} is an extension of \Cref{prop:RW-L1-with-X-L2}.

\section{Concluding Remarks}

In this paper, we have introduced W-DMR in marginal problems for both non-overlapping and overlapping marginals and established fundamental results including strong duality, finiteness of the proposed W-DMR, and existence of an optimizer at each radius. We have also shown continuity of the W-DMR-MP as a function of the radius. Applicability of the proposed W-DMR in marginal problems and established properties is demonstrated via distinct applications when the sample information comes from multiple data sources and only some marginal reference measures are identified. 
To the best of the authors' knowledge, this paper is the first systematic study of W-DMR in marginal problems. Many open questions remain including the structure of optimizers of W-DMR for both non-overlapping and overlapping marginals, efficient numerical algorithms, and estimation and inference in each motivating example. Another useful extension is to consider objective functions that are nonlinear in the joint probability measure such as the Value-at-Risk of a linear portfolio of risks in \citet{Puccetti_2012} and robust spectral measures of risk in \citet{Ghossoub_2023,Ennaji2022}. 

\newpage

\printbibliography

\newpage

\appendix

\section{Appendix A: Preliminaries}

In this appendix, we provide a self-contained review of interchangeability principle, strong duality for marginal problems, and probability measures given marginals.

Additional notations used in the appendices are collected here. For any set $A$, we denote by $2^A$ the power set of $A$. Suppose $f: \mathcal{X} \rightarrow \mathcal{Y}$ and $g: \mathcal{Y} \rightarrow \mathcal{Z}$, let $g\circ f$ denote the composite of $f$ and $g$, i.e., a map $x \mapsto g (f(x))$ that maps $\mathcal{X}$ into $\mathcal{Z}$.  Given a Polish measurable space $(\mathcal{X}, \mathcal{B}_{\mathcal{X}})$. For any Borel measure $\mu$ on $\mathcal{X}$, let $\mathrm{Supp}(\mu)$ denote the smallest closed set $A \subset \mathcal{X}$ such that $\mu(A) = 1$. For $j \in [n]$, let $\mathcal{X}_j$ be a Polish space equipped with Borel $\sigma$-algebra $\mathcal{B}_{\mathcal{X}_j}$.   Let $\mathcal{S} := \prod_{j \in [n]} \mathcal{X}_j$. For any subset $K \subset [n]$, we write $\mathcal{S}_K := \prod_{j\in K} \mathcal{X}_j$ and the projection map $\operatorname{proj}_{K}$ from $\mathcal{S}$ to $\mathcal{S}_K$ as $\operatorname{proj}_{K}: (x_j)_{j\in [n]} \mapsto (x_j)_{j\in K}$.  Given $\mu_j\in \mathcal{P}(\mathcal{X}_j)$ for $j \in [n]$, let $\Pi(\mu_1, \cdots, \mu_n)$ denote the set of probability measures $\mu$ on $\mathcal{S}$ such that $\mathrm{proj}_{\{j\}} \# \mu = \mu_j$. Finally, let $\mathcal{N}(\mu, \Sigma)$ denote the multivariate normal distribution with mean $\mu$ and covariance matrix $\Sigma$.

\subsection{Interchangeability Principle} \label{appendix:interchangeability}

Let $(\mathcal{T}, \mathcal{B}_{\mathcal{T}},\mu)$ be a probability space, $(\mathcal{X},  \mathcal{B}_{\mathcal{X}} )$ be a sample  space and $\varphi: \mathcal{T}\times \mathcal{X} \rightarrow  \overline{\mathbb{R}}$ be a measurable function.  We denote by $\mathcal{B}_{\mathcal{T}}^{\mu}$ the $\mu$-completion of $\mathcal{B}_{\mathcal{T}}$. Let $\Gamma(\mu,\varphi )$ denote the set of probability measures $\pi$ on $(\mathcal{T} \times \mathcal{X},   \mathcal{B}_{\mathcal{T}} \otimes \mathcal{B}_{\mathcal{X}} )$ such that  $\pi(A \times  \mathcal{X}) = \mu(A)$ for all $A \in \mathcal{B}_{\mathcal{T}}$ and $\int_{ \mathcal{T} \times \mathcal{X}} \varphi d\pi $ is well defined. If there is no such $\pi$,  the natural convention is to take $\Gamma(\mu, \varphi) =\emptyset$.

\begin{definition}
A measurable function $\varphi:\mathcal{T} \times \mathcal{X} \rightarrow \overline{\mathbb{R}}$ satisfies the interchangeability principle with respect to $\mu$ if the function $t \mapsto \sup_{x \in \mathcal{X}} \varphi(t,x)$ is $\mathcal{B}_\mathcal{T}^{\mu}$-measurable and satisfies
\[
\int_{\mathcal{T}}  \left[ \sup_{x\in\mathcal{X} } \varphi( t,x)  \right] d\mu(t)  = \sup_{\pi \in \Gamma(\mu, \varphi) } \int_{\mathcal{T} \times \mathcal{X} }  \varphi\left(t,x\right) d \pi (t,x).
\]
\end{definition} 
The interchangeability principle allows us to interchange the supremum and integral operators. It is a weak condition. As explained in Example 2 in \citet{zhang2022simple}, this condition is satisfied when the space is Polish and $\varphi$ is measurable. 

We extend the interchangeability principle with respect to a measure to a class of measures in the definition below.
\begin{definition}
Let $\mathcal{G}$ be a set of probability measures on $(\mathcal{X}, \mathcal{B}_{\mathcal{X}})$.
A measurable function $\varphi:\mathcal{T} \times \mathcal{X} \to \overline{\mathbb{R}}$ satisfies the  interchangeability principle  with respect to $\mathcal{G}$ if $\varphi$ satisfies the {\it interchangeability principle} with respect to $\mu$ for all $\mu \in \mathcal{G}$.
\end{definition}

\begin{lemma}\label{IPforGroup}
Suppose that $\varphi$  satisfies the interchangeability principle with respect to $\mathcal{G}$. Let $\Gamma(\mathcal{G}, \varphi) := \cup_{\mu \in \mathcal{G}} \Gamma(\mu, \varphi)$. Then,
\[
\sup_{\pi \in  \Gamma(\mathcal{G}, \varphi )  } \int_{\mathcal{T} \times \mathcal{X} }  \varphi(t,x) d\pi(t,x) =  \sup_{\mu \in \mathcal{G} }  \left\{ \int_{\mathcal{T}} \left[ \sup_{x \in \mathcal{X} }  \varphi(t,x)  \right] d \mu(t) \right\} .
\]
\end{lemma}

\begin{proof}
With the convention, we write $\sup A = -\infty$ if $A =\emptyset$. It is easy to see that 
\begin{align*}
\sup_{\pi \in  \Gamma(\mathcal{G}, \varphi )  } \int_{\mathcal{T} \times \mathcal{X} }  \varphi  d\pi & = \sup_{\mu \in \mathcal{G}}  \left[  \sup_{\pi \in \Gamma(\mu, \varphi) }     \int_{\mathcal{T} \times \mathcal{X}} \varphi(t, x) d \pi(t, x)  \right] \\
  & = \sup_{\mu \in \mathcal{G}} \left[\sup _{\pi \in \Gamma(\mu, \varphi)} \int_{\mathcal{T} \times \mathcal{X} } \varphi d \pi \right]= \sup_{\mu \in \mathcal{G} }  \left\{ \int_{\mathcal{T}} \left[ \sup_{x \in \mathcal{X} }  \varphi(t,x)  \right] d \mu(t) \right\}.
\end{align*}
\end{proof}

\subsection{Strong Duality for Marginal Problems} \label{appendix:duality-OT}

The strong duality results for marginal problems are well-established in the literature, see \citet{Kellerer:1984to,villani2009optimal,villani2021topics,Beiglboeck2011}. Here, we present a strong duality result based on \citet{Kellerer:1984to}.
\begin{theorem} \label{thm:duality-OT} 
Given probability measures $\mu_\ell$ on Polish space $\mathcal{X}_\ell$ equipped with Borel algebra $\mathcal{B}_{\mathcal{X}_\ell}$ for $\ell \in [L]$. Let $\mathcal{X} = \prod_{\ell \in [L] } \mathcal{X}_\ell$ and $f: \mathcal{X} \rightarrow \overline{\mathbb{R}}$ be an extended real-valued function. Consider the following marginal problem:
\[
\sup_{\pi \in \Pi(\mu_1, \ldots, \mu_L)}  \int_{\mathcal{X}} f(x) d \pi (x).
\]
Suppose $\{x \in \mathcal{X}:f(x) \ge u \}$ is analytic for all $u \in \overline{\mathbb{R}}$ and there exist $f_\ell < \infty$, $f_\ell \in L^1(\mu_\ell)$ for $\ell \in [L]$ such that $f(x) \ge \sum_{\ell=1}^{L} f_\ell(x_\ell)$ for all $x := (x_1, \dotsc, x_L) \in \mathcal{X}$.
Let $\Phi_f$ be the set of all measurable functions $\left(\phi_\ell \right)_{\ell\in [L]}$, where $\phi_\ell \in L^1(\mu_\ell)$ and $\phi_\ell > - \infty$ for all $\ell \in [L]$ such that
\[
 \sum_{\ell=1}^{L} \phi_\ell(x_\ell) \geq f(x), \quad \forall x = (x_1, \ldots, x_L) \in \mathcal{X}.
\] 
Then, 
\begin{align*}
 \sup_{\pi \in \Pi(\mu_1, \dotsc, \mu_L)} \int_{\mathcal{X}} f  d\pi   =
    \inf_{ (\phi_\ell)_{\ell \in [L]} \in \Phi_f  }
    \left\{\sum_{\ell=1}^{L} \int_{\mathcal{X}_\ell} \phi_\ell d \mu_\ell \right\}.
\end{align*}
\end{theorem}

\begin{proof}
    This theorem is a direct application of \citet[Proposition 2.3 and Theorem 2.14]{Kellerer:1984to} to Polish spaces. Since $\{x \in \mathcal{X}:f(x) \ge u \}$ is analytic for every $u \in \overline{\mathbb{R}}$ and $\mathcal{X}$ is a Polish space, it is a $\mathfrak{F}_{\mathcal{X} }$-Suslin set, where $\mathfrak{F}_{\mathcal{X} }$ is the collection of closed sets of $\mathcal{X}$. Therefore, conditions in \citet[Proposition 2.3 and Theorem 2.14]{Kellerer:1984to} are satisfied with the outer integral in the primal problem. 
    
    Since $\{x \in \mathcal{X}:f(x) \ge u \}$ is analytic for every $u \in \overline{\mathbb{R}}$ and $\mathcal{X}$ is Polish space, $f$ is universally measurable, see \citet[Theorem 12.41]{Aliprantis2006}) and \citet{Bertsekas1978}. For each $\pi \in \Pi(\mu_1, \dotsc, \mu_L)$, there exists a Borel measurable function $f^*$ such that $f^* = f$, $\pi$-almost surely. As a result, we can replace the outer integral by the integral with respect to $\pi$-completion using Lemma 1.2.1 of \citet{vandervaart1996}.
\end{proof}

For the function $\varphi$ defined in \Cref{appendix:interchangeability}, \citet[Proposition 7.47]{Bertsekas1978} implies that the set $\left\{t \in \mathcal{T}: \sup_{x \in \mathcal{X}} \varphi(t, x) \ge u\right\}$ is analytic for all $u\in \overline{\mathbb{R}}$.   In our context, the functions $f_\lambda: \mathcal{V} \rightarrow \mathbb{R}$ and $g_\lambda: \mathcal{V} \rightarrow \mathbb{R}$ may not be Borel measurable; however, $\left\{ f_{\lambda} \geq u \right\}$ and $\left\{ g_{\lambda}  \geq u \right\}$ are both analytic for all $u \in \overline{\mathbb{R} }$. In the following corollaries, we apply \Cref{thm:duality-OT} to the inner marginal problems in $\mathcal{J}_{\mathrm{D}}(\delta)$ and $\mathcal{J}(\delta)$, where we use the convention that the infimum over an empty set is defined as $\infty$.
\begin{corollary} \label{corollary:ID-dual-full}
        In addition to conditions in \Cref{thm:ID-duality}, assume that there exist some measurable functions $a_1 \in L^1(\mu_1)$ and $a_2 \in L^1(\mu_2)$ such that $a_1 < \infty$, $a_2 < \infty$, and
        \begin{align*}
            g(s_1, s_2) \ge a_1(s_1) + a_2(s_2), \quad \forall  (s_1, s_2) \in \mathcal{S}_1 \times \mathcal{S}_2.
        \end{align*}
        Then, for $\delta \in \mathbb{R}^2_{++}$, we have 
        \begin{align*}
            \mathcal{J}_{\mathrm{D}}(\delta) 
             = 
             \inf_{\substack{\lambda \in \mathbb{R}_{+}^2 \\ (\psi, \phi ) \in L^1(\mu_{1}) \times  L^1(\mu_{2})}  } 
            \Bigg\{ \langle \lambda, \delta \rangle & + \int_{\mathcal{S}_1} \psi d \mu_1 + \int_{\mathcal{S}_2} \phi d \mu_2 :  \psi, \phi > - \infty\\
            & \quad \psi(s_1) + \phi(s_2) \ge g_\lambda(s_1, s_2),  \Bigg\}.
        \end{align*}
    \end{corollary}

\begin{corollary} \label{corollary:I-dual-full}
In addition to conditions in \Cref{thm:I-duality}, assume that for each $\lambda$, there exist some measurable functions $a_{\lambda, 1} \in L^1(\mu_1)$ and $a_{\lambda, 2} \in L^1(\mu_2)$ such that $a_{\lambda, 1} < \infty, a_{\lambda, 2} < \infty$, and 
        \begin{align*}
            f_\lambda(s_1, s_2) \ge a_{\lambda, 1}(s_1) + a_{\lambda, 2}(s_2) , \quad \forall  (s_1, s_2) \in \mathcal{S}_1 \times \mathcal{S}_2.
        \end{align*}
        Then, for $\delta \in \mathbb{R}^2_{++}$, we have 
        \begin{align*}
            \mathcal{J}(\delta) 
            = 
            \inf_{\substack{\lambda \in \mathbb{R}_{+}^2 \\ (\psi, \phi ) \in L^1(\mu_{13}) \times  L^1(\mu_{23})}  } 
            \Bigg\{ \langle \lambda, \delta \rangle & +  \int_{\mathcal{S}_1} \psi d \mu_{13} + \int_{\mathcal{S}_2} \phi d \mu_{23}: \psi, \phi > - \infty  \\
            &\quad  \psi(s_1) + \phi(s_2) \ge f_\lambda(s_1, s_2) \ \forall (s_1, s_2)     \Bigg\}.
        \end{align*}
    \end{corollary}

\begin{corollary} \label{corollary:ID-dual-full-multi-marginals}
In addition to conditions in \Cref{thm:ID-duality-multi-marginals}, assume that there exist some measurable functions $a_\ell \in L^1(\mu_\ell)$ for $\ell \in [L]$ such that $a_\ell < \infty$, and
	\begin{align*}
		g(s) \ge \sum_{\ell=1}^{L} a_\ell(s_\ell), \quad \forall s = (s_1, \ldots, s_L) \in \prod_{\ell \in [L]} \mathcal{S}_\ell.
	\end{align*}
	Then, for $\delta \in \mathbb{R}^L_{++}$, we have 
	\begin{align*}
		\mathcal{J}_{\mathrm{D}}(\delta) 
		& = 
		 \inf_{\substack{ \lambda \in \mathbb{R}_{+}^L,  \ \psi_\ell > - \infty \\ (\psi_\ell)_{\ell \in [L]} \in \prod_{\ell \in [L]} L^1(\mu_\ell) }  } 
		\Bigg\{ \langle \lambda, \delta \rangle + \sum_{\ell=1}^L \int \psi_\ell d \mu_\ell : 
		 \sum_{\ell=1}^L \psi_\ell(s_\ell) \ge g_{\lambda, [L]}(s),  \forall s \Bigg \}.
	\end{align*}
\end{corollary}

\begin{corollary} \label{corollary:I-dual-full-multi-marginals}
In addition to conditions in \Cref{thm:I-duality-multi-marginals}, assume that for each $\lambda$, there exist some measurable functions $a_{\lambda, \ell} \in L^1(\mu_\ell)$ for $\ell \in [L]$ such that $a_{\lambda, \ell} < \infty$, and
	\begin{align*}
		f_{\lambda, [L]}(s) \ge \sum_{\ell=1}^{L} a_{\lambda, \ell}(s_\ell), \quad \forall s = (s_1, \ldots, s_L) \in \prod_{\ell \in [L]} \mathcal{S}_\ell.
	\end{align*}
Then, for $\delta \in \mathbb{R}^L_{++}$, we have 
\begin{align*}
    \mathcal{J}(\delta) 
    & = 
     \inf_{\substack{ \lambda \in \mathbb{R}_{+}^L,  \ \psi_\ell > - \infty \\ (\psi_\ell)_{\ell \in [L]} \in \prod_{\ell \in [L]} L^1(\mu_\ell) }  }
    \Bigg\{ \langle \lambda, \delta \rangle + \sum_{\ell=1}^L \int_{\mathcal{S}_\ell } \psi_\ell d \mu_{\ell,L+1} : 
     \sum_{\ell=1}^L \psi_\ell(s_\ell) \ge f_{\lambda, [L]}(s),  \forall s  \Bigg\}.
\end{align*}
\end{corollary}

\subsection{Probability Measures with Given Marginals}
\label{appendix:CPMS}

The existence of probability measures with given marginals was studied by \citet{vorob1962consistent}, \citet{kellerer1964verteilungsfunktionen}, and \citet{shortt1983combinatorial}. If the indices of the marginals are overlapping, then there may not be a probability measure compatible with the given marginals. 
In this section, we review a sufficient condition for the existence of such a measure.

We first define a {\it consistent product marginal system} (CPMS) by following \citet[p. 466]{shortt1983combinatorial}.
Let $\mathcal{S} = \prod_{j \in [n]} \mathcal{X}_j$.
Given a finite index collection $\{K_1,\ldots, K_N \}$ with $K_j \subset [n]$ and probability measure $\mu_j$ on $\mathcal{S}_j := \mathcal{S}_{K_j}$ for $j \in [N]$.  A {\it  product marginal system}  $\mathcal{F}\left(\mathcal{S}; (\mu_j)_{j=1}^N \right)$ consists of a product space $\mathcal{S}$  and probability measures $(\mu_j)_{j=1}^N$. 

\begin{definition}[Consistent product marginal system (CPMS)] \label{definition:CPMS}
The product marginal system $\mathcal{F}\left(\mathcal{S}; (\mu_j)_{j=1}^N \right)$ is said to be {\it consistent} if  for any $K_i, K_j \subset [n]$ with $K_i \cap K_j \neq \emptyset$, the projections of $\mu_i$ and $\mu_j$ on $\mathcal{S}_{K_i \cap K_j}$ are the same, i.e.,
\[
 \left(    \mathord{\operatorname{proj}}_{K_i \cap K_j}  \circ    { \operatorname{proj}_{K_i }  }^{-1} \right)\# \mu_i  =   \left(    \mathord{\operatorname{proj}}_{K_i \cap K_j}  \circ   { \operatorname{proj}_{K_j }  }^{-1} \right)   \# \mu_j  .
\] 
\end{definition}

A CPMS is not necessarily nonempty. To illustrate this, we consider the following examples.

\begin{example}
Let $\mathcal{S} = \mathcal{X}_1 \times \mathcal{X}_2$, $K_j = \{j\}$ for $j \in [2]$.  Given probability measures $\mu_j$ on $\mathcal{S}_j:= \mathcal{X}_j$ for $j \in [2]$, the CPMS $\mathcal{F} (\mathcal{S};\mu_1, \mu_2) $ is given by 
\[
\mathcal{F} (\mathcal{S};\mu_1, \mu_2) =  \left\{ \pi \in \mathcal{P}(\mathcal{X}_1 \times \mathcal{X}_2): \pi \circ  \operatorname{proj}_{\{j\}}^{-1}  = \mu_j, \ \forall j = 1, 2 \right\}.
\]
Obviously, $\mathcal{F} (\mathcal{S};\mu_1, \mu_2)$ is identical to $\Pi(\mu_1, \mu_2)$ and is nonempty.
\end{example}

\begin{example}
Let $\mathcal{X}_j= \mathbb{R}$ for $j\in [4]$. Let $K_j = \{j,j+1 \}$ and $\mathcal{S}_j = \mathcal{X}_{j}\times \mathcal{X}_{j+1}$ for $j \in [3]$. To make the example more concrete,  let $\mu_j = \mathcal{N}(0, I_2)$ for all $j \in [3]$. We note that
\[
\left( \mathord{\operatorname{proj}}_{K_j \cap K_{j+1}}   \circ { \operatorname{proj}_{K_j} }^{-1}  \right) \# \mu_j  =  \mathcal{N}(0,1), \quad \forall j \in [3].
\]
Moreover, it is easy to verify $\mathcal{F}\left(\mathcal{S}; (\mu_j)_{j=1}^3 \right)$ is consistent and nonempty, since $\mathcal{N}(0, I_3)$ is an element of $\mathcal{F}\left(\mathcal{S}; (\mu_j)_{j=1}^3 \right)$. 
\end{example}

\begin{example}
Let $\mathcal{X}_j= \mathbb{R}$ for $j\in [3]$, $K_1 = \{1,2 \}, K_2 = \{2,3 \}, K_3 = \{1,3\}$ and $\mathcal{S}_j := \mathcal{S}_{K_j}$ for $j \in [3]$. We define
\[
\mu_1 = \mathcal{N}\left( 0, \left[\begin{array}{cc}
2 & -1 \\
-1 & 4
\end{array}\right]  \right), \quad  \mu_2 = \mathcal{N}\left( 0, \left[\begin{array}{cc}
4 & -2 \\
-2 & 4
\end{array}\right]  \right), \quad  \mu_3 = \mathcal{N}\left( 0, \left[\begin{array}{cc}
2 & -2 \\
-2 & 4
\end{array}\right]  \right).
\]
It is easy to verify $\mathcal{F}\left(\mathcal{S}; (\mu_j)_{j=1}^3 \right)$ is consistent but is an empty set. Suppose $\pi \in \mathcal{F}\left(\mathcal{S}; (\mu_j)_{j=1}^3 \right) $, then the covariance matrix of $\pi$ is 
\[
\Sigma = \left[\begin{array}{ccc}
2 & -1 & -2 \\
-1& 4 & -2\\
-2 & -2 & 4 \\
\end{array}\right].
\]
However, $\Sigma$ is not positive semi-definite so can not be a covariance matrix.
\end{example}

A sufficient condition for a CPMS to be non-empty is the decomposability of its index set. We restate the definition of decomposibility from \citet[Definition 11]{fan2023VectorCopulas}, \citet[Section 3.7]{joe1997MultivariateModelsMultivariate}, and \citet{kellerer1964verteilungsfunktionen}.

\begin{definition}[Decomposability]
A collection $\{K_1, \ldots, K_N \}$ of subsets of $[n]$ is called decomposable if there is a permutation $\sigma$ of $[N]$ such that
\begin{equation} \label{eq:decomposability}
\tag{DC}
\left( \bigcup_{j< m} K_{\sigma(j)}  \right) \cap K_{\sigma(m)} \in  \bigcup_{j< m} 2^{ K_{\sigma(j) } }, \quad \  \forall m \in [N].
\end{equation}
\end{definition}

For Euclidean spaces, \citet{kellerer1964verteilungsfunktionen} proves that a CPMS is nonempty if its index set is decomposable, while \citet{shortt1983combinatorial} extends this result to separable spaces. Below, we present a statement of this result for Polish spaces and give a simple proof.

\begin{proposition}\label{prop:MultiVBP}
Let $\mathcal{S} = \Pi_{j \in [n]} \mathcal{X}_{j}$ where $\mathcal{X}_j$ are Polish spaces with the Borel algebras. 
Suppose that $\mathcal{F}\left(\mathcal{S} ;\left(\mu_j\right)_{j=1}^N\right)$ is a CPMS and the associated index collection $\{K_1, \ldots, K_N\}$ with $K_i \subset [n]$ is decomposable. Then $\mathcal{F}\left(\mathcal{S} ;\left(\mu_j\right)_{j=1}^N\right)$ is nonempty.
\end{proposition}

The proof of \cref{prop:MultiVBP} below is based on two results. The first is Theorem 1.1.10 in \citet{dudley2014uniform} restated in \cref{thm:VBP} and the second is \cref{Lemma1.10}, a direct consequence of \Cref{definition:CPMS}.

\begin{theorem}[Vorob'ev-Berkes-Philip]\label{thm:VBP}
Let $\mathcal{Y}_1, \mathcal{Y}_2, \mathcal{X}$ be Polish spaces with Borel algebras and let $\mathcal{S}:= \mathcal{Y}_1\times \mathcal{Y}_2\times \mathcal{X}$. Let $\mu_0$ and $\mu_1$ be Laws on $\mathcal{S}_1 := \mathcal{Y}_1 \times \mathcal{X}$ and $\mathcal{S}_2 := \mathcal{Y}_2 \times \mathcal{X}$ respectively.  Suppose $\mathcal{F}\left(\mathcal{S}; \mu_1, \mu_2 \right)$ is a consistent product marginal system. Then $\mathcal{F}\left(\mathcal{S}; \mu_1, \mu_2 \right)$ is nonempty.
\end{theorem}

\begin{lemma} \label{Lemma1.10}
Suppose that $\mathcal{F}\left(\mathcal{S}; (\mu_j)_{j=1}^N \right)$ is a CPMS, $K_i, K_j \subset [n]$ and $K_i \cap K_j \neq \emptyset$. If $Q\subset K_i \cap K_j$ and $Q\neq \emptyset$, then  the projections of $\mu_i$ and $\mu_j$ on $\mathcal{S}_{Q}$ are the same, i.e.,
\[
\left(   \mathord{\operatorname{proj}}_{Q} \circ  {\operatorname{proj}_{K_i} }^{-1}  \right) \# \mu_i = \left(   \mathord{\operatorname{proj}}_{Q} \circ  {\operatorname{proj}_{K_j} }^{-1}  \right) \# \mu_j.
\] 
Moreover, for all $\pi \in \mathcal{F}\left(\mathcal{S}; (\mu_j)_{j=1}^N \right)$,
\[
\mathrm{proj}_{Q}  \#\pi  = \left(   \mathord{\operatorname{proj}}_{Q}  \circ  { \operatorname{proj}_{K_j } }^{-1} \right)\# \mu_j, \quad \forall j \in [N].
\]
\end{lemma}

\begin{proof}[Proof of \Cref{prop:MultiVBP}]
We give a proof by induction on $N$. Without loss of generality, assume that the permutation $\sigma$ in \Cref{eq:decomposability} satisfies $\sigma(j) = j$ for $j \in [N]$. 
\Cref{prop:MultiVBP} holds trivially when $N=1$. When $N=2$, it holds by \Cref{thm:VBP}. Let $\mathcal{H}_{N-1} := \prod_{j=1}^{N-1} \mathcal{S}_j$ and assume that $\mathcal{F}\left(\mathcal{H}_{N-1} ;  (\mu_j)_{j=1}^{N-1} \right)\neq \emptyset$. Then, there is a $\gamma \in \mathcal{F}\left( \mathcal{H}_{N-1} ;\left(\mu_j\right)_{j=1}^{N-1}\right)$. Let us verify that $\mathcal{F}( \mathcal{H}_{N-1} \times \mathcal{S}_N; \gamma, \mu_N )$ is consistent.

Let $Q = \cup_{j=1}^{N-1} K_j$.
Since $\{K_1,\ldots, K_N\}$ is decomposable, $Q \cap K_N \in \cup_{j<N} 2^{K_j}$. As a result, we must have $(Q  \cap K_N ) \subset K_\ell$ for some $\ell \in [N-1]$ and hence $(Q  \cap K_N ) \subset (K_\ell \cap K_N)$. If $\left(Q \cap K_N\right) = \emptyset$, the proof is trivial. In the rest of the proof, we suppose $\left(Q \cap K_N\right) \neq \emptyset$. Since $\mathcal{F}\left(\mathcal{S}; (\mu_j)_{j=1}^N \right)$ is consistent,  by \Cref{Lemma1.10},
\[
\begin{aligned}
 \left(  \operatorname{proj}_{K_N \cap Q}  \circ {\operatorname{proj}_{K_N}}^{-1}  \right) \# \mu_N & =   \left( \operatorname{proj}_{K_N \cap Q} 
 \circ  {\operatorname{proj}_{K_\ell} }^{-1}\right) \# \mu_\ell.
\end{aligned}
\]
Since $\mathcal{F}\left(\mathcal{H}_{N-1} ;\left(\mu_j\right)_{j=1}^{N-1}\right)$ is consistent,  \Cref{Lemma1.10} also implies 
\[
\left(  \mathord{\operatorname{proj}}_{K_N \cap Q}  \circ {\mathrm{proj}_{Q} }^{-1} \right)  \# \gamma =  \left( \mathord{\operatorname{proj}}_{K_N \cap Q} \circ { \mathrm{proj}_{K_\ell} } ^{-1} \right) \# \mu_\ell.
\]
This shows
\[
\left(\mathord{\operatorname{proj}}_{K_N \cap Q}  \circ { \operatorname{proj}_{K_N}  }^{-1} \right) \# \mu_N = \left(  \mathord{\operatorname{proj}}_{K_N \cap Q}  \circ {\mathrm{proj}_{Q} }^{-1}  \right) \# \gamma,
\]
and $\mathcal{F}( \mathcal{H}_{N-1} \times \mathcal{S}_N; \gamma, \mu_N )$ is consistent. The proof is complete by using \Cref{thm:VBP} again.

\end{proof}

\section{Appendix B: Technical Lemmas}

\begin{lemma}\label{lemma:ID-I-concavity} \leavevmode
\begin{lemmaenum}
    \item \label{lemma:ID-concavity}
    Suppose \Cref{assumption:cost-function,assumption:g-bounded-below} hold. Then, the function $\mathcal{I}_{\mathrm{D} }(\delta)$ is concave, non-decreasing in $\delta\in \mathbb{R}_{+}^2$, and $\mathcal{I}_{\mathrm{D} }(\delta) > -\infty$ for all $\delta \in \mathbb{R}_{+}^2$.

    \item \label{lemma:I-concavity}    
    Suppose \Cref{assumption:cost-function,assumption:f-bounded-below} hold. Then, the function $\mathcal{I}(\delta)$ is concave, non-decreasing in $\delta\in \mathbb{R}_{+}^2$, and $\mathcal{I}(\delta) > -\infty$ for all $\delta \in \mathbb{R}_{+}^2$.
\end{lemmaenum}
\end{lemma}
\begin{proof}[Proof of \Cref{lemma:ID-I-concavity}]
We show the claims on $\mathcal{I}(\delta)$ only since the proof for $\mathcal{I}_{\mathrm{D}}(\delta)$ is almost identical to that for $\mathcal{I}(\delta)$. Note that $\mathcal{I}(\delta)$ is well-defined since \Cref{assumption:cost-function} implies that $\Sigma(\delta)$ is non-empty. 

Note that under \Cref{assumption:g-bounded-below}, for any $\delta \in \mathbb{R}_{+}^2$, 
\begin{equation}
	\mathcal{I}(\delta) \ge \mathcal{I}(0) \ge \int_{\mathcal{S} } f(s) d \nu(s) > - \infty
\end{equation}
for some $\nu \in \mathcal{F}(\mu_1, \mu_2)$. 
The monotonicity of $\mathcal{I}$ can be seen from the definition. We now show the concavity of $\mathcal{I}$.  Fix $\delta = (\delta_1, \delta_2) \in \mathbb{R}^2_+$, $\delta^\prime = (\delta_1^\prime, \delta_2^\prime)\in \mathbb{R}_+^2$ and $\lambda \in (0,1)$. For any $\gamma \in\Sigma(\delta),\gamma^\prime \in \Sigma(\delta^\prime)$, consider the probability measure $\gamma^{\prime \prime} = \lambda \gamma +(1-\lambda) \gamma^\prime $. Since $\boldsymbol{K}_\ell$ is Optimal Transport cost, $\nu \mapsto \boldsymbol{K}_\ell (\mu_\ell, \nu)$ is convex. So,
we have for $\ell =1,2$,
\[
\begin{aligned}
\boldsymbol{K}_\ell\left(\mu_\ell,\gamma^{\prime \prime}_{\ell,3} \right) &  \leq \lambda \boldsymbol{K}_\ell( \mu_\ell,  \gamma_\ell) + (1-\lambda) \boldsymbol{K}_\ell \left(  \mu_\ell,\gamma_\ell^\prime \right)   \leq  \lambda \delta_\ell + (1- \lambda) \delta_\ell^\prime.
\end{aligned}
\]
This shows that $\gamma^{\prime \prime} \in \Sigma\left( \lambda \delta + (1-\lambda) \delta^\prime \right)$ and hence
\[
\begin{aligned}
\mathcal{I}(\lambda \delta + (1-\lambda) \delta^\prime ) & =  \sup_{\nu \in   \Sigma\left( \lambda \delta + (1-\lambda) \delta^\prime \right)}  \int_{\mathcal{S} } f(s) d\nu(s) \\
& \geq   \int_{\mathcal{S} } f d \gamma^{\prime \prime} = \lambda  \int_{\mathcal{S} } f d  \gamma  + (1-\lambda)   \int_{\mathcal{S} } f  d \gamma^{\prime }.
\end{aligned}
\]
Taking the supremum over $\gamma \in \Sigma(\delta)$ and $\gamma^\prime \in \Sigma(\delta^\prime)$ yields 
\[
\begin{aligned}
\mathcal{I}\left(\lambda \delta + (1-\lambda) \delta^\prime  \right) &\geq  \lambda \sup_{\gamma \in \Sigma(\delta)} \int_{\mathcal{S} } f(s) d  \gamma(s)   + (1-\lambda) \sup_{\gamma^\prime \in \Sigma(\delta^\prime)}\int_{\mathcal{S} } f(s)  d \gamma^{\prime }(s)  \\
& \geq \lambda \mathcal{I}(\delta) +(1-\lambda)   \mathcal{I}(\delta^\prime).
\end{aligned}
\]
\end{proof}

\begin{lemma} \label{lemma: Legendre_transform}
Let $\varphi: \mathbb{R}^n_{+} \rightarrow \mathbb{R} \cup \{\infty\}$ be a concave and  non-decreasing function. For all $\lambda \in \mathbb{R}^n_+$, define 
\[
\varphi^\star(\lambda) = \sup_{x \in \mathbb{R}^n_+ }\left \{\varphi(x) - \langle \lambda,x \rangle \right\}.
\]
Then for all $x \in \mathbb{R}^n_{++}$, one has
\[
\varphi(x) = \inf_{\lambda \in \mathbb{R}^n_+ } \left\{ \langle \lambda, x \rangle  + \varphi^\star(\lambda) \right \}.
\]
\end{lemma}

\begin{proof}[Proof of \Cref{lemma: Legendre_transform}]
If $\varphi(x_0) = \infty$ for some $x_0 \in \mathbb{R}^n_{++}$, then $\varphi(x) = \infty$ for all $x \in \mathbb{R}^n_{++}$. In fact, for any $x \in \mathbb{R}^n_{++}$, there is $x_1 \in B(x, \delta)$ such that $x = t x_0 + (1- t) x_1$ for some $t \in (0,1)$ and the concavity of $\varphi$ implies
\[
\varphi(x) = \varphi(t x_0 + (1- t) x_1 ) \geq t \varphi(x_0) + (1-t) \varphi(x_1) = \infty.
\]

Now we assume $\varphi(x) < \infty$ for all $x \in \mathbb{R}_{++}^n$. 
Define a new function $\psi: \mathbb{R}^n \rightarrow \mathbb{R} \cup \{ \infty\}$ as
\[
\psi(x):= \begin{cases} - \varphi(x) &  x \in \mathbb{R}^n_+ \\ \infty &   x \notin   \mathbb{R}^n_+  .\end{cases}
\]
It is easy to see $\psi$ is convex and the Legendre–Fenchel transform of $\psi$ is given by 
\[
\begin{aligned}
\psi^\star(\lambda)&  = \sup _{x \in \mathbb{R}^n} \left\{\langle\lambda, x\rangle - \psi(x)\right \} = \sup _{x \in \mathbb{R}^n_{+}}\{ \langle\lambda, x\rangle - \psi(x)   \} \\
&   =  \sup_{x \in \mathbb{R}^n_{+} }  \{ \varphi(x)  -  \langle -\lambda, x\rangle \}  =  \begin{cases}  \varphi^\star (- \lambda) &  - \lambda \in  \mathbb{R}^n_+  \\ \infty &   -\lambda \notin  \mathbb{R}^n_+    \end{cases}.
\end{aligned}
\]
The Legendre–Fenchel transform of $\psi^\star(\lambda)$ is given by 
\[
\begin{aligned}
\psi^{\star \star}(x)  & = \sup_{\lambda \in \mathbb{R}^n } \{\langle\lambda, x\rangle-\psi^\star(\lambda )\} = \sup_{ -\lambda \in \mathbb{R}_{+}^n } \{\langle\lambda, x\rangle-\psi^\star(\lambda )\} \\
& =  \sup_{ -\lambda \in \mathbb{R}_{+}^n } \{\langle\lambda, x\rangle- \varphi^{\star}(-\lambda)  \} = - \inf_{  \lambda \in \mathbb{R}_{+}^n   } \{\langle \lambda, x\rangle +  \varphi^{\star}(\lambda)  \} 
\end{aligned}
\]
Since $\psi^{\star \star}$ is the double Legendre–Fenchel transform of  $\psi$, then $\psi^{\star \star}$ is the lower-semicontinuous convex envelope of $\psi$ from below. The convexity of $\psi$ implies $\psi = \psi^{\star\star}$ in the interior of $\{x : \psi(x) < \infty \} $ which is $\mathbb{R}^n_{++}$. The desired result follows.
\end{proof}

\begin{lemma}\label{lemma:decomposability-V}
Let $K := \left\{K_1, K_2, K_3\right\}$, where $K_1 = \{3,4\}$, $K_2 = \{1,3\}$, and $K_3 = \{2,4\}$. Then $K$ is decomposable. 
\end{lemma}

\begin{proof}[Proof of \protect\Cref{lemma:decomposability-V}]
When $m = 1$, The condition \eqref{eq:decomposability} holds obviously. When $m=2$,
	\[
	\left(\bigcup_{\ell<2} K_{\ell }\right) \cap K_{2} = K_1 \cap K_2 = \{3\} \in \bigcup_{\ell<2} 2^{K_{\ell}} = 2^{K_1}.
	\]
When $m = 3$,
	\[
	\left(\bigcup_{\ell<3} K_{\ell}\right) \cap S_{3} = (K_1\cup K_2) \cap K_3 = \{4\} \in \bigcup_{\ell<3} 2^{K_{\ell}} = 2^{K_1} \cup 2^{K_2}.
	\]
\end{proof}

\begin{lemma}\label{lemma:decomposability-S}
 Let $K := \left\{K_1, K_2, K_3\right\}$ where $K_1 = \{3, 4, 5\}$, $K_2 = \{1, 3, 5\}$, and $K_3 = \{2, 4, 5\}$. Then $K$ is decomposable. 
\end{lemma}

\begin{proof}[Proof of \protect\Cref{lemma:decomposability-S}]
When $m = 1$, the condition \eqref{eq:decomposability} holds trivially. When $m=2$,
\[
\left(\bigcup_{\ell<2} K_{\ell }\right) \cap K_{2} = K_1 \cap K_2 = \{3, 5\} \in \bigcup_{\ell<2} 2^{K_{\ell}} = 2^{K_1}.
\]
When $m = 3$,
\[
\left(\bigcup_{\ell<3} K_{\ell}\right) \cap K_{3} = (K_1\cup K_2) \cap K_3 = \{4, 5\} \in \bigcup_{\ell<3} 2^{K_{\ell}} = 2^{K_1} \cup 2^{K_2}.
\]
\end{proof}

\begin{lemma}\label{lemma:decomposability-V-multi-marginals}
	Let $K := \left\{K_1, \dotsc, K_{L+1}\right\}$ where $K_1 = \{L+1, \dotsc, 2L\}$ and $K_\ell = \{\ell -1 , L+\ell-1 \}$ for $2 \leq \ell \leq L+1$. Then $K$ is decomposable. 
	\end{lemma}
	
	\begin{proof}[Proof of \protect\Cref{lemma:decomposability-V-multi-marginals}]
When $m = 1$, the condition \eqref{eq:decomposability} holds trivially. When $1 < m \le L+1$,
\[
\left(\bigcup_{\ell<m} K_{\ell }\right) \cap K_{m} = \bigcup_{\ell<m} \left(K_{\ell }\cap K_{m}\right)= K_{1} \cap K_{m} \in 2^{K_{1}} \subset \bigcup_{\ell<m} 2^{K_{\ell}}.
\]		
This shows that the condition \eqref{eq:decomposability} holds.
\end{proof}

\begin{lemma}\label{lemma:decomposability-S-multi-marginals}
	Let $K := \left\{K_1, \dotsc, K_{L+1}\right\}$, where $K_1 = \{L+1, \dotsc, 2L+1\}$ and $K_{\ell+1} = \{ \ell , L + \ell , 2L+1\}$ for $1 \leq \ell \leq L$. Then $K$ is decomposable. 
   \end{lemma}
   
   \begin{proof}[Proof of \protect\Cref{lemma:decomposability-S-multi-marginals}]
When $m = 1$, the condition \eqref{eq:decomposability} holds trivially. 
   When $1 < m \le L+1$,
   \[
   \left(\bigcup_{\ell<m} K_{\ell }\right) \cap K_{m} = \bigcup_{\ell<m} (K_{\ell } \cap K_{m} ) = (K_{1} \cap K_{m})  \in 2^{K_{1}} \subset \bigcup_{\ell<m} 2^{K_{\ell}}.
   \]
   This shows that the condition \eqref{eq:decomposability} holds.
   \end{proof}

\section{Appendix C: Proofs of Main Results}
\subsection{Proofs in \texorpdfstring{\Cref{sec:WDRO-main-theory}}{}}

\subsubsection{Proof of \texorpdfstring{\Cref{thm:ID-duality}}{}}

The expressions of $\mathcal{I}_{\mathrm{D}}(\delta_1, 0)$ and $\mathcal{I}_{\mathrm{D}}(0,\delta_2)$ can be derived from $\mathcal{I}_{\mathrm{D}}(\delta_1, \delta_2)$ for $\delta_1, \delta_2 >0$ with appropriate modifications of the cost function. In particular, consider another cost function $\widehat{c}_2(s_2, s_2^\prime ) =  \infty \mathds{1} \{ s_2 \neq s_2^\prime \}$ and the optimal transport distance $\widehat{\boldsymbol{K}}_2$ associated with $\widehat{c}_2$. Define an uncertainty set $\widehat{\Sigma}_{\mathrm{D}}(\delta_1, \delta_2)$ depending on $\boldsymbol{K}_1$ and  $\widehat{\boldsymbol{K}}_2$ as 
\[
\widehat{\Sigma}_{\mathrm{D}}(\delta_1, \delta_2) =  \left\{ \gamma \in \mathcal{P}(\mathcal{S}_1 \times \mathcal{S}_2): \boldsymbol{K}_1( \gamma_1, \mu_1) \leq \delta_1, \widehat{\boldsymbol{K}}_2( \gamma_2, \mu_2) \leq \delta_2 \right\}.
\]
Moreover, we define $\widehat{\mathcal{I}}_{\mathrm{D}}: \mathbb{R}^2_+ \rightarrow \mathbb{R}$ as
\[
\widehat{\mathcal{I}}_{\mathrm{D}} (\delta_1, \delta_2) = \sup_{\gamma \in \widehat{\Sigma}_{\mathrm{D}} (\delta_1, \delta_2)} \int_{\mathcal{V} } g(s_1,s_2) \, d\gamma(s_1,s_2).
\]
We note $\widehat{\boldsymbol{K}}_2(\mu, \nu) = 0$ if and only if $\mu = \nu$. For all $\delta_2 > 0$, $\widehat{\Sigma}_{\mathrm{D}}(\delta_1, \delta_2) = \Sigma_{\mathrm{D}}(\delta_1, 0)$ and $\widehat{\mathcal{I}}_{\mathrm{D}} (\delta_1, \delta_2) =  \mathcal{I}_{\mathrm{D}}(\delta_1, 0)$. Using the dual reformulation of $\widehat{\mathcal{I}}_{\mathrm{D}}$ on $\mathbb{R}^2_{++}$, we have
\[
\mathcal{I}_{\mathrm{D}}(\delta_1, 0) = \widehat{\mathcal{I}}_{\mathrm{D}} (\delta_1, \delta_2) = \inf _{\lambda \in \mathbb{R}_{+}^2 }\left[ \langle\lambda, \delta \rangle+\sup _{\varpi \in \Pi\left(\mu_1, \mu_2\right)} \int_{\mathcal{V}} g_\lambda \left(s_1, s_2\right) d \varpi(s_1,s_2)\right],
\]
where 
\[
\begin{aligned}
g_\lambda(s_1, s_2) & = \sup _{s_1^{\prime} \in \mathcal{S}_1, s_2^\prime \in \mathcal{S}_2 }\left\{g\left(s_1^{\prime}, s^\prime_2\right)-\lambda_1 c_1 (s_1, s_1^{\prime})  - \lambda_2  \widehat{c}_2 (s_2, s_2^\prime)   \right\} \\
& =  \sup _{s_1^{\prime} \in \mathcal{S}_1 }\left\{g\left(s_1^{\prime}, s_2\right)-\lambda_1 c_1 (s_1, s_1^{\prime})  \right\} = g_{\lambda, 1}\left(s_1, s_2\right).
\end{aligned}
\]
Since $g_{\lambda, 1}\left(s_1, s_2\right)$ is independent of $\lambda_2$, letting $\lambda_2 = 0$ yields
\[
\mathcal{I}_{\mathrm{D}}(\delta_1, 0) = \inf _{\lambda_1 \in \mathbb{R}_{+}}\left[\lambda_1 \delta_1+\sup _{\varpi \in \Pi\left(\mu_1, \mu_2\right)} \int_{\mathcal{V}} g_{\lambda, 1}(v) \, d \varpi(v)\right].
\]
Using the same reasoning, we can get the expression of $\mathcal{I}_{\mathrm{D}}(0,\delta_2)$. 

In the rest of the proof, we show the dual reformulation of $\mathcal{I}_{\mathrm{D}}$ on $\mathbb{R}^2_{++}$. Let $\mathcal{P}_{\mathrm{D}}$ denote the set of $\gamma \in \mathcal{P}(\mathcal{V})$ that satisfies $\boldsymbol{K}_1(\mu_1, \gamma_1)< \infty$, $\boldsymbol{K}_2(\mu_2, \gamma_2) < \infty$, and $\int_\mathcal{V} g d\gamma > -\infty$. Taking the Legendre transform on $\mathcal{I}$ yields that any $\lambda\in \mathbb{R}_{++}^2$, 
\[
\begin{aligned}
\mathcal{I}_{\mathrm{D} }^\star (\lambda) :=&\sup_{\delta \in \mathbb{R}_{+}^2 }  \left\{ \mathcal{I}_{\mathrm{D} }(\delta)  - \langle \lambda, \delta \rangle \right\}  =\sup_{\delta \in \mathbb{R}_{+}^2 } \sup_{\gamma \in \Sigma(\delta)  } \left\{ \int_\mathcal{V} g d\gamma - \langle \lambda, \delta \rangle \right\}\\
=&  \sup_{\delta \in \mathbb{R}_{+}^2 }    \sup_{\gamma \in \mathcal{P}(\mathcal{V} ) }  \left\{ \int_\mathcal{V} g d\gamma - \langle \lambda, \delta \rangle:   \boldsymbol{K}_\ell( \mu_\ell,\gamma_\ell) \leq \delta_\ell, \forall \ell \in [2] \right\} \\
=&      \sup_{\gamma \in \mathcal{P}(\mathcal{V} ) }  \sup_{\delta \in \mathbb{R}_{+}^2 }  \left\{ \int_\mathcal{V} g d\gamma - \langle \lambda, \delta \rangle:  \boldsymbol{K}_\ell( \mu_\ell,\gamma_\ell) \leq \delta_\ell, \forall \ell \in [2]  \right\} \\
= &   \sup_{\gamma \in \mathcal{P}_{\mathrm{D}} }  \underbrace{\left\{ \int_\mathcal{V} g d\gamma - \lambda_1 \boldsymbol{K}_1(\mu_1,\gamma_1)- \lambda_2 \boldsymbol{K}_2(\mu_2,\gamma_2 ) \right\}}_{:= I_{\mathrm{D}, \lambda}[\gamma]} = \sup_{\gamma \in \mathcal{P}_{\mathrm{D}} } I_{\mathrm{D}, \lambda}[\gamma].
\end{aligned}
\]

We note that the expression above also holds for $\lambda \in \mathbb{R}^2_+ \setminus  \mathbb{R}^2_{++}$.   Let $\mathcal{G}_{\mathrm{D}, \lambda}$ be the set of all probability measures $\pi$ on $ \mathcal{V} \times \mathcal{V}$ such that $\int_{\mathcal{V} \times \mathcal{V}} \varphi_\lambda  d \pi$ is well-defined and the first and second marginals are $\mu_1$ and $\mu_2$.\footnote{To be more precise, $\pi( (A_1 \times \mathcal{S}_2) \times \mathcal{V}) =  \mu_1(A_1)$ and $\pi( (\mathcal{S}_1 \times A_2)  \times \mathcal{V} ) =  \mu_2(A_2)$ for all sets $A_1 \in \mathcal{B}_{\mathcal{S}_1 }$ and $A_2 \in\mathcal{B}_{\mathcal{S}_2 } $. }   \Cref{lemma:C3} implies $\mathcal{I}_{\mathrm{D} }^\star (\lambda) = \sup _{\pi \in  \mathcal{G}_{\mathrm{D}, \lambda} } \int_{\mathcal{V} \times \mathcal{V}} \varphi_\lambda d \pi$. By \Cref{lemma:C4}, we have for all $\lambda \in \mathbb{R}^2_{+}$,
\begin{align*}
\mathcal{I}_{\mathrm{D} }^{\star}(\lambda)= \sup _{\pi \in \mathcal{G}_{\mathrm{D}, \lambda}} \int_{\mathcal{V} \times \mathcal{V}} \varphi_\lambda d \pi =
\sup_{ \pi \in \bar{ \Gamma} }  \int_{\mathcal{V} \times  \mathcal{V} } \varphi_{\lambda}  d \pi.
\end{align*}
where we write $\bar{\Gamma}= \Gamma\left( \Pi(\mu_1, \mu_2) ,  \varphi_\lambda \right)$ for simplicity. From \Cref{lemma:ID-concavity}, $\mathcal{I}_{\mathrm{D} }$ is bounded from below, non-decreasing, and concave. As a result,  $\mathcal{I}_{\mathrm{D} }< \infty$  or $\mathcal{I}_{\mathrm{D} } = \infty$ on $\delta \in \mathbb{R}_+^2$. 
In the first case, by \Cref{lemma: Legendre_transform}, for all $\delta \in \mathbb{R}^2_+$,
\begin{align*}
\mathcal{I}_{\mathrm{D} }(\delta) = \inf_{\lambda \in \mathbb{R}^2_+}  \left\{ \langle \lambda, \delta \rangle +   \sup _{\pi \in \bar{ \Gamma} } \int_{\mathcal{V} \times \mathcal{V}} \varphi_\lambda \, d \pi \right\} .
\end{align*}
In the second case, by definition $\mathcal{I}_{\mathrm{D} }^{\star}(\lambda) = \infty$ for all $\lambda \in \mathbb{R}_{+}^2$ and the above is also true. Moreover, Example 2 of \citet{zhang2022simple} implies that $\varphi_\lambda$ satisfies the interchangeability principle with respect to $\Pi(\mu_1, \mu_2)$. So \Cref{IPforGroup} implies that for all $\lambda \in \mathbb{R}_{++}^2$,
\begin{align*}
\sup_{ \pi \in \bar{ \Gamma} }  \int_{\mathcal{V} \times  \mathcal{V} } \varphi_{\lambda} \,  d \pi  =
 \sup _{\gamma \in \Pi\left(\mu_1, \mu_2 \right)} \int_{\mathcal{V}} g_\lambda(v) \, d \gamma(v),
\end{align*}
where $g_\lambda(v) = \sup _{v^{\prime} \in \mathcal{V}} \varphi_\lambda\left(v, v^{\prime}\right)$. This shows for all $\delta \in \mathbb{R}^2_{++}$,
\begin{align*}
\mathcal{I}_{\mathrm{D} }(\delta) = \inf_{\lambda \in \mathbb{R}^2_+}  \left\{ \langle \lambda, \delta \rangle +   \sup _{\gamma \in \Pi\left(\mu_1, \mu_2 \right)} \int_{\mathcal{V}} g_\lambda \, d \gamma  \right\} .
\end{align*}

\begin{lemma}  \label{lemma:C1}
If $\lambda_1>0$ and $\lambda_2 >0$, then 
\begin{align*}
\sup_{\gamma \in \mathcal{P}_{\mathrm{D}}  } I_{\mathrm{D}, \lambda}[\gamma] = \sup_{\gamma \in \mathcal{P}_{\mathrm{D}} }  \sup_{\pi \in  \Pi(\mu_1, \mu_2, \gamma)}  \int_{\mathcal{V} \times \mathcal{V}} \varphi_\lambda \, d \pi.
\end{align*}
\end{lemma}

\begin{proof}[Proof of \protect\Cref{lemma:C1}]
Fix any $\epsilon >0$ and $\gamma \in \mathcal{P}_{\mathrm{D}}$. By the definition of $\mathcal{P}_{\mathrm{D}}$, we have $\boldsymbol{K}_\ell (\mu_\ell, \gamma_\ell)< \infty$ and hence there is $\nu_\ell \in \Pi(\mu_\ell, \gamma_\ell)$ such that $\boldsymbol{K}_\ell (\mu_\ell, \gamma_\ell) \geq \int_{\mathcal{S}_\ell \times \mathcal{S}_\ell}  c_\ell \, d\nu_\ell - \epsilon/(\lambda_1 + \lambda_2)$. Let $K = \{K_1, K_2, K_3 \}$ with $K_1 =\{ 1,3 \}$, $K_2=\{2, 4\}$ and $K_3=\{3, 4\}$. Since $K$ is decomposable, then by \Cref{prop:MultiVBP} there is a measure $\widetilde{\pi}$ on $\mathcal{S}_1 \times \mathcal{S}_2 \times \mathcal{S}_1 \times \mathcal{S}_2$ with marginals given by $\pi_{1,3} = \nu_1$,  $\pi_{2,4} = \nu_2$ and $\pi_{3,4} = \gamma$.   Moreover, we note $\int_{\mathcal{V} \times \mathcal{V} } c_\ell(s_\ell, s_\ell^\prime) \, d \widetilde{\pi} = \int_{\mathcal{S}_{\ell} \times \mathcal{S}_{\ell}} c_{\ell} \, d \nu_{\ell}  \leq  \boldsymbol{K}_{\ell}(\mu_{\ell}, \gamma_{\ell}) + \epsilon/(\lambda_1 + \lambda_2) < \infty$. Now, we show the LHS is not bigger than the RHS. When $I_{\mathrm{D}, \lambda}[\gamma]  =\infty$, provided $ \boldsymbol{K}_\ell(\mu_\ell, \gamma_\ell) \in (0, \infty)$ for $\ell=1,2$, we must have $\int_{\mathcal{V} } g \, d\gamma = \infty$.  Then, it is apparent that $\int \varphi_\lambda \, d \widetilde{\pi} = \infty$ and hence  $I_{\mathrm{D},\lambda}[\gamma]  \leq \int \varphi_\lambda d \widetilde{\pi}  + \epsilon$. When $I_{\mathrm{D}, \lambda}[\gamma]  < \infty$, then $\int_{\mathcal{V}} g \, d \gamma < \infty$. Therefore, the integral given by 
\begin{align*}
\int_{\mathcal{V} \times \mathcal{V} } \varphi_\lambda \, d \widetilde{\pi} = \int_{\mathcal{V}} g \, d \gamma - \int_{\mathcal{S}_{1} \times \mathcal{S}_{1}  } \lambda_1 c_{1} \, d \nu_{1}-\int_{\mathcal{S}_{2} \times \mathcal{S}_{2}  }  \lambda_2 c_{2} \, d \nu_{2} < \infty,
\end{align*}
is well-defined. The desired result follows from the estimate below
\begin{align*}
\int_{\mathcal{V} \times \mathcal{V} } \varphi_\lambda d \widetilde{\pi} \geq   \int_{\mathcal{V}} g \, d \gamma - \lambda_1 \boldsymbol{K}_{1} (\mu_1, \gamma_1) - \lambda_2 \boldsymbol{K}_{2} (\mu_2, \gamma_2) -  \epsilon = I_{\mathrm{D}, \lambda}[\gamma] - \epsilon.
\end{align*}
Therefore, we have $I_{\mathrm{D}, \lambda}[\gamma] \leq \int_{\mathcal{V} \times \mathcal{V}} \varphi_\lambda d \widetilde{\pi}+\epsilon$. Since $\epsilon >0$ and $\gamma\in \mathcal{P}_{\mathrm{D}}$ are arbitrary, we have 
\begin{align*}
\sup _{\gamma \in \mathcal{P}_{\mathrm{D}}} I_{\mathrm{D}, \lambda}[\gamma] \leq \sup _{\gamma \in \mathcal{P}_{\mathrm{D}}} \sup _{\pi \in \Pi\left(\mu_1, \mu_2, \gamma\right)} \int_{\mathcal{V} \times \mathcal{V}} \varphi_\lambda \, d \pi.
\end{align*}

Next, we prove that the reversed direction holds by showing that if $\gamma \in \mathcal{P}_{\mathrm{D}}$, then $I_{\mathrm{D}, \lambda}[\gamma] \geq  \sup_{\pi \in \Pi(\mu_1, \mu_2, \gamma)}  \int_{\mathcal{V} \times \mathcal{V}} \varphi_\lambda \, d \pi$. Fix $\gamma \in \mathcal{P}_{\mathrm{D}}$. When $\int_{\mathcal{V} } g \, d \gamma = \infty$, $I_{\mathrm{D}, \lambda}[\gamma] =  \infty $ and then the proof is done.  Next, when $\int_{\mathcal{V} } g \, d \gamma < \infty$, for any $\pi \in \Pi(\mu_1, \mu_2, \gamma)$ such that $\int \varphi_\lambda \, d \pi$ is well-defined, 
\begin{align*}
I_{\mathrm{D}, \lambda}[\gamma]  & =  \int_{\mathcal{V} } g \, d \gamma -  \lambda_1 \boldsymbol{K}_{1} (\mu_{1}, \gamma_{1} ) - \lambda_2 \boldsymbol{K}_{2} (\mu_{2}, \gamma_{2} ) \\
& \geq   \int_{\mathcal{V} } g(s_1^\prime, s_2^\prime) \, d \pi_{3,4} -  \lambda_1    \int_{\mathcal{S}_1 \times \mathcal{S}_1 }  c_1(s_1, s_1^\prime) \, d \pi_{1, 3}  -  \lambda_2 \int_{\mathcal{S}_2 \times \mathcal{S}_2 }    c_2(s_2, s_2^\prime) \, d \pi_{2, 4}  \\
& = \int_{\mathcal{V} \times \mathcal{V}} \varphi_\lambda \, d \pi.
\end{align*}
With the convention that $\sup \empty = - \infty$, if the integral $\int \varphi_\lambda \, d \pi$ is not well-defined for all $\pi \in \Pi(\mu_1, \mu_2, \gamma)$,  then $I_{\mathrm{D}, \lambda}[\gamma] \geq \sup_{\pi \in  \Pi(\mu_1, \mu_2, \gamma)} \int_{\mathcal{V} \times \mathcal{V}} \varphi_\lambda \, d \pi$ holds trivially. Otherwise, taking the supremum over $\pi \in \Pi(\mu_1, \mu_2, \gamma)$ on the RHS of the inequality above yields $I_{\mathrm{D}, \lambda}[\gamma] \geq \sup_{\pi \in  \Pi(\mu_1, \mu_2, \gamma)}  \int_{\mathcal{V} \times \mathcal{V}} \varphi_\lambda \, d \pi$. The desired result follows.
\end{proof}

\begin{lemma}\label{lemma:C2}
If $\lambda_1 > 0$ and $\lambda_2 >0$, then
\begin{align*}
\sup _{\gamma \in \mathcal{P}_{\mathrm{D}} } \sup _{\pi \in \Pi\left(\mu_1, \mu_2, \gamma\right)} \int_{\mathcal{V} \times \mathcal{V}} \varphi_\lambda \, d \pi  = \sup_{\pi \in \mathcal{G}_{\mathrm{D}, \lambda} }  \int_{\mathcal{V} \times \mathcal{V} }  \varphi_\lambda \, d \pi.
\end{align*}
\end{lemma}

\begin{proof}[Proof of \protect\Cref{lemma:C2}]

We divide the proof into the following two steps. The first step is to show that the LHS is less than or equal to the RHS. Fix any $\gamma \in \mathcal{P}_{\mathrm{D}}$.  If $\int_{\mathcal{V}} g \, d \gamma=\infty$, from the proof of \Cref{lemma:C1}, we can see that $\int_{\mathcal{V} \times \mathcal{V}} \varphi_\lambda \, d\widetilde{\pi} = \infty$ for some $\widetilde{\pi}  \in \Pi(\mu_1, \mu_2, \gamma)$ and the LHS is $\infty$.
So, the integral  $\int_{\mathcal{V} \times \mathcal{V}} \varphi_\lambda \, d\widetilde{\pi}$ is well-defined and $\widetilde{\pi} \in \mathcal{G}_{\mathrm{D}, \lambda}$. We must have $ \sup _{\pi \in \mathcal{G}_{\mathrm{D}, \lambda}} \int_{\mathcal{V} \times \mathcal{V}} \varphi_\lambda \, d \pi = \infty$ and the statement of the lemma is true.  Now suppose $\int_{\mathcal{V} } g \, d \gamma < \infty$ holds.  For any $\pi \in \Pi(\mu_1, \mu_2, \gamma)$, since $\int_{\mathcal{V} \times \mathcal{V} } \left(\lambda_1 c_1+\lambda_2 c_2\right) \, d \pi \geq 0$, the integral
\begin{align*}
\int_{\mathcal{V} \times \mathcal{V} }  \varphi_\lambda \, d \pi =  \int_{\mathcal{V} } g \, d\gamma - \int_{\mathcal{V} \times \mathcal{V}}\left(\lambda_1 c_1+\lambda_2 c_2\right) \, d \pi < \infty,
\end{align*}
is well-defined. This shows $\pi \in \mathcal{G}_{\mathrm{D}, \lambda}$, and we have $\int_{\mathcal{V} \times \mathcal{V}} \varphi_\lambda d \pi \leq \sup _{\pi \in \mathcal{G}_{\mathrm{D}, \lambda}} \int_{\mathcal{V} \times \mathcal{V}} \varphi_\lambda d \pi$. Taking the supremum over $\pi \in \Pi(\mu_1, \mu_2, \gamma)$ yields 
\begin{align*}
\sup_{\pi \in \Pi(\mu_1,\mu_2, \gamma)} \int_{\mathcal{V} \times \mathcal{V}} \varphi_\lambda d \pi \leq \sup _{\pi \in \mathcal{G}_{\mathrm{D}, \lambda}} \int_{\mathcal{V} \times \mathcal{V}} \varphi_\lambda d \pi.
\end{align*}
Thus, we showed that the inequality above holds for all $\gamma \in \mathcal{P}_{\mathrm{D}}$ and this ends the first step.

The second step is to show that the LHS is greater than or equal to the RHS. Fix any $\pi \in \mathcal{G}_{\mathrm{D}, \lambda}$. It suffices to show 
\begin{equation}
\sup _{\gamma \in \mathcal{P}_{\mathrm{D}}} \sup _{\pi \in \Pi\left(\mu_1, \mu_2, \gamma\right)} \int_{\mathcal{V} \times \mathcal{V}} \varphi_\lambda \, d \pi  \geq  \int_{\mathcal{V} \times \mathcal{V} }  \varphi_\lambda \, d \pi. \label{eq:lemmaC2-geq}
\end{equation}
When $\int_{\mathcal{V} \times \mathcal{V} } \varphi_\lambda \, d \pi > - \infty$, we have $\int (\lambda_1 c_1 + \lambda_2 c_2) \, d \pi > -\infty$ and hence $\int_{\mathcal{V}} g \, d \pi_{3,4}>-\infty$. It follows that  $\pi \in \Pi(\mu_1, \mu_2, \pi_{3,4})$ and 
\[
\int_{\mathcal{V} \times \mathcal{V} }  \varphi_\lambda \, d \pi \leq \sup _{\widetilde{\pi} \in \Pi\left(\mu_1, \mu_2, \gamma\right)} \int_{\mathcal{V} \times \mathcal{V}} \varphi_\lambda \, d \widetilde{\pi} \leq  \sup _{\gamma \in \mathcal{P}_{\mathrm{D}}} \sup _{\widetilde{ \pi}  \in \Pi\left(\mu_1, \mu_2, \gamma\right)} \int_{\mathcal{V} \times \mathcal{V}  } \varphi_\lambda \,  d  \widetilde{\pi}.
\]
When $\int_{\mathcal{V} \times \mathcal{V}} \varphi_\lambda \, d \pi=-\infty$, the inequality \eqref{eq:lemmaC2-geq} holds trivially.
\end{proof}

\begin{lemma}\label{lemma:C3}
For all $\lambda \in \mathbb{R}_{+}^2$,  one has 
\begin{equation}
\mathcal{I}_{\mathrm{D} }^\star (\lambda) = \sup _{\pi \in  \mathcal{G}_{\mathrm{D}, \lambda} } \int_{\mathcal{V} \times \mathcal{V}} \varphi_\lambda d \pi.
\end{equation}
\end{lemma}

\begin{proof}[Proof of  \protect\Cref{lemma:C3} ]
We divide the proof into the following four cases. When $\lambda_1, \lambda_2 >0$, the equality (1) follows from \Cref{lemma:C1,lemma:C2}.  When $\lambda_1 =  \lambda_2 =0$, we show that equality (1) holds. Let $A_\ell = \{ (v, v^\prime) \in \mathcal{V} \times \mathcal{V}: c_\ell(s_\ell, s_\ell^\prime) < \infty\})$, and for simplicity, we write $g: (v, v^\prime) \mapsto g(v^\prime)$ and $c_\ell: (v, v^\prime) \mapsto c_\ell(s_\ell, s_\ell^\prime)$ for $\ell=1,2$. By the convention, $0 c_\ell = 0,\pi$-a.s. if and only if $c_\ell < \infty,\pi$-a.s., it follows that
\begin{align*}
\sup _{\pi \in \mathcal{G}_{\mathrm{D}, \lambda}} \int_{\mathcal{V} \times \mathcal{V}}  \varphi_\lambda \, d\pi & = \sup \left\{ \int_{\mathcal{V} \times \mathcal{V} } g(v^\prime) \, d \pi(v, v^\prime)  : \pi \in \mathcal{G}_{\mathrm{D}, \lambda}, \pi(A_1 \cap A_2) = 1, \right\}\\
& \geq  \sup \left\{ \int_{\mathcal{V} \times \mathcal{V} } g \, d \pi  : \pi \in \mathcal{G}_{\mathrm{D}, \lambda},  \int c_\ell \, d\pi < \infty \text{ for } \ell =1,2 \right\}\\
& \geq  \sup \left\{ \int_{\mathcal{V} } g \, d \gamma  : \gamma \in \mathcal{P}_{\mathrm{D}} \right\},
\end{align*}
where the last inequality holds since for all $\pi \in \mathcal{G}_{\mathrm{D}, \lambda}$ with $\int c_{\ell} d \pi<\infty$ for $\ell=1,2$, the marginal $\pi_{3,4} \in \mathcal{P}_{\mathrm{D}}$, i.e. $\pi(\mathcal{V} \times \cdot) \in \mathcal{P}_{\mathrm{D}}$. On the other hand, for any $\pi \in  \mathcal{G}_{\mathrm{D}, \lambda}$ with $\pi(A_1 \cap A_2) =1$, define a measure $\pi_n$ on $\mathcal{V}\times \mathcal{V}$ as
\begin{align*}
\pi_n(\cdot) = \frac{  \pi\left(\cdot \cap (A_{1n} \cap A_{2n} ) \right) }{ \pi(A_{1n}\cap A_{2n})},
\end{align*}
where $A_{\ell n} = \left\{  (v, v^{\prime}) \in \mathcal{V} \times \mathcal{V}: c_{\ell} (s_{\ell}, s_{\ell}^{\prime})<n \right\}$ for $\ell = 1,2$. Since $c_\ell < n$, $\pi_n$-a.s. for $\ell=1,2$, then the second marginal of $\pi_n$ is in $\mathcal{P}_{\mathrm{D}}$.\footnote{To be more precise, the measure $\pi_n(\mathcal{V}\times \cdot)$ is in $\mathcal{P}_{\mathrm{D}}$.} By the monotone convergence theorem, 
\[
\lim_{n \rightarrow  \infty} \int_{\mathcal{V}\times \mathcal{V} } g^+ \, \mathds{1}_{A_{1n} \cap A_{2n}} d \pi =  \int_{\mathcal{V}\times \mathcal{V} } g^+ \, d \pi, \text{ and }
\lim_{n \rightarrow  \infty} \int_{\mathcal{V}\times \mathcal{V} } g^- \, \mathds{1}_{A_{1n} \cap A_{2n}} \, d \pi =  \int_{\mathcal{V}\times \mathcal{V} } g^- \, d \pi.
\]
Moreover, since $\pi(A_{1n} \cap A_{2n}) \rightarrow 1$, 
\begin{align*}
\lim_{n \rightarrow  \infty} \int_{\mathcal{V}\times \mathcal{V} } g^+ d \pi_n  =  \lim_{n \rightarrow  \infty} \frac{\int_{\mathcal{V}\times \mathcal{V} } g^+\mathds{1}_{A_{1n} \cap A_{2n}}  d \pi }{\pi(A_{1n} \cap A_{2n})}  =  \int_{\mathcal{V}\times \mathcal{V} } g^+ d \pi.
\end{align*}
Similarly, $\lim_{n \rightarrow  \infty} \int_{\mathcal{V}\times \mathcal{V} } g^+ d \pi_n =  \int_{\mathcal{V}\times \mathcal{V} } g^- d \pi$. Since $\int g d \pi$ is well-defined, we can exclude the case $\int g^{+} d \pi = \int g^{-} d \pi = \infty$. Therefore, 
\begin{align*}
\int_{\mathcal{V} \times \mathcal{V}}  g \, d \pi = \lim_{n \rightarrow \infty} \int_{\mathcal{V} \times \mathcal{V}} g \, d \pi_n \leq \sup_{\gamma \in \mathcal{P}_{\mathrm{D}}} \int_{\mathcal{V}  }  g \, d \gamma.
\end{align*}
This shows $\sup _{\pi \in \mathcal{G}_{\mathrm{D}, \lambda}} \int_{\mathcal{V} \times \mathcal{V}} \varphi_\lambda \, d \pi = \sup_{\pi \in \mathcal{P}_{\mathrm{D}}} \int_{\mathcal{V} \times \mathcal{V}} g \, d \pi$ and hence equality (1) holds for $\lambda_1 = \lambda_2 = 0$.

Next, we show that equality \eqref{eq:lemmaC2-geq} when $\lambda_1 > 0, \lambda_2 =0$.  By definition, the integral $\int \varphi_\lambda \, d \pi$ is well-defined for all $\pi \in \mathcal{G}_{\mathrm{D}, \lambda}$. If $\int \varphi_\lambda \, d\pi = \infty$ for some $\pi \in \mathcal{G}_{\mathrm{D}, \lambda}$, then $\sup _{\pi \in \mathcal{G}_{\mathrm{D}, \lambda}} \int_{\mathcal{V} \times \mathcal{V}} \varphi_\lambda \, d \pi \geq \sup_{\gamma \in \mathcal{P}_{\mathrm{D}}}  \int g \, d \gamma$. Without loss of generality, assume $\int \varphi_\lambda \, d\pi < \infty$ for all $\pi \in \mathcal{G}_{\mathrm{D}, \lambda}$. It follows that
\begin{align*}
\lambda_1 \int  c_1  d\pi \leq  \int (g^- + \lambda_1 c_1 + \lambda_2 c_2) d \pi < \infty,
\end{align*}
and $\int c_1 d\pi < \infty$ and $\pi(A_1) = 1$. By convention,  $0 \times c_2 = 0, \pi$-a.s. if and only if $0 \times c_2 < \infty, \pi$-a.s. We find that
\begin{align*}
\sup _{\pi \in \mathcal{G}_{\mathrm{D}, \lambda}} \int_{\mathcal{V} \times \mathcal{V}}  \varphi_\lambda \, d\pi & = \sup \left \{ \int_{\mathcal{V} \times \mathcal{V} } g(v^\prime) \, d \pi(v, v^\prime)  : \pi \in \mathcal{G}_{\mathrm{D}, \lambda}, \pi( A_2) = 1 \right\}\\
& = \sup \left \{ \int_{\mathcal{V} \times \mathcal{V} } g(v^\prime) \, d \pi(v, v^\prime)  : \pi \in \mathcal{G}_{\mathrm{D}, \lambda}, \pi( A_1 \cap  A_2) = 1 \right\}\\
& \geq  \sup_{\gamma \in \mathcal{P}_{\mathrm{D}} } \int_{\mathcal{V} } g \, d \gamma .
\end{align*}
On the other hand, for any $\pi \in \mathcal{G}_{\mathrm{D}, \lambda}$ with $\pi(A_2)=1$,  define a measure $\pi_n^\prime$ on $\mathcal{V} \times \mathcal{V}$ as 
\begin{align*}
\pi_n(\cdot)=\frac{\pi\left(\cdot \cap\left(A_{1 n} \right)\right)}{\pi\left(A_{1 n} \right)}.
\end{align*}
Using a similar argument as shown above, we can show $\int_{\mathcal{V} \times \mathcal{S} } g \, d \pi \leq \sup _{\gamma \in \mathcal{P}_{\mathrm{D}}} \int_{\mathcal{V}} g \, d \gamma$ and hence equality \eqref{eq:lemmaC2-geq} holds when $\lambda_1 >0$ and $\lambda_2=0$. In the same way, we can show that equality \eqref{eq:lemmaC2-geq} when $\lambda_1 = 0, \lambda_2 > 0$.
\end{proof}

\begin{lemma}\label{lemma:C4}
Let $\lambda \in \mathbb{R}^2_+$. If $\varphi_\lambda$ is interchangeable with respect to $\Pi(\mu_1, \mu_2)$, then
\[
\sup _{\pi \in \mathcal{G}_{\mathrm{D}, \lambda}} \int_{\mathcal{V} \times \mathcal{V}} \varphi_\lambda d \pi =
\sup_{ \pi \in \Gamma\left( \Pi(\mu_1, \mu_2) ,  \varphi_\lambda \right) }  \int_{\mathcal{V} \times  \mathcal{V} } \varphi_{\lambda}  d \pi.
\] 
\end{lemma}

\begin{proof}[Proof of \protect\Cref{lemma:C4}]
For any $\pi \in \mathcal{G}_{\mathrm{D}, \lambda}$, it is obvious that $\pi_{1,2} \in \Pi(\mu_1, \mu_2)$ and hence $\pi \in \Gamma\left(\Pi\left(\mu_1, \mu_2\right), \varphi_\lambda\right)$. This shows $\mathcal{G}_{\mathrm{D}, \lambda} \subset \Gamma (\Pi (\mu_1, \mu_2), \varphi_\lambda )$ and the LHS
is less than or equal to the RHS.

Next, we show the LHS is not less than the RHS. We adopt the convention that the supremum of an empty set is $-\infty$. If $\int \varphi_\lambda d\pi$ is not well-defined for all $\pi \in \Gamma\left(\Pi\left(\mu_1, \mu_2\right), \varphi_\lambda\right)$, then the proof is done trivially. Now let $\pi$ be any measure in $\Gamma (\Pi (\mu_1, \mu_2), \varphi_\lambda )$ for which integral $\int_{\mathcal{V} \times  \mathcal{V} } \varphi_{\lambda} d \pi$ is well-defined. To finish the proof, it suffices to show
\begin{align}
	\sup _{\pi \in \mathcal{G}_{\mathrm{D}, \lambda}} \int_{\mathcal{V} \times \mathcal{V}} \varphi_\lambda(v, v^\prime ) d \pi (v, v^\prime ) \geq  \int_{\mathcal{V} \times  \mathcal{V} } \varphi_{\lambda} (v , v^\prime)  d \pi (v, v^\prime ). \label{eq:IC-geq}
\end{align}

When $\int_{\mathcal{V} \times  \mathcal{V} } \varphi_{\lambda} , d \pi = - \infty$, inequality \eqref{eq:IC-geq} holds trivially. Now suppose $\int_{\mathcal{V} \times  \mathcal{V} } \varphi_{\lambda} \, d \pi = \infty$. Because $c_1, c_2 \geq 0$, we have $\int_{\mathcal{V}\times \mathcal{V} } g(v^\prime) \, d \pi(v,v^\prime) = \infty$ and is well-defined. We note $\varphi_\lambda = g^+ - g^- - (\lambda_1 c_1 + \lambda_2 c_2)$ and hence $\varphi_\lambda^+ = g^+$ and $\varphi_\lambda^- = g^- + (\lambda_1 c_1 + \lambda_2 c_2)$. Since $\int_{\mathcal{V} \times  \mathcal{V} } \varphi_{\lambda} \, d \pi$ is well-defined, then $\int_{\mathcal{V} \times \mathcal{V} } \left(\lambda_1 c_1+\lambda_2 c_2\right) \, d\pi \leq \int_{\mathcal{V} \times \mathcal{V} } \varphi^-_\lambda \, d \pi  < \infty$. This shows that $\pi \in \mathcal{G}_{\mathrm{D}, \lambda}$ and inequality \eqref{eq:IC-geq} holds. Next, suppose $\int_{\mathcal{V} \times  \mathcal{V} } \varphi_{\lambda} \, d \pi < \infty$. Given that the integral is well-defined, using the same reasoning as demonstrated above, we have $\int_{\mathcal{V} \times \mathcal{V}} g(v^{\prime}) \, d \pi(v, v^{\prime}) < \infty$ and $\int_{\mathcal{V} \times \mathcal{V}} (\lambda_1 c_1+\lambda_2 c_2) \, d \pi < \infty$. So $\pi \in \mathcal{G}_{\mathrm{D}, \lambda}$ and the proof is done.
\end{proof}

\subsubsection{Proof of \texorpdfstring{\Cref{corollary:ID-duality-marakov}}{}}

We provide only the derivation of the upper bound $\mathcal{I}_{\mathrm{D}}(\delta) = \sup_{\gamma \in \Sigma_{\mathrm{D}}(\delta)} \int \mathds{1}(s_1 + s_2 \le z) \, d \gamma(s_1, s_2)$. We can derive the expression of the lower bound $\inf_{\gamma \in \Sigma_{\mathrm{D}}} \int \mathds{1}(s_1 + s_2 \le z) \, d \gamma(s_1, s_2)$ by the similar reasoning and the following identity.
\begin{align*}
    \inf_{\gamma \in \Sigma_{\mathrm{D}}(\delta)} \int \mathds{1}(s_1 + s_2 \le z) \, d \gamma(s_1, s_2) = 1 - \sup_{\gamma \in \Sigma_{\mathrm{D}}(\delta)} \int \mathds{1}(\{s_1 + s_2 > z\}) \, d \gamma(s_1, s_2).
\end{align*}

When $\lambda_1 = 0$ or $\lambda_2 = 0$,  $g_\lambda(s_1 , s_2) = 0$ for all $(s_1,s_2) \in \mathcal{S}_1 \times \mathcal{S}_2$. When $\lambda_1 \neq 0$ and $\lambda_2 \neq 0$, we have
\begin{align*}
g_\lambda(s_1, s_2) 
&= \sup_{s_1', s_2'} \left[  \mathds{1}(s_1' + s_2' \le z)  - \lambda_1 |s_1 - s_1'|^2 -\lambda_2 |s_2 - s_2'|^2   \right] \\
& = \left(  1- \inf_{s_1' + s_2' \le z} \left[  \lambda_1 |s_1 - s_1'|^2  + \lambda_2 |s_2 - s_2'|^2    \right]  \right)^{+}\\
& = \begin{cases} 
1 & \text { if } s_1 + s_2 \leq z  \\ 
\left[1 - \frac{ \lambda_1\lambda_2 (s_1 + s_2 - z)^2}{\lambda_1 + \lambda_2 } \right]^+ & \text { if } \{s_1 + s_2 > z\} 
\end{cases}.
\end{align*}
By some simple algebra, we have.
\begin{align*}
    g_{\lambda, 1}(s_1, s_2) 
    & = 
    \sup_{s_1'} \left[\mathds{1}(s_1' + s_2 \le z) - \lambda_1 |s_1 - s_1'|^2 \right] \\
    & =
    \begin{cases}
        1 & \text{ if } s_1 + s_2 \le z, \\
       \left(1 - \lambda_1 | s_1 + s_2 - z |^2 \right)^{+}   & \text{ if } \{s_1 + s_2 > z\},
    \end{cases}
\end{align*}
and 
\begin{align*}
    g_{\lambda, 2}(s_1, s_2) 
    & = 
    \sup_{s_2'} \left[\mathds{1}(s_1 + s_2' \le z) - \lambda_2 | s_2 - s_2'|^2 \right] \\ 
    & =
    \begin{cases}
        1 & \text{ if } s_1 + s_2 \le z, \\
       \left(1 - \lambda_2 | s_1 + s_2 - z |^2 \right)^{+}   & \text{ if } \{s_1 + s_2 > z\}.
    \end{cases}
\end{align*}
By applying \Cref{thm:ID-duality}, we have that for each $\delta = (\delta_1, \delta_2) \in \mathbb{R}_{++}^2$, 
\begin{align*}
    \mathcal{I}_{\mathrm{D}}(\delta) 
    & = \inf_{\lambda \in \mathbb{R}_{+}^2} \left[\langle \lambda, \delta \rangle + \sup_{\pi \in \Pi(\mu_1, \mu_2)} \int g_{\lambda}(s_1, s_2) \, d \pi(s_1, s_2)  \right].
\end{align*}
However, in the rest of proof, we show for all $\delta = (\delta_1, \delta_2) \in \mathbb{R}^2_+$,
\[
\mathcal{I}_{\mathrm{D}} (\delta) =   \inf _{\lambda \in \mathbb{R}_{+}^2}  \sup _{\pi \in \Pi\left(\mu_1, \mu_2\right)} \left[\langle\lambda, \delta\rangle+\int_{\mathcal{V}} g_\lambda \, d \pi\right]  = \sup _{\pi \in \Pi\left(\mu_1, \mu_2\right)} \inf _{\lambda \in \mathbb{R}_{+}^2}\left[\langle\lambda, \delta\rangle+\int_{\mathcal{V}} g_\lambda \, d \pi\right] .
\]
Define a function $F:   \Pi(\mu_1, \mu_2) \times  \mathbb{R}^2_+ \rightarrow \mathbb{R}$ as
\[
F: ( \pi, \lambda ) \mapsto   - \langle \lambda, \delta \rangle  - \int_{\mathcal{S}_1 \times \mathcal{S}_2 } g_\lambda \, d \pi.
\]
We note that for any $(s_1, s_2)$, the function $\lambda \mapsto g_\lambda(s_1, s_2)$ is convex since it is the supremum of a set of affine functions in $\lambda$. As a result,   $\lambda \mapsto  - \int g_\lambda d \pi$ is concave for each fixed $\pi$. For any $\lambda \in \mathbb{R}^2_+$, the function $\pi \mapsto F(\pi, \lambda)$ is continuous due to continuous and bounded $g_\lambda$ and  Portmanteau's theorem. Moreover, it is easy to verify that $\pi \mapsto F(\pi, \lambda)$ is convex. By \citet[Theorem 2]{Fan_1953_MinimaxThms}' minimax theorem, we have
\[
\inf_{\pi \in \Pi(\mu_1, \mu_2)} \sup_{\lambda \in \mathbb{R}^2_+}  F(\pi, \lambda)   =\sup_{\lambda \in \mathbb{R}^2_+}  \inf_{\pi \in \Pi(\mu_1, \mu_2)}    F(\pi, \lambda).
\]
As a result, we have for all $\delta = (\delta_1, \delta_2) \in \mathbb{R}^2_{++}$,
\begin{align*}
\mathcal{I}_{\mathrm{D}}(\delta) & =  \inf_{\lambda \in \mathbb{R}^2_+} \sup_{\pi \in\Pi(\mu_1, \mu_2)}  - F(\pi, \lambda)  = - \sup_{\lambda \in \mathbb{R}^2_+}  \inf_{\pi \in \Pi(\mu_1, \mu_2)} F(\pi, \lambda) \\
& = -   \inf_{\pi \in \Pi(\mu_1, \mu_2)}  \sup_{\lambda \in \mathbb{R}^2_+}   F(\pi, \lambda) = \sup_{\pi \in \Pi(\mu_1, \mu_2)}  \inf_{\lambda \in \mathbb{R}^2_+}    -  F(\pi, \lambda) \\
 & =  \sup _{\pi \in \Pi\left(\mu_1, \mu_2\right)}  \inf _{\lambda \in \mathbb{R}_{+}^2}\left[\langle\lambda, \delta\rangle+ \int_{\mathcal{V}} g_\lambda \, d \pi\right].
\end{align*}
Using the same reasoning as above, the application of \citet[Theorem 2]{Fan_1953_MinimaxThms} to $\mathcal{I}_{\mathrm{D}}(\delta_1, 0)$ yields
\[
\mathcal{I}_{\mathrm{D}}(\delta_1, 0) = \sup _{\pi \in \Pi\left(\mu_1, \mu_2\right)}  \inf _{\lambda_1 \in \mathbb{R}_{+} }\left[  \lambda_1 \delta_1 + \int_{\mathcal{V}} g_{\lambda, 1} \, d \pi\right].
\]
Since $g_{\lambda} \downarrow g_{\lambda, 1}$ as $\lambda_2 \uparrow \infty$, the monotone convergence theorem implies
\begin{align*}
\inf _{\lambda \in \mathbb{R}^2_{+} }\left[ \langle \lambda , \left(\delta_1, 0 \right) \rangle + \int_{\mathcal{V}} g_{\lambda} d \pi\right] 
&  =  \inf _{\lambda_1 \in \mathbb{R}_{+} }\left[ \lambda_1 \delta_1 + \inf_{\lambda_2 \in \mathbb{R}_+} \int_{\mathcal{V}} g_{\lambda} \, d \pi\right] \\
& =   \inf _{\lambda_1 \in \mathbb{R}_{+} }\left[ \lambda_1 \delta_1 + \lim_{\lambda_2 \rightarrow \infty} \int_{\mathcal{V}} g_{\lambda} \, d \pi\right]  \\
& =  \inf _{\lambda_1 \in \mathbb{R}_{+} }\left[ \lambda_1 \delta_1 + \int  g_{\lambda, 1} \, d \pi \right].
\end{align*}
Taking the supremum over $\pi \in \Pi(\mu_1, \mu_2)$ on both sides yields that for $\delta_1 > 0$,
\[
\mathcal{I}_{\mathrm{D}}(\delta_1, 0) = \sup_{\pi \in \Pi(\mu_1, \mu_2) }   \inf _{\lambda_1 \in \mathbb{R}_{+} }\left[ \lambda_1 \delta_1 + \int  g_{\lambda, 1} \, d \pi \right] = \sup_{\pi \in \Pi(\mu_1, \mu_2) } \inf _{\lambda \in \mathbb{R}^2_{+} }\left[ \langle \lambda , \left(\delta_1, 0 \right) \rangle + \int_{\mathcal{V}} g_{\lambda} \, d \pi\right].
\]
Similarly, we can show that for $\delta_2 > 0$,
\[
\mathcal{I}_{\mathrm{D}}\left(0, \delta_2\right)=\sup _{\pi \in \Pi\left(\mu_1, \mu_2\right)} \inf _{\lambda_2 \in \mathbb{R}_{+}}\left[\lambda_2 \delta_2+\int_{\mathcal{V}} g_{\lambda, 2} \, d \pi\right]= \sup_{\pi \in \Pi(\mu_1, \mu_2) } \inf _{\lambda \in \mathbb{R}^2_{+} }\left[ \langle \lambda , \left( 0, \delta_2 \right) \rangle + \int_{\mathcal{V}} g_{\lambda} \, d \pi\right] .
\]
In addition, when $\delta_1 = \delta_2 = 0$, we note $g_\lambda \downarrow g$ as $\lambda_1, \lambda_2 \uparrow \infty$ and the monotone convergence theorem implies $\inf_{\lambda \in \mathbb{R}^2_+} \int g_\lambda d \pi = \int g d \pi$ and 
\[
\mathcal{I}_{\mathrm{D}}(0) = \sup_{\pi \in \Pi(\mu_1, \mu_2)} \inf_{\lambda \in \mathbb{R}^2_+ } \int g_\lambda \, d \pi =\inf_{\lambda \in \mathbb{R}^2_+}  \sup_{\pi \in \Pi(\mu_1, \mu_2)}   \int g_\lambda \, d \pi =  \sup_{\pi \in \Pi(\mu_1, \mu_2)} \int g \, d \pi.
\]
This completes the proof that for all $\delta = (\delta_1, \delta_2) \in \mathbb{R}^2_+$
\[
\mathcal{I}_{\mathrm{D}} (\delta) = \sup _{\pi \in \Pi\left(\mu_1, \mu_2\right)} \inf _{\lambda \in \mathbb{R}_{+}^2}\left[\langle\lambda, \delta\rangle+\int_{\mathcal{V}} g_\lambda \, d \pi\right] =   \inf _{\lambda \in \mathbb{R}_{+}^2}  \sup _{\pi \in \Pi\left(\mu_1, \mu_2\right)} \left[\langle\lambda, \delta\rangle+\int_{\mathcal{V}} g_\lambda \, d \pi\right].
\]

\subsubsection{Proof of \texorpdfstring{\Cref{thm:I-duality}}{}}

The expressions of $\mathcal{I}(\delta_1, 0)$ and $\mathcal{I}(0,\delta_2)$ can be derived from $\mathcal{I}(\delta_1, \delta_2)$ for $\delta_1, \delta_2 >0$ with appropriate modifications of the cost function. In particular, consider another cost function $\widehat{c}_2(s_2, s_2^\prime ) =  \infty \mathds{1} \{ s_2 \neq s_2^\prime \}$ and the optimal transport distance $\widehat{\boldsymbol{K}}_2$ associated with $\widehat{c}_2$. Define an uncertainty set $\widehat{\Sigma}(\delta_1, \delta_2)$ depending on $\boldsymbol{K}_1$ and  $\widehat{\boldsymbol{K}}_2$ as 
\begin{align*}
\widehat{\Sigma}(\delta_1, \delta_2) =  \left\{ \gamma \in \mathcal{P}(\mathcal{S}): \boldsymbol{K}_1( \gamma_{13}, \mu_{13}) \leq \delta_1, \widehat{\boldsymbol{K}}_2( \gamma_{23}, \mu_{23}) \leq \delta_2 \right\}.
\end{align*}
Moreover, we define $\widehat{\mathcal{I}}: \mathbb{R}^2_+ \rightarrow \mathbb{R}$ as
\begin{align*}
\widehat{\mathcal{I}} (\delta_1, \delta_2) = \sup_{\gamma \in \widehat{\Sigma}(\delta_1, \delta_2)} \int_{\mathcal{V} } f(v) \, d\gamma(v).
\end{align*}
We note $\widehat{\boldsymbol{K}}_2(\mu, \nu) = 0$ if and only if $\mu = \nu$. So, for all $\delta_2 > 0$, $\widehat{\Sigma}(\delta_1, \delta_2) = \Sigma(\delta_1, 0)$ and $\widehat{\mathcal{I}} (\delta_1, \delta_2) =  \mathcal{I}(\delta_1, 0)$. Using the dual reformulation of $\widehat{\mathcal{I}}$ on $\mathbb{R}^2_{++}$, we have
\[
\mathcal{I}(\delta_1, 0) = \widehat{\mathcal{I}} (\delta_1, \delta_2) = \inf _{\lambda_1 \in \mathbb{R}_{+}}\left[\langle\lambda, \delta\rangle +\sup _{\varpi \in \Pi\left(\mu_{13}, \mu_{23}\right)} \int_{\mathcal{V}} f_{\lambda}(s_1, s_2) \, d \varpi(s_1,s_2) \right] ,
\]
where 
\[
\begin{aligned}
f_\lambda(s_1, s_2) & = \sup _{ (y_1^{\prime},y_2^{\prime}, x^\prime) \in \mathcal{S} }\left\{ f(y_1^\prime, y_2^\prime, x^\prime)-\lambda_1 c_1 \left(s_1, (y_1^{\prime}, x^\prime) \right)  - \lambda_2  \widehat{c}_2 \left(s_2, (y_2^{\prime}, x^\prime)  \right) \right\} \\
& =  \sup _{s_1^{\prime} \in \mathcal{S}_1 }\left\{f\left(y_1^{\prime}, y_2, x_2\right)-\lambda_1 c_1 \left(s_1, ( y_1^{\prime}, x_2)   \right) \right\} = f_{\lambda, 1}\left(s_1, s_2\right).
\end{aligned}
\]
Since $f_{\lambda, 1}\left(s_1, s_2\right)$ is independent of $\lambda_2$, letting $\lambda_2 = 0$ yields
\begin{align*}
\mathcal{I}(\delta_1, 0) = \inf _{\lambda_1 \in \mathbb{R}_{+}}\left[\lambda_1 \delta_1+\sup _{\varpi \in \Pi\left(\mu_{13}, \mu_{23}\right)} \int_{\mathcal{V}} f_{\lambda, 1}(v) \, d \varpi(v)\right].
\end{align*}
Using the same reasoning, we can get the expression of $\mathcal{I}(0,\delta_2)$.

In the rest of the proof, we show that the dual reformulation of $\mathcal{I}$ on $\mathbb{R}^2_{++}$ holds. Let $\bar{ \mathcal{P} }$ denote the set of $\gamma \in \mathcal{P}(\mathcal{S})$ such that $\boldsymbol{K}_\ell(\mu_{\ell 3}, \gamma_{\ell 3}) < \infty$ for $\ell =1,2$ and $\int_{\mathcal{S} } f d\gamma > - \infty$. Taking the Legendre transform on $\mathcal{I}$ gives
\begin{align*}
\mathcal{I}^\star(\lambda) & :=  \sup _{\delta \in \mathbb{R}_{+}^2}\{\mathcal{I}(\delta)-\langle\lambda, \delta\rangle\} = \sup _{\delta \in \mathbb{R}_{+}^2} \sup _{\gamma \in \Sigma(\delta)}\left\{\int_{\mathcal{S}} f \, d \gamma-\langle\lambda, \delta\rangle\right\} \\
& =  \sup_{\delta \in \mathbb{R}_+^2 } \sup_{\gamma \in \bar{\mathcal{P}}  }   \left\{  \int_{\mathcal{S}} f \, d \gamma-\langle\lambda, \delta\rangle: \boldsymbol{K}_\ell (\mu_{\ell 3}, \gamma_{\ell3}) \leq \delta_\ell, \forall \ell \in [2]  \right\}\\
& = \sup_{\gamma \in \bar{\mathcal{P}}  }  \sup_{\delta \in \mathbb{R}_+^2 }  \left\{  \int_{\mathcal{S}} f \, d \gamma-\langle\lambda, \delta\rangle: \boldsymbol{K}_\ell (\mu_{\ell 3}, \gamma_{\ell3}) \leq \delta_\ell,  \forall \ell \in [2]   \right\} \\
& =  \sup_{\gamma \in \bar{\mathcal{P}}  } \underbrace{   \left\{ \int_{\mathcal{S}} f \, d \gamma-\lambda_1 \boldsymbol{K}_1 (\mu_{13}, \gamma_{23})-\lambda_2 \boldsymbol{K}_2(\mu_{23}, \gamma_{23})  \right\} }_{: = I_{\lambda}[\gamma]}  = \sup _{\gamma \in \bar{ \mathcal{P} } } I_{\lambda}[\gamma].
 \end{align*}
We note that the expression above still holds when $\lambda \in \mathbb{R}_+^2 \setminus \mathbb{R}^2_{++}$.
Recall the definition of the function $\phi_\lambda: \mathcal{V}\times \mathcal{S} \rightarrow \mathbb{R}$.
Let $\mathcal{G}_\lambda$ denote the set of $\pi \in \mathcal{P}(\mathcal{V} \times \mathcal{S})$ such that $\int_{\mathcal{V} \times \mathcal{S}} \phi_\lambda \, d \pi$ is well-defined and the first and second marginals coincides with $\mu_{13}$ and $\mu_{23}$ respectively.\footnote{To be more precise, $\pi( (A_1 \times \mathcal{S}_2) \times \mathcal{S} ) =  \mu_{13}(A_1)$ and $\pi( ( \mathcal{S}_1\times A_2 ) \times \mathcal{S} ) =  \mu_{23}(A_2)$ for all Borel sets $A_1 \in \mathcal{B}_{\mathcal{S}_1}$ and  $A_2 \in \mathcal{B}_{\mathcal{S}_2}$.}
 \Cref{lemma:AO3} implies $\mathcal{I}^\star(\lambda) = \sup_{\pi \in \mathcal{G}_\lambda } \int_{\mathcal{V} \times \mathcal{S} } \phi_\lambda \,  d\pi$.  By \Cref{lemma:AO4}, we have for all $\lambda \in \mathbb{R}^2_{+}$,
\[
\mathcal{I}^\star(\lambda) =\sup _{\pi \in \Gamma\left(\Pi\left(\mu_{13}, \mu_{23}\right), \phi_\lambda\right)} \int_{\mathcal{V} \times \mathcal{V}} \phi_\lambda \, d \pi.
\]
Example 2 of \citet{zhang2022simple} implies that $\phi_\lambda: \mathcal{V} \times\mathcal{S} \rightarrow \mathbb{R}$ satisfies the interchangeability principle with respect to $\Pi(\mu_{13}, \mu_{23} )$. As a result, \Cref{IPforGroup} implies that for all $\lambda \in \mathbb{R}_+^2$,
\begin{align*}
\mathcal{I}^\star(\lambda) = \sup _{\gamma \in \Pi\left(\mu_1, \mu_2\right)} \int_{\mathcal{V}} f_\lambda(v) \, d \gamma(v),
\end{align*}
where $f_\lambda(v) = \sup_{s \in \mathcal{S}} \phi_{\lambda} (v, s)$.

From \Cref{lemma:ID-concavity}, $\mathcal{I}$ is bounded from below, non-decreasing, and concave. As a result, $\mathcal{I} (\delta) = \infty$ for all $\delta \in \mathbb{R}^2_+$ or $\mathcal{I}(\delta) < \infty$ for all $\delta \in \mathbb{R}^2_+$. 
In the first case, $\mathcal{I}^\star = \infty$ on $\mathbb{R}^2_+$ by definition and hence we have $\mathcal{I}(\delta) = \inf _{\lambda \in \mathbb{R}_{+}^2}\left\{\langle\lambda, \delta \rangle+\mathcal{I}^{\star}(\lambda)\right\} = \infty$.
For the second case, by \Cref{lemma: Legendre_transform}, for all $\delta \in \mathbb{R}^2_{++}$,
\[
\mathcal{I}(\delta) = \inf_{\lambda \in \mathbb{R}^2_+}  \left\{ \langle \lambda, \delta \rangle +  \mathcal{I}^{\star}(\lambda) \right\} =  \inf_{\lambda \in \mathbb{R}^2_+}  \left\{ \langle \lambda, \delta \rangle +  \sup _{\gamma \in \Pi\left(\mu_1, \mu_2\right)} \int_{\mathcal{V}} f_\lambda(v) \, d \gamma(v) \right\},
\]
and the proof is complete.

\

\begin{lemma}\label{lemma:AO1}
If $\lambda_1 >0$ and $\lambda_2 >0$, then 
\[
\sup _{\gamma \in \bar{ \mathcal{P} } } I_{\lambda}[\gamma]=\sup _{\gamma \in \bar{ \mathcal{P} }  } \sup _{ \pi \in \Pi\left(\mu_{13}, \mu_{23}, \gamma\right)    } \int_{\mathcal{V} \times \mathcal{S}} \phi_\lambda (v, s^\prime ) \, d \pi (v, s^\prime) .
\]
\end{lemma}

\begin{proof}[Proof of \Cref{lemma:AO1}]

The proof is almost identical to that of \Cref{lemma:C1}, so we only give the sketch. For notational convenience, we write $c_\ell: (s_1, s_2, y_1, y_2, x) \mapsto c_\ell( s_\ell, (y_\ell, x) )$ for $\ell =1,2$ and $f: (s_1, s_2, s^\prime) \mapsto f(s^\prime)$.

Fix any $\epsilon >0$ and $\gamma \in \bar{\mathcal{P} }$.
Let $K = \{ K_1, K_2, K_3\}$ with $K_1 = \{3,4,5\}$, $K_2 = \{1,3,5 \}$  and $K_3 = \{ 2,4,5 \}$ and we note that $K$ is decomposable. By \Cref{prop:MultiVBP}, there is a $\widetilde{\pi} \in \Pi(\mu_{13}, \mu_{23}, \gamma )$ satisfying $I_{\lambda}[\gamma] \leq \int_{\mathcal{V}\times \mathcal{S} } \phi_\lambda \, d \widetilde{\pi}+\epsilon$. Since $\epsilon >0$ and $\gamma \in \bar{\mathcal{P}}$ are arbitrary, this shows LHS $\leq$ RHS.
The proof of LHS $\geq$ RHS is identical to the proof of \Cref{lemma:C1}.

\end{proof}

\begin{lemma}\label{lemma:AO2}
If $\lambda_1 > 0$ and $\lambda_2 >0$, then 
\[
\sup _{\gamma \in \bar{ \mathcal{P} } } \sup _{\pi \in \Pi\left(\mu_1, \mu_2, \gamma\right)} \int_{\mathcal{V} \times \mathcal{S} } \phi_\lambda \, d \pi  = \sup_{\pi \in \mathcal{G}_{\lambda} }  \int_{\mathcal{V} \times \mathcal{S} }  \phi_\lambda \, d \pi.
\]
\end{lemma}
\begin{proof}[Proof of \Cref{lemma:AO2}]
The proof is the same as that of \Cref{lemma:C2}.
\end{proof}

\begin{lemma}\label{lemma:AO3}
For all $\lambda \in \mathbb{R}_+^2$, one has $\mathcal{I}^{\star}(\lambda)=\sup _{\pi \in \mathcal{G}_{ \lambda}} \int_{\mathcal{V} \times \mathcal{S}} \phi_\lambda \, d \pi$.
\end{lemma}

\begin{proof}[Proof of \Cref{lemma:AO3} ]

The proof is almost the same as \Cref{lemma:C3} as long as we replace $g$ with $f$, $\varphi_\lambda$ with $\phi_\lambda$, $A_\ell$ with $B_\ell$,  $A_{\ell n}$ with $B_{\ell n}$ and $\mathcal{P}_{\mathrm{D}}$ with $\bar{\mathcal{P} }$, where  
\[
B_\ell = \{ ( (s_1,s_2), (y_1,y_2,x) ) \in \mathcal{V} \times \mathcal{S}: c_\ell(s_\ell, (y_\ell,x) ) < \infty\},
\]
and 
\[
B_{\ell n} = \left \{ ( (s_1,s_2), (y_1,y_2,x) ) \in \mathcal{V} \times \mathcal{S}: c_\ell(s_\ell, (y_\ell,x) ) < n \right \},
\]
for $\ell = 1,2$. 
\end{proof}

\begin{lemma}\label{lemma:AO4}
Let $\lambda \in \mathbb{R}^2_{+}$. If $\phi_\lambda: \mathcal{V} \times \mathcal{S}\rightarrow \mathbb{R}$ is interchangeable with respect to $\Pi(\mu_1, \mu_2)$, then
\begin{align*}
\sup _{\pi \in \mathcal{G}_{ \lambda}} \int_{\mathcal{V} \times \mathcal{S} } \phi_\lambda \, d \pi =
\sup_{ \pi \in \Gamma\left( \Pi(\mu_1, \mu_2) ,  \phi_\lambda \right) }  \int_{\mathcal{V} \times  \mathcal{S} } \phi_{\lambda} \,  d \pi 
\end{align*}
\end{lemma}

\begin{proof} [Proof of \protect \Cref{lemma:AO4}]
The proof is the same as  \protect  \Cref{lemma:C4} .
\end{proof}

\subsection{Proofs in \texorpdfstring{\Cref{sec:additional-properties}}{}}

\subsubsection{Proof of \texorpdfstring{\Cref{thm:ID-finite}}{}}

First, assuming that condition \eqref{eq:GC-I} does not hold, we show $\mathcal{I}_{\mathrm{D} }(\delta) = \infty$. Fix any $\lambda = (\lambda_1, \lambda_2) \in \mathbb{R}^2_{+}$ and  $v = (s_1, s_2) \in \mathcal{V}$.  For any $B \geq \lambda_1 \vee \lambda_2$, there is  $v^\prime = (s_1^\prime, s_2^\prime) \in \mathcal{V}$  such that 
	\[
	g(s_1^\prime, s_2^\prime ) > B \left[   1  + \boldsymbol{d}_{\mathcal{S}_1} (s_1, s_1^\prime )^{p_1}+ \boldsymbol{d}_{ \mathcal{S}_2 }(s_2, s_2^\prime )^{p_2}  \right],
	\]
	and hence
	\[
	\begin{aligned}
		\varphi_\lambda (v, v^{\prime}) & =g (s_1^{\prime}, s_2^{\prime} )  - \lambda_1 \boldsymbol{d}_{\mathcal{S}_1} (s_1, s_1^\prime)^{p_1} -\lambda_2 \boldsymbol{d}_{ \mathcal{S}_2 }(s_2, s_2^\prime )^{p_2}  \\
		& >  B \left[   1  + \boldsymbol{d}_{\mathcal{S}_1} (s_1, s_1^\prime )^{p_1}+ \boldsymbol{d}_{ \mathcal{S}_2 }(s_2, s_2^\prime )^{p_1}  \right]  -  \lambda_1 \boldsymbol{d}_{\mathcal{S}_1} (s_1, s_1^\prime)^{p_1} -\lambda_2 \boldsymbol{d}_{ \mathcal{S}_2 }(s_2, s_2^\prime )^{p_2}  \\
		& \geq  B + (B - \lambda_1) \boldsymbol{d}_{\mathcal{S}_1} (s_1, s_1^\prime)^{p_1} + (B- \lambda_2)  \boldsymbol{d}_{ \mathcal{S}_2 }(s_2, s_2^\prime )^{p_2}  \geq B. 
	\end{aligned}
	\]
This shows that for all $\lambda \in \mathbb{R}^2_{+}$ and $B$ large enough, we have $g_\lambda(v) = \sup _{v^{\prime} \in \mathcal{V}} \varphi_\lambda(v, v^{\prime}) \geq B$ for all $v \in \mathcal{V}$. Therefore, by \Cref{thm:ID-duality}, we have 
\begin{align*}
	\mathcal{I}_{\mathrm{D} }(\delta)  \geq   \sup _{\pi \in \Pi\left(\mu_1, \mu_2 \right)} \int_{\mathcal{V} } g_\lambda(v) \, d \pi (v) \geq B,
\end{align*}
	for all $B$ large enough. As a result, $\mathcal{I}_{\mathrm{D} }(\delta) = \infty$. 
	
Conversely, assuming that the growth condition \eqref{eq:GC-I} holds, we show $\mathcal{I}_{\mathrm{D} }(\delta)< \infty$. For all $\pi \in \Sigma_{\mathrm{D} }(\delta)$,
	\begin{align*}
		\int_{\mathcal{V}} f(v) \, d\pi(v) 
		& \leq 
		\int_{\mathcal{S}_1 \times \mathcal{S}_2 } M \left[   1  + \boldsymbol{d}_{\mathcal{S}_1} (s_1^\star, s_1)^{p_1}+ \boldsymbol{d}_{ \mathcal{S}_2 }(s_2^\star, s_2 )^{p_2}  \right]  d \pi(s_1, s_2) \\
& =  M + M \boldsymbol{W}_{p_1} ( \pi_1 ,  \delta_{s_1^\star} )^{p_1} + M \boldsymbol{W}_{p_2} ( \pi_2 ,  \delta_{s_2^\star} )^{p_2} \\
	& \leq  M + \sum_{j=1}^2 M \left[ \boldsymbol{W}_{p_j} ( \pi_j ,  \mu_j ) + \boldsymbol{W}_{p_j} ( \mu_j, \delta_{s_j^\star}  ) \right]^{p_j}   <  \infty,
	\end{align*}
where $\pi_j$ denotes the marginal measure of $\pi$ on $\mathcal{S}_j$ and $\delta_{s_j^\star}$ denotes the Dirac measure at $s_j^\star \in \mathcal{S}_j$. The last step follows from $\mu_j \in \mathcal{P}_{p_j}(\mathcal{S}_j)$  for $j=1,2 $ and $\pi \in \Sigma_{\mathrm{D} }(\delta)$, i.e., $ \boldsymbol{W}_{p_j} ( \pi_j ,  \mu_j )^{p_j} \leq \delta_j$ for $j=1, 2$.

\subsubsection{Proof of \texorpdfstring{\Cref{thm:I-finite}}{}}
 
First, we assume condition \eqref{eq:GC-II} does not hold and aim to show $\mathcal{I}(\delta) = \infty$. Fix any $\lambda = (\lambda_1, \lambda_2) \in \mathbb{R}_{+}^2$. For any $v= (s_1, s_2) \in \mathcal{V}$ and $B \geq \lambda_1 \vee \lambda_2$, there exists  $s^\prime = (y_1^\prime, y_2^\prime, x^\prime)$  such that 
\[
f(s^\prime) \geq B \left[   1  + \boldsymbol{d}_{\mathcal{S}_1} (s_1, s_1^\prime )^{p_1}+ \boldsymbol{d}_{ \mathcal{S}_2 }(s_2, s_2^\prime )^{p_2}  \right].
\]
Therefore, 
\[
\begin{aligned}
\phi_\lambda (v, s^{\prime})&=f (s^{\prime})-\lambda_1 \boldsymbol{d}_{\mathcal{S}_1}  (s_1, s_1^{\prime})^{p_1} -\lambda_2 \boldsymbol{d}_{ \mathcal{S}_2} (s_2, s_2^{\prime})^{p_2} \\
& \geq B  + (B - \lambda_1)  \boldsymbol{d}_{\mathcal{S}_1} (s_1, s_1^\prime )^{p_1}  + (B - \lambda_2)  \boldsymbol{d}_{ \mathcal{S}_2 }(s_2, s_2^\prime )^{p_2}  \geq B.
\end{aligned}
\]
As a result,  $f_{\lambda}(v) = \sup_{s^\prime \in \mathcal{S} }  \phi_\lambda (v, s^{\prime}) \geq B$ for all $v\in \mathcal{V}$ and all $B$ large enough. Since $B >0$ is arbitrary, we must have $\sup _{\varpi \in \Pi\left(\mu_{13}, \mu_{23}\right)} \int_{\mathcal{V}} f_\lambda(v) d \varpi(v)  =\infty$. By \Cref{thm:I-duality}, we have $\mathcal{I}(\delta) = \infty$. 

Conversely, we show that the condition \eqref{eq:GC-II} implies $\mathcal{I}(\delta)< \infty$. For any $\gamma \in \Sigma(\delta)$,
	\begin{align*}
		\int_{ \mathcal{S} }  f(s) \, d \gamma (s) 
		& \leq 
		 \int_{ \mathcal{S} } M \left[   1  +\boldsymbol{d}_{\mathcal{S}_1} (s_1^\star, s_1)^{p_1}+\boldsymbol{d}_{ \mathcal{S}_2 }(s_2^\star, s_2)^{p_2}  \right] d \gamma (s)    \\
		& \leq 
		M + M \boldsymbol{W}_{p_1}( \delta_{s_1^\star }, \gamma_{13} )^{p_1} +  M \boldsymbol{W}_{p_2}( \delta_{s_2^\star }, \gamma_{23} )^{p_2} \\
		& \leq 
		M +  \sum_{j=1}^2 M \left[  \boldsymbol{W}_{p_j}( \delta_{s_j^\star }, \mu_{j3} )  + \boldsymbol{W}_{p_j}( \mu_{j3}, \gamma_{j3} ) \right]^{p_j}    < \infty,
	\end{align*} 
where $\gamma_{j3}$ is the marginal measure of $\gamma$ on $\mathcal{S}_j = \mathcal{Y}_j \times \mathcal{X}$ and  $\delta_{s_j^\star}$ is the Dirac measure concentrated at $\{s_j^\star\}$.The last step follow $\gamma \in \Sigma(\delta)$ and $\mu_{j3} \in \mathcal{P}_{p_j}(\mathcal{S}_j)$ for $j=1, 2$.

\subsubsection{Proof of \texorpdfstring{\Cref{thm:ID-existence}}{}}

In this section, we first prove the weak compactness of $\Sigma_{\mathrm{D} }(\delta)$ for all $\delta \in \mathbb{R}_{+}^2$ when $\mathcal{S}_1$ and $\mathcal{S}_2$ are both proper and $c_j =  \boldsymbol{d}^{p_j}_{\mathcal{S}_j }  $ for some $p_j \geq 1$. As a result, $\boldsymbol{K}_j =\boldsymbol{W}_{p_j}^{p_j} $ and the set $\Sigma_{\mathrm{D} }(\delta )$ can be written as
\[
\Sigma_{\mathrm{D} }(\delta) := \left\{ \gamma  \in \mathcal{P}(\mathcal{S}_1 \times \mathcal{S}_2):  \boldsymbol{W}_{p_1}(\gamma_1, \mu_1) \leq \delta_1^{1/p_1},  \  \boldsymbol{W}_{p_2}(\gamma_2, \mu_2) \leq \delta_2^{1/p_2} \right\}.
\] 
For any Polish metric space $\mathcal{X}$, let $B_{\mathcal{P}_p(\mathcal{X} )}( \mu, \delta) := \{  \gamma \in \mathcal{P}(\mathcal{X}) :   \boldsymbol{W}_p(\mu, \gamma) \leq \delta  \}$ denote the ball centered at $\mu$ in Wasserstein space $\mathcal{P}_p(\mathcal{X})$. When there is no ambiguity we will abbreviate this notation by referring to  $B_p( \mu, \delta)$.  

\begin{proposition}\label{prop:non-overlapping-compactness}
Suppose \Cref{assumption:finite-moments-ID,assumption:metric-cost,assumption:proper} hold. Then, $\Sigma_{\mathrm{D} }(\delta)$ is weakly compact.
\end{proposition}

\begin{proof}[Proof of \protect\Cref{prop:non-overlapping-compactness}]
Theorem 1 of  \citet{yue2022linear} implies $B_p( \mu, \delta)$ weakly compact whenever $\mu$ has a finite $p$-th moment. As a result, the set $\Sigma_{\mathrm{D} }(\delta)$ can be written as 
\[
\Sigma_{\mathrm{D} }(\delta) = \Pi \left(\mathcal{B}_1,\mathcal{B}_2  \right), \quad \text{where }  \mathcal{B}_1  = B_{p_1}( \mu_1, \delta_1^{1/p_1}) \text{ and } \mathcal{B}_2 =  B_{p_2} ( \mu_2, \delta_2^{1/p_2}). 
\]
Since $\mathcal{B}_1$ and $\mathcal{B}_2$ are weakly compact in $\mathcal{P}(\mathcal{S}_1)$ and  $\mathcal{P}(\mathcal{S}_1)$, respectively, then they are uniformly tight by Prokhorov’s theorem. By Lemma 4.4 of \citet{villani2009optimal}.  $\Sigma_{\mathrm{D} }(\delta)$ is tight in $\mathcal{P}(\mathcal{S}_1 \times \mathcal{S}_2)$. By Prokhorov’s theorem again, $\Sigma_{\mathrm{D} }(\delta)$ has a compact closure under the topology of weak convergence. To show the weakly compactness of $\Sigma_{\mathrm{D} }(\delta)$, it suffices to show it is closed.

Let $\pi^n \in \Sigma_{\mathrm{D} }(\delta) \equiv \Pi(\mathcal{B}_1,\mathcal{B}_2)$ be a sequence converging weakly to $\pi^\infty \in \mathcal{P}(\mathcal{S}_1 \times \mathcal{S}_2)$. We have 
\[
\boldsymbol{W}_{p_1}(\pi^n_1, \mu_1) \leq \delta_1^{1 / p_1} \ \text{ and } \  \boldsymbol{W}_{p_2}(\pi^n_2, \mu_2) \leq \delta_1^{1 / p_2}.
\]
Let $\pi^n_j$ denote the marginal distribution of $\pi^n$ on $\mathcal{S}_j$. For any open $U_1$ in $\mathcal{S}_1$, the Portmanteau theorem implies
\[
\liminf_{n \rightarrow \infty}\pi^{n}_1(U_1) = \liminf_{n \rightarrow \infty}\pi^{n}(U_1 \times \mathcal{S}_2) \geq \pi^\infty(U_1 \times \mathcal{S}_2) = \pi^\infty_1(U_1).
\]
This shows $\pi^{n}_1$ weakly converges to $\pi^\infty_1$. Moreover, $\boldsymbol{W}_{p_1}(\pi_1^\infty, \mu_1 ) \leq \delta_1^{1 / p_1}$ can be seen from weakly closedness of $\mathcal{B}_1$. Using the identical argument, we can show $\pi^{n}_2$ weakly converges to $\pi^\infty_2$ and $\boldsymbol{W}_{p_2} (\pi_2^\infty, \mu_2 ) \leq \delta_1^{1 / p_2}$. This shows $\pi^\infty \in \boldsymbol{W}_{p_2}\left(\pi_2^n, \mu_2\right) \leq \delta_1^{1 / p_2}$ and hence $\Sigma_{\mathrm{D} }(\delta)$ is weakly closed.
\end{proof}

The weak compactness of $\Sigma_{\mathrm{D} }(\delta)$ does not depend on the functional forms of metrics $\boldsymbol{d}_{\mathcal{S}_1}$ and $\boldsymbol{d}_{\mathcal{S}_2}$. Essentially,  the topological properties of $\mathcal{S}_1$ and $\mathcal{S}_2$, mainly properness, determines the weak compactness of $\Sigma_{\mathrm{D} }(\delta)$.

\begin{proof}[Proof of \texorpdfstring{\Cref{thm:ID-existence}}{} ]
Since \Cref{prop:non-overlapping-compactness} implies that $\Sigma_{\mathrm{D} }(\delta)$ is weakly compact, by Weierstrass’ theorem, it suffices to show $\pi \mapsto \int_{\mathcal{V} } g \, d\pi$ is weakly upper semi-continuous. Let $\{ \pi^k \}_{k=1}^{\infty}$ be any sequence in $\Sigma_{\mathrm{D} }(\delta)$ that weakly converges to $\pi^{\infty} \in \Sigma_{\mathrm{D} }(\delta)$, we show $\limsup_{n \rightarrow \infty} \int_{\mathcal{V} } g \, d\pi^k  \leq \int_{\mathcal{V} } g \, d\pi^\infty$. For any $\rho >0$, define an auxiliary function $f_\rho: \mathcal{V} \rightarrow \mathbb{R}$ as $g_\rho (v) =   f(v) \wedge \left[ M (1+\rho^{p_0^{\prime}}+\rho^{p_1^{\prime}} ) \right]$. Let $A_1 = \{ (s_1, s_2) \in \mathcal{V} :\boldsymbol{d}_{\mathcal{S}_1}(s_1^\star, s_1) \geq \rho \}$ and $A_2 = \{ (s_1, s_2) \in \mathcal{V} : \boldsymbol{d}_{\mathcal{S}_2}(s_2^\star,s_2) \geq \rho \}$.  It is easy to verify that for all $v \in \mathcal{V}$, 
	\[
	\left| g(v)-g_\rho(v)\right| 
	\leq 
	\begin{cases} 
		M \left[\boldsymbol{d}_{\mathcal{S}_1} ( s_1^\star, s_1)^{p_1^{\prime}}+\boldsymbol{d}_{\mathcal{S}_2}  (s_2^\star, s_2)^{p_2^{\prime }} \right] & \text { if } v \in A_1 \cap A_2 , \\ 
		M\; \boldsymbol{d}_{\mathcal{S}_1} (s_1^\star, s_1)^{p_1^{\prime}}     & \text { if }  v \in A_1 \cap A_2^c, \\ 
		M \; \boldsymbol{d}_{\mathcal{S}_2} ( s^\star_2,s_2)^{p_2^{\prime}}   & \text { if }  v \in A_1^c \cap A_2 , \\ 
		0 & \text { otherwise. }\end{cases}
	\]
For any $\pi \in \Sigma_{\mathrm{D} }(\delta)$, we have
	\begin{align*}
		\left|  \int_{\mathcal{V} } g \, d\pi -  \int_{\mathcal{V} } g_\rho \, d\pi \right| 
		& \leq 
		\int_{\mathcal{V} } |g - g_{\rho}| \, d \pi \\
		& \leq 
		\int_{A_1  \cap A_2 }  |g - g_{\rho}| \, d \pi + \int_{A_1  \cap A_2^c }  |g- g_{\rho}| \, d \pi +  \int_{A_1^c \cap A_2}  |g - g_{\rho}| \, d \pi .
	\end{align*}
	By Lemma 1 in \citet{yue2022linear}, there exists $B > 0$ such that $W_{p_j} (\pi_j, \delta_{s_j^{\star}})^{p_j} \leq B$ for $j = 1, 2$ and all $\pi \in \Sigma_{\mathrm{D} }(\delta)$, where $\pi_j$ is the marginal of $\pi$ on $\mathcal{S}_j$ and $\delta_{s_j^{\star}}$ is a Dirac measure at $\{s_j^{\star} \}$. Therefore, we have 
	\[
	\begin{aligned}
		\int_{A_1  \cap A_2^c }  |g - g_{\rho}| \, d \pi &  \leq   M \int_{A_1  \cap A_2^c } \boldsymbol{d}_{\mathcal{S}_1}\left(s_1^{\star}, s_1\right)^{p_1^{\prime}}  d \pi  \leq  M {\rho^{p_1 - p_1^\prime }}  \int_{A_1  \cap A_2^c }  \boldsymbol{d}_{\mathcal{S}_1} (s_1, s^\star_1)^{p_1} \, d\pi   \\
		& \leq M {\rho^{p_1 - p_1^\prime }}   W_{p_1}( \pi_1, \delta_{s_1^\star} )^{p_1}   \leq  B \rho^{p_1^\prime - p_1 }  . 
	\end{aligned}
	\]
Similarly, we can show  $\int_{A_1^c  \cap A_2 }  |g - g_{\rho}| \, d \pi  \leq B \rho^{p_2^\prime - p_2}$ and
\[
\begin{aligned}
\int_{A_1  \cap A_2}  |g - g_{\rho}| \, d \pi &  \leq \int_{A_1  \cap A_2 }   M \left[\boldsymbol{d}_{\mathcal{S}_1} (s^\star_1, s_1)^{p_1^{\prime}}+\boldsymbol{d}_{\mathcal{S}_2}  (s_2, s^\star_1)^{p_2^{\prime}} \right]  d\pi(s_1, s_2) \\
	& \leq    B (\rho^{p_1^\prime - p_1 }   + \rho^{ p_2^\prime - p_2} ).
\end{aligned}
\]
Therefore, we have for all $\pi \in \Sigma_{\mathrm{D} }(\delta)$,
\begin{align*}
	\left|   \int_{\mathcal{V} } g \, d\pi -  \int_{\mathcal{V} } g_\rho \, d\pi \right|  
	\leq 
	\int_{\mathcal{V}}\left|g-g_\rho\right| \, d \pi 
	\leq  
	2 B (\rho^{p_1^\prime - p_1 }   + \rho^{ p_2^\prime - p_2} ).
\end{align*}
For any $\epsilon > 0$, there is a $\rho>0$ large enough such that $4B (\rho^{p_1^\prime - p_1}   + \rho^{ p_2^\prime - p_2} ) < \epsilon/2$. By Lemma 3 in \citet{yue2022linear},  we have $	\limsup_{k \rightarrow \infty } \int_{\mathcal{V} }  g_\rho \, d\pi^k  \leq  \int_{\mathcal{V} }  g_\rho \, d\pi^\infty$ and hence there is a $k(\epsilon)$ large enough such 
	\[
	\int_{\mathcal{V} }  g_\rho \, d\pi^k  -  \int_{\mathcal{V} }  g_\rho \, d\pi^\infty < \frac{\epsilon}{2}, \quad \text{for all } k > k(\epsilon).
	\]
	Consequently, for all $ k > k(\epsilon)$, the following holds:
	\[
	\begin{aligned}
		\int_{\mathcal{V} }  g \, d\pi^k   -  \int_{\mathcal{V} }  g \, d\pi^\infty & \leq   \int_{\mathcal{V} } \left| g  - g_\rho \right| \, d\pi^k +   \int_{\mathcal{V} }  g_\rho \, d\pi^k  -  \int_{\mathcal{V} }  g_\rho \, d\pi^\infty +   \int_{\mathcal{V} }  \left| g_\rho -  g \right| \, d\pi^\infty \\
		& \leq  4B (\rho^{p_1^\prime - p_1 }   + \rho^{ p_2^\prime - p_2} ) + \int_{\mathcal{V} }  g_\rho \, d\pi^k  -  \int_{\mathcal{V} }  g_\rho \, d\pi^\infty  < \epsilon.
	\end{aligned}
	\]
	Since $\epsilon$ is arbitrary, we must have $\limsup _{k \rightarrow \infty} \int_{\mathcal{V}} g \, d \pi^k \leq \int_{\mathcal{V}} g \, d \pi^{\infty}$. This completes the proof.
	
\end{proof}

\subsubsection{Proof of \texorpdfstring{\Cref{thm:I-existence}}{}}

Here, we will only show that $\Sigma(\delta)$ is weakly compact. This is because the upper semi-continuity of $\gamma \rightarrow \int f \, d \gamma$ over $\gamma \in \Sigma(\delta)$ can be shown using the same argument for the proof of \Cref{thm:ID-existence}. We  write
\[
\Sigma(\delta) = \left\{ \gamma \in \mathcal{P}(\mathcal{S}):  \boldsymbol{W}_{p_1}(\gamma_{1}, \mu_1 ) \leq \delta_1^{1/p_1},  \  \boldsymbol{W}_{p_2}(\gamma_{2}, \mu_2 )\leq \delta_2^{1/p_2}   \right\}.
\]

\begin{lemma}  \label{lemma:overlapping-tightness}
For $j =1,2$, let $\mathcal{G}_j $ be an uniformly tight subset of $\mathcal{P}(\mathcal{S}_j)$. Then the following set
\[
\Gamma (\mathcal{G}_1, \mathcal{G}_2) := \left\{\gamma \in \mathcal{P}(\mathcal{S}): \gamma_{13} \in \mathcal{G}_1, \gamma_{23} \in \mathcal{G}_2  \right\},
\]
is tight in $\mathcal{P}(\mathcal{S})$.

\end{lemma}

\begin{proof}[Proof of \protect\Cref{lemma:overlapping-tightness}]

First, we assume there exist $\mu \in \mathcal{G}_1$ and  $\nu \in \mathcal{G}_2$ such that $\mu(\mathcal{Y}_1 \times A) = \nu(\mathcal{Y}_2 \times A) $ for all $A \in \mathcal{B}_\mathcal{X}$, i.e. $\mu$ and $\nu$ have same marginal distribution on $\mathcal{X}$. Otherwise, $\Gamma (\mathcal{G}_1, \mathcal{G}_2)$ will be empty and hence the statement holds trivially.

Since $\mathcal{G}_1$ is uniformly tight, then for any $\epsilon > 0$, there is a compact set $K_\epsilon \subset \mathcal{S}_1 \equiv \mathcal{Y}_1 \times \mathcal{X}$ such that $\mu( K^c_\epsilon) \leq \epsilon$ for all $\mu \in \mathcal{G}_1$. Similarly, there is a compact set $L_\epsilon \subset \mathcal{S}_2 \equiv \mathcal{Y}_2\times \mathcal{X}$ such that $\nu( L^c_\epsilon) \leq \epsilon$ for all $\nu \in \mathcal{G}_2$. 
Moreover, define a mapping $\sigma: \mathcal{S} \rightarrow \mathcal{S}$ as $\sigma: (y_1, y_2, x) \mapsto (y_1, x, y_2)$. Trivially, $\sigma$ is a homeomorphism (a continuous mapping whose inverse is also continuous) from $\mathcal{S}$ to $\mathcal{S}$. Let $E_\epsilon =  \sigma^{-1} ( K_{\epsilon} \times \mathcal{Y}_2)$ and $G_\epsilon  = \mathcal{Y}_1 \times  L_\epsilon$. Explicitly, $(y_1,y_2, x) \in E_\epsilon \Leftrightarrow (y_1,x) \in K_\epsilon$. Fix any $\gamma \in \Gamma (\mathcal{G}_1, \mathcal{G}_2)$, let $S = (Y_1,Y_2, X)$ be a random variable with $\gamma$ as its law, i.e. $\operatorname{Law}(S) = \gamma$. We must have $\gamma_{j3} \in \mathcal{G}_j$ for $j=1,2$. Then,
\[
\begin{aligned}
\mathbb{P} \left[ S \notin  E_\epsilon \cap G_\epsilon   \right] & \leq \mathbb{P} \left[  S \notin  E_\epsilon     \right]  +  \mathbb{P} \left[  S \notin  G_\epsilon     \right] \\
& = \mathbb{P} \left[  (Y_1, Y_2, X) \notin  E_\epsilon     \right]  +  \mathbb{P} \left[  (Y_1, Y_2, X) \notin  G_\epsilon     \right]  \\
& =  \mathbb{P} \left[  (Y_1, X) \notin  K_\epsilon     \right]  +  \mathbb{P} \left[  (Y_2, X) \notin  L_\epsilon     \right] \\
& \leq \gamma_{1 3} (K_\epsilon^c) + \gamma_{23} (L_\epsilon^c)\\
& \leq 2\epsilon.
\end{aligned}
\]
The desired result follows from the compactness of $E_\epsilon \cap G_\epsilon $  in $\mathcal{S}$. To see this, we note $\operatorname{proj}_{\mathcal{Y}_1}: (y_1, x) \mapsto y_1$ is continuous from $\mathcal{S}_1$ to $\mathcal{Y}_1$ and  hence $\operatorname{proj}_{\mathcal{Y}_1}(K_\epsilon)$ is compact. As a result, $\operatorname{proj}_{\mathcal{Y}_1}(K_\epsilon) \times L_\epsilon$ is compact. Since $E_\epsilon \cap G_\epsilon$ is a subset of a compact set and its compactness follows from the closedness of $E_\epsilon$ and $G_\epsilon$.

\end{proof}

\begin{proposition} \label{prop:overlapping-compactness}
Suppose \Cref{assumption:finite-moments-I,assumption:metric-cost,assumption:proper} hold. Then, $\Sigma(\delta)$ is weakly compact.
\end{proposition}

\begin{proof}
By abuse of notations, let $\mathcal{B}_1= B_{p_1}(\mu_{13}, \delta_1^{1 / p_1})$ and $\mathcal{B}_2= B_{p_2}(\mu_{23}, \delta_2^{1 / p_2})$. We can rewrite $\Sigma(\delta)=\Gamma (\mathcal{B}_1, \mathcal{B}_2)$. By  \Cref{lemma:overlapping-tightness}, $\Sigma(\delta)$ is tight and hence has a compact closure under weak topology. Using a similar argument in the proof of  \Cref{prop:non-overlapping-compactness}, we can show $\Sigma(\delta)$ is weakly closed.  Therefore, $\Sigma(\delta)$ is weakly compact in $\mathcal{P}(\mathcal{S})$.
\end{proof}

\subsubsection{Proof of \texorpdfstring{\Cref{prop:Identified-Sets-Interval}}{}}

We focus on $\Theta(\delta)$ since the proof of $\Theta_{\mathrm{D}}(\delta)$ is identical to that of $\Theta(\delta)$. The proof of \Cref{prop:Identified-Sets-Interval} for $\Theta(\delta)$ follows form the following two lemmas.

\begin{lemma}\label{lemma: T-continuity}
Suppose that the Assumptions in \Cref{prop:Identified-Sets-Interval} hold.  
Then, the linear functional $T: \Sigma(\delta) \rightarrow \mathbb{R}$ given by $ \pi \mapsto \int_{\mathcal{S} } f d \pi$ is continuous.
\end{lemma}

\begin{proof}
Since $\mu_{\ell 3}$ has finite $p_\ell$-th moment, then  for all $\pi \in \Sigma(\delta)$, $\pi_{\ell 3}$, i.e.,  the projection onto $\mathcal{Y}_\ell \times \mathcal{X}$  also has finite $p_\ell$-th moment. Define a function $h: \mathcal{S} \rightarrow \mathbb{R}$ as 
\[
h(s) = M \left[   1  + \boldsymbol{d}_{\mathcal{S}_1} (s_1^\star,s_1 )^{p_1'}+ \boldsymbol{d}_{ \mathcal{S}_2 }(s_2^\star,s_2)^{p_2'} \right],
\]
where $s = (y_1, y_2, x)$, $s_1 = (y_1, x)$ and $s_2 = (y_2, x)$.
We note $h \in L^1 (\pi)$ for all $\pi \in  \Sigma(\delta)$. Using the identical argument in the proof of \Cref{thm:ID-existence}, we can show that $\pi \mapsto \int f d\pi$ is upper semicontinuous on $\Sigma(\delta)$. By replacing $f$ by $-f$, we can see that $\pi \mapsto \int (-f) d\pi$ is upper semicontinuous and hence $\pi \mapsto \int f d\pi$ is lower semicontinuous on $\Sigma(\delta)$. As a result, $\pi \mapsto \int f \, d\pi$ is continuous on $\Sigma(\delta)$.
\end{proof}

\begin{lemma}\label{lemma: Path-Connected}
Suppose that \Cref{assumption:finite-moments-I,assumption:metric-cost} hold. Then $\Sigma(\delta)$ is connected under weak topology.
\end{lemma}

\begin{proof}
Fix any $\pi$ and $\pi^\prime$ in $\Sigma(\delta)$. It suffices to show $\nu : t \mapsto t \pi + (1- t)\pi^\prime$ is continuous from $[0,1]$ into $\Sigma(\delta)$. We note $\Sigma(\delta) \subset \mathcal{P}_p(\mathcal{S})$ is metrizable under $\boldsymbol{W}_{p}$ for $p = p_1 \wedge p_2$. Fix any $t_0 \in [0,1]$. Let $t_1 \neq t_0$ be any point in $[0,1]$ such that  $\Delta =| t_1 - t_0| > 0$ is sufficiently small. Without loss of generality, we assume $t_0 < t_1$. For simplicity, we write $\gamma = t_0 \pi + (1- t_1) \pi^\prime \geq 0$. By the triangle inequality,
\[
\begin{aligned}
\boldsymbol{W}_{p}( \nu (t_0), \nu (t_1) ) & =   \boldsymbol{W}_{p}\left(  \nu (t_0), \gamma + \Delta \pi^\prime   \right) \\
& \leq    (1-\Delta)  \boldsymbol{W}_{p}\left(   \nu (t_0), (1-\Delta)^{-1} \gamma   \right)    + \underbrace{ \Delta \boldsymbol{W}_{p}\left(   \nu (t_0) ,  \pi^\prime  \right)}_{= O(\Delta)}.  \\
\end{aligned}
\]
Consider the following derivation:
\[
\begin{aligned}
\boldsymbol{W}_{p}\left(  \nu (t_0), (1-\Delta)^{-1} \gamma  \right)  &= \boldsymbol{W}_{p} \Bigg( \nu(t_0)  ,  \underbrace{ \frac{  \nu(t_0) - \Delta \pi^\prime  }{1- \Delta}  }_{= \rho_\Delta } \Bigg ) = \boldsymbol{W}_{p}\left(  (1-\Delta) \rho _\Delta+ \Delta \pi^\prime,  \rho_\Delta  \right)\\
& \leq \Delta \boldsymbol{W}_{p}\left( \pi^\prime ,  \rho_\Delta  \right)  =  \Delta \boldsymbol{W}_{p}\left( \pi^\prime ,  \frac{  \nu(t_0) - \Delta \pi^\prime  }{1- \Delta} \right).
\end{aligned}
\]
Since $\lim_{\Delta \rightarrow 0}  \frac{  \nu(t_0) - \Delta \pi^\prime  }{1- \Delta}  =  \nu(t_0)$ in weak topology induced by $\boldsymbol{W}_{p}$, then 
\[
\lim_{\Delta \rightarrow 0}\boldsymbol{W}_{p}\left( \pi^\prime ,  \frac{  \nu(t_0) - \Delta \pi^\prime  }{1- \Delta} \right) =  \boldsymbol{W}_{p}\left( \pi^\prime ,  \nu(t_0)  \right) < \infty.
\]
As a result, 
\[
\boldsymbol{W}_p\left(\nu\left(t_0\right),(1-\Delta)^{-1} \gamma\right) \leq \Delta \boldsymbol{W}_{p}\left( \pi^\prime ,  \frac{  \nu(t_0) - \Delta \pi^\prime  }{1- \Delta} \right) \rightarrow 0, \quad \text{as } \Delta \rightarrow 0,
\]
and hence
\[
\boldsymbol{W}_p\left(\nu\left(t_0\right), \nu\left(t_1\right)\right)  \rightarrow 0,  \quad \text{as } \Delta \rightarrow 0.
\]
Interchange the role of $t_0$ and $t_1$, we can show the case when $\boldsymbol{W}_p\left(\nu\left(t_0\right), \nu\left(t_1\right)\right) \rightarrow 0$ as $\Delta = |t_1 -t_0| \rightarrow 0$. This shows $\nu: t \mapsto t \pi+(1-t) \pi^{\prime}$ is continuous on $[0,1]$. So $\Sigma(\delta)$ is path-connected and hence connected under weak topology.
\end{proof}

\subsection{Proofs in  \texorpdfstring{\Cref{sec:continuity}}{}}

\subsubsection{Proof of \texorpdfstring{\Cref{thm:ID-continuity}}{}}

Note that the proof of \Cref{lemma: Legendre_transform} implies that 
If $\mathcal{I}_{\mathrm{D}}(\delta)$ is finite for some $\delta > 0$, then $\mathcal{I}_{\mathrm{D}}(\delta)$ is finite for all $\delta > 0$ because $\mathcal{I}(\delta)$ is concave.

\begin{lemma}\label{lemma:ID-continuity-zero}
    Suppose that \Cref{assumption:g-bounded-below,assumption:g-Psi-growth} hold. Then for any $\delta =(\delta_1, \delta_2) \in \mathbb{R}_+^2$, we have
    \[
    0 \leq \mathcal{I}_{\mathrm{D} }(\delta_1, \delta_2)-\mathcal{I}_{\mathrm{D} }(0,0) \leq \Psi(\delta_1, \delta_2).
    \] 
    Moreover, $\mathcal{I}_{\mathrm{D} }$ is continuous on $(0,0)$.
    \end{lemma}
   
    \begin{proof}
  Fix any $\widetilde{\gamma} \in \Sigma_{\mathrm{D} }(\delta)$ and any $\epsilon >0$. We can construct random variables $\widetilde{V} = ( \widetilde{S}_1, \widetilde{S}_2) \in \mathcal{V}$ with $\widetilde{\gamma}  = \operatorname{Law} (\widetilde{V} )$ and write  $\widetilde{\gamma}_j =   \operatorname{Law}(\widetilde{S}_j)$ for $j \in [2]$. Let $K = \{K_1, K_2,K_3\}$ with $K_1 = \{ 1,3 \}$, $K_2 = \{ 2,4 \}$ and $K_3 = \{3,4\}$. It is easy to see $K$ is decomposable, and \Cref{prop:MultiVBP} implies that there are random variables  $(V,  \widetilde{V}) = (S_1, S_2,  \widetilde{S}_1, \widetilde{S}_2 ) \in  \mathcal{V} \times \mathcal{V}$ such that $\mu_1 = \operatorname{Law} (S_1)$, $\mu_2 = \operatorname{Law} (S_2)$ and  $\mathbb{E}  \left[ c_j(S_j, \widetilde{S}_j) \right]   \leq   \boldsymbol{K}_j (\mu_j, \widetilde{\gamma}_j) + \epsilon \leq \delta_j +\epsilon$ for $j \in [2]$.  Let $\pi$ denote the law of $(V,  \widetilde{V} )$. 
  Therefore, with $\gamma = \operatorname{Law}\left(S_1, S_2\right) \in \Sigma_{\mathrm{D} }(0)$, we have
\[
\begin{aligned}
\int_{\mathcal{V} } g \, d \widetilde{\gamma} - \mathcal{I}_{\mathrm{D} }(0,0) & \leq  \int_{\mathcal{V} } g \, d \widetilde{\gamma} -   \int_{\mathcal{V} } g \, d \gamma  = \int_{\mathcal{V} \times \mathcal{V}  } \left[ g(v) -  g(\tilde{v} ) \right] \, d \pi(v, \tilde{v}) \\ 
& =  \mathbb{E} \left[  g(V) -  g(\widetilde{V} ) \right]     \leq   \mathbb{E} \left[ \Psi\left(c_1 (S_1, \widetilde{S}_1), c_2 (  S_2, \widetilde{S}_2)\right)  \right]  \\
& \leq  \Psi\left(   \mathbb{E} \left[c_1 (S_1, \widetilde{S}_1)\right], \mathbb{E}  \left[c_2 (  S_2, \widetilde{S}_2) \right] \right) \\
& \leq  \Psi\left(  \delta_1 + \epsilon , \delta_2 + \epsilon  \right). 
\end{aligned}
\]
Since the measure $\widetilde{\gamma} \in \Sigma_{\mathrm{D} }(\delta)$ is arbitrary, we must have 
\[
\mathcal{I}_{\mathrm{D} }(\delta_1, \delta_2)- \mathcal{I}_{\mathrm{D} }(0,0) = \sup_{\widetilde{\gamma} \in \Sigma_{\mathrm{D} }(\delta)} \int_{\mathcal{V} } g \, d \widetilde{\gamma} - \mathcal{I}_{\mathrm{D} }(0,0) \leq \Psi(  \delta_1 + \epsilon , \delta_2 + \epsilon). 
\]
Since $\Psi$ is continuous and $\epsilon >0$ is arbitrary, then $\mathcal{I}_{\mathrm{D} }(\delta_1, \delta_2) - \mathcal{I}_{\mathrm{D} }(0,0) \leq \Psi(  \delta_1  , \delta_2)$. The monotonicity of $\mathcal{I}_{\mathrm{D} }$ implies $\mathcal{I}_{\mathrm{D} }(\delta_1, \delta_2) \geq \mathcal{I}_{\mathrm{D} }(0,0)$.  In addition, the continuity of $\mathcal{I}_{\mathrm{D} }$ at $(0,0)$ follows from the continuity of $\Psi$ at $(0,0)$ and letting $(\delta_1, \delta_2) \rightarrow (0,0)$.
\end{proof}

In fact, \Cref{lemma:ID-concavity} and Proof of \Cref{lemma:ID-continuity-zero} implies the effective domain of $\mathcal{I}_{\mathrm{D}}$ is either $\mathbb{R}_+^2$ or $\emptyset$ because $\mathcal{I}_{\mathrm{D}}$ is non-decreasing and concave.

\begin{lemma} \label{lemma:ID-inequality}
    Suppose that \Cref{assumption:g-bounded-below,assumption:g-Psi-growth} hold, and $\mathcal{I}_{\mathrm{D}}(\delta)$ is finite for some $\delta \in \mathbb{R}_{++}^2$. If $\eta_0 > \eta \geq 0 $ and $\delta \geq 0$, one has
    \[
    0 \leq \mathcal{I}_{\mathrm{D} }(\eta_0,\delta) - \mathcal{I}_{\mathrm{D} }(\eta,\delta)  \leq \Psi( \eta_0 - \eta, 0 ).
    \]
    and 
    \[
    0 \leq \mathcal{I}_{\mathrm{D} }(\delta, \eta_0) - \mathcal{I}_{\mathrm{D} }(\delta, \eta)  \leq \Psi( 0,  \eta_0 - \eta ).
    \]
    
    \end{lemma}
    
\begin{proof}
 We assume that for all $\eta, \delta \geq 0$, there exists $\gamma^{\eta, \delta} \in \Sigma_{\mathrm{D} }(\eta, \delta)$ such that $\mathcal{I}_{\mathrm{D} }(\eta, \delta) = \int g \, d \gamma^{\eta, \delta}$.  Otherwise, due to the continuity of $\Psi$ on $\mathbb{R}^2_+$, we can repeat the proof with $\epsilon$-approximation optimizer and let $\epsilon \downarrow 0$. In addition, since $\mathcal{I}_{\mathrm{D}}(\delta) < \infty$ for some $\delta \in \mathbb{R}^2_+$, the $\mathcal{I}_{\mathrm{D}} (\delta)  < \infty$ for all $\delta\in \mathbb{R}^2_+$.

Let $\gamma^{\eta, \delta}_\ell$ denote the marginal of $\gamma^{\eta_0, \delta}$ on $\mathcal{S}_\ell$.  Fix $\gamma^{\eta_0, \delta} \in \mathcal{P}( \mathcal{S}_1 \times \mathcal{S}_2)$. Define a probability measure $\gamma_1^\star$ on $\mathcal{S}_1$ as
\[
\gamma_1^\star = \left( \frac{\eta}{\eta_0} \right) \gamma_1^{\eta_0, \delta}  + \left( \frac{\eta_0 - \eta }{\eta_0} \right ) \mu_1.
\]
By definition,   $\boldsymbol{K}_1 ( \gamma_1^{\eta_0, \delta},  \mu_1  ) \leq \eta_0$ and $\boldsymbol{K}_2 ( \gamma_2^{\eta_0, \delta},  \mu_2  ) \leq \delta$. By convexity of $\nu \mapsto \boldsymbol{K}_1( \nu, \mu_1)$, we have  $\boldsymbol{K}_1(\gamma_1^\star , \mu_1) \leq    \eta$ and  $\boldsymbol{K}_1(\gamma_1^\star , \gamma_1^{\eta_0, \delta} )  \leq \eta_0 - \eta$.  Without loss of generality, suppose there is an optimal coupling $\nu \in \Pi( \gamma^{\eta, \delta}_1,  \gamma_1^\star )$ such that 
\[
\boldsymbol{K}_1( \gamma_1^{\eta_0, \delta}, \gamma_1^\star ) = \int_{\mathcal{S}_1 \times \mathcal{S}_1} c_1 \, d \nu.
\]
By gluing lemma, we can construct random variables $(S_1, S_2, \widetilde{S}_1) \in \mathcal{V}\times \mathcal{S}_1$ with a probability measure $\widehat{\pi} \equiv \operatorname{Law} (S_1, S_2, \widetilde{S}_1)$ such that  
\[
\widehat{\pi}_{1,2} = \text{Law}(S_1, S_2) =  \gamma^{\eta_0, \delta},  \quad  \widehat{\pi}_{1,3} =  \text{Law}( S_1,  \widetilde{S}_1 )  =\nu \in  \Pi (\gamma_1^{\eta, \delta}, \gamma_1^{\star}) ,
\]
and 
\[
\boldsymbol{K}_1 (\gamma_1, \gamma_1^{\eta_0, \delta}) =  \mathbb{E} \left[c_1(S_1, \widetilde{S}_1) \right] \leq  \eta_0 - \eta.
\]
Let $\gamma = \operatorname{Law} ( \widetilde{S}_1, S_2) \in \mathcal{P}(\mathcal{V} )$ and it is obvious that $\widetilde{\gamma}_1 \in \Sigma_{\mathrm{D} }(\eta, \delta)$. Next, consider the following derivation:
\begin{align*}
\mathcal{I}_{\mathrm{D} }(\eta_0, \delta) - \mathcal{I}_{\mathrm{D} }(\eta, \delta) & \leq  \int g (v) \, d \gamma^{\eta_0, \delta}(v) -   \int g(v) \, d \gamma(v)  \\
   & = \int_{\mathcal{V} \times \mathcal{V} }   \left[g(s_1, s_2) - g( \tilde{s}_1, s_2 ) \right] d  \widehat{\pi}  ( s_1, s_2 ,  \tilde{s}_1) \\
 & =   \mathbb{E} \left[  g(S_1, S_2) -   g( \widetilde{S}_1, S_2) \right]   \leq  \mathbb{E}\left[   \Psi\left(c_1( S_1, \widetilde{S}_1  ), 0 \right)  \right]  \\
& \leq   \Psi\left(    \mathbb{E} \left[ c_1( S_1, \widetilde{S}_1  ) \right] , 0 \right)    \leq    \Psi\left(   \eta_0 - \eta , 0 \right).
\end{align*}
Using the same argument, we can show $ \mathcal{I}_{\mathrm{D} } (\delta, \eta_0)-\mathcal{I}_{\mathrm{D} }(\delta, \eta) \leq \Psi(0, \eta_0-\eta)$.  
 \end{proof}
    
Now we present the proof of \Cref{thm:ID-continuity}.

 \begin{proof}[{Proof of \Cref{thm:ID-continuity}}]
Since $\mathcal{I}_{\mathrm{D} }$ is concave on $\mathbb{R}^2_{+}$, then $\mathcal{I}_{\mathrm{D} }$ is continuous on $\mathbb{R}^2_{++}$. By \Cref{lemma:ID-continuity-zero} , $\mathcal{I}_{\mathrm{D} }$ is continuous at $(0,0)$.  Let $E_0 = \{ (x,0)\in \mathbb{R}^{2}_+: x>0 \} $  and $E_1 = \{ (0,y)\in \mathbb{R}^{2}_+: y>0 \}$.  To complete the proof, it suffices to show $\mathcal{I}_{\mathrm{D} }$ is continuous at all $\delta \in E_0 \cup E_1$. 

Fix any $(\eta,0) \in E_0$. For any $\eta_0 \geq \eta$ and any $\delta >0$, we have
\[
\begin{aligned}
\mathcal{I}_{\mathrm{D} }(\eta_0, \delta ) -  \mathcal{I}_{\mathrm{D} }(\eta, 0 ) & = \mathcal{I}_{\mathrm{D} }(\eta_0, \delta ) -  \mathcal{I}_{\mathrm{D} }(\eta, \delta ) +  \mathcal{I}_{\mathrm{D} }(\eta, \delta ) -  \mathcal{I}_{\mathrm{D} }(\eta, 0 )  \\
& \leq \Psi (\eta_0-\eta, 0 ) +  \Psi (0, \delta ) = \Psi ( |\eta_0-\eta|, 0 ) +  \Psi (0, \delta ).
\end{aligned}
\]
Similarly, for any $\eta_0 < \eta$ and $\delta > 0$, 
\[
\mathcal{I}_{\mathrm{D} }(\eta, \delta ) -  \mathcal{I}_{\mathrm{D} }(\eta_0 ,0 )  \leq  \Psi (|\eta_0-\eta|, 0 ) +  \Psi (0, \delta ).
\]
This shows for all $\eta, \eta_0$ and $\delta$ in $(0, \infty)$, one has
\[
\left|\mathcal{I}_{\mathrm{D} }(\eta_0, \delta ) -  \mathcal{I}_{\mathrm{D} }(\eta, 0 ) \right |  \leq  \Psi (|\eta_0-\eta|, 0 ) +  \Psi (0, \delta ).
\]
The continuity of $\mathcal{I}_{\mathrm{D} }$ at $(\eta, 0)$ follows from the continuity of $\Psi$ at $(0,0)$ and letting $(\eta_0, \delta) \rightarrow (\eta, 0)$. 
Since $(\eta, 0) \in E_0$ is arbitrary, $\mathcal{I}_{\mathrm{D} }$ is continuous at all $x \in E_0$. Using the same argument, we can show $\mathcal{I}_{\mathrm{D} }$ is continuous at all $x\in E_1$. The desired result follows.
\end{proof}

\subsubsection{Proof of \texorpdfstring{\Cref{thm:I-continuity}}{}}

Note that the proof of \Cref{lemma: Legendre_transform} implies that 
If $\mathcal{I}(\delta)$ is finite for some $\delta \in \mathbb{R}_{++}^2$, then $\mathcal{I}(\delta)$ is finite for all $\delta \in\mathbb{R}_{++}^2$ because $\mathcal{I}(\delta)$ is concave. Based on this, we give the following lemma that is used to show   the continuity of $\mathcal{I}$.

\begin{lemma}\label{lemma:I-inequality}
    Let $\delta \geq 0$, $\eta_0 > \eta \geq 0$. 
    Suppose that $\mathcal{I}(\delta) < \infty$ for some $\delta \in \mathbb{R}_{++}^2$.
    Under \Cref{assumption:f-bounded-below,assumption:finite-moments-I,assumption:metric-cost,assumption:quasi-metric,assumption:f-Psi-growth}, there is a constant $M >0$ such that
    \[
    \mathcal{I}(\eta_0, \delta) - \mathcal{I}(\eta, \delta) \leq \Psi_1\left( \eta_0 - \eta  ,  M   ( 1- \eta/ \eta_0)   \right) ,
    \]
    and 
    \[
    \mathcal{I}(\delta, \eta_0) - \mathcal{I}(\delta, \eta) \leq \Psi_2 \left( M   ( 1- \eta/ \eta_0)  , \eta_0-\eta \right).
    \]
    \end{lemma}

    \begin{proof}
For simplicity, assume that for any $\eta, \delta \geq 0$, one has $\gamma^{\eta, \delta}  = \arg \max_{\gamma \in \Sigma(\eta, \delta)} \int_{\mathcal{S} } f d\gamma$, equivalently, $\mathcal{I}(\eta, \delta) = \int_{\mathcal{S} }  f  d \gamma^{\eta, \delta}$.   Otherwise, due to the global continuity of $\Psi_j$, we can repeat the proof with an $\epsilon$-approximation argument and let $\epsilon \downarrow 0$.

For fixed $\eta_0 >0$ and $\delta >0$,  we have $\boldsymbol{K}_1(\gamma^{\eta_0, \delta}_{1,3}, \mu_1) \leq \eta_0$ and $ \boldsymbol{K}_2(\gamma^{\eta_0, \delta}_{2,3}, \mu_2) \leq \delta$  by the definition of $\gamma^{\eta_0, \delta}$.  Let $K_1 = \{ 1,2,3\}$, $K_2 = \{ 1,3,4,6 \}$ and $K_3 = \{ 5,6\}$ and it is easy to verify the collection $\{K_1, K_2, K_3 \}$ is decomposable. As a result, by \Cref{prop:MultiVBP}, we can construct random variables  
\[
(S, \widetilde{S} ) \equiv \left(Y_1, Y_2, X, \widetilde{Y}_1, \widetilde{Y}_2, \widetilde{X} \right) \in \mathcal{S} \times \mathcal{S},
\]
such that 
\[
 \text{Law}(Y_1, Y_2, X) =  \gamma^{\eta_0, \delta} , \quad \text{Law}( \widetilde{Y}_1 , \widetilde{X}) = \mu_1, \quad \text{Law}( \tilde{Y}_2 , \widetilde{X}) = \mu_2,
\]
and 
\[
\boldsymbol{K}_1\left(\gamma_{1,3}^{\eta_1, \delta}, \mu_1\right) = \mathbb{E} \left[ c_1(  S_1, \widetilde{S}_1 )  \right] \leq \eta_0, \quad \text{where } S_1 = (Y_1, X) \text{ and } \widetilde{S}_1 = (\widetilde{Y}_1, \widetilde{X}). 
\]
Let $\varepsilon$ be a Bernoulli random variable that is independent of $(S, \widetilde{S} )$ with $\mathbb{P}(\varepsilon = 1) =  \eta/ \eta_0$. Define new random variables:
\[
\widehat{S} \equiv ( \widehat{Y}_1,\widehat{Y}_2, \widehat{X} )  =   \varepsilon  (Y_1, Y_2,X) + (1-\varepsilon)   (\widetilde{Y}_1,\widetilde{Y}_2,  \widetilde{X} ),
\]
and let $\widehat{\gamma} =  \text{Law} ( \widehat{Y}_1, \widehat{Y}_2,  \widehat{X})$. For any measurable set $A \in \mathcal{B}_{ \mathcal{S} }$, we have
\[
\begin{aligned}
\widehat{\gamma}(A)  &  =  \mathbb{P} ( \widehat{S} \in A ) =  \mathbb{E} \left[   \mathbb{P} ( \widehat{S} \in A | \varepsilon  )  \right]    \\ 
& =   \left( \eta/ \eta_0 \right) \mathbb{P} ( S \in A)   +   \left(  1- \eta/ \eta_0 \right) \mathbb{P} ( \widetilde{S} \in A).
\end{aligned}
\]
This shows 
\[
\widehat{\gamma} =   \left( \eta/ \eta_0 \right)   \gamma^{\eta_0, \delta}  +     \left(  1- \eta/ \eta_0 \right)  \widetilde{\gamma}, \quad \text{where}   \   \widetilde{\gamma} = \text{Law}(\widetilde{Y}_1, \widetilde{Y}_2, \widetilde{X}).
\]
Next, we verify $\widehat{\gamma}  \in \Sigma(\eta,  \delta)$. Since $\nu \mapsto \boldsymbol{K}_1\left( \nu  , \mu_1\right)$ is convex and $\widetilde{\gamma}_{1,3} = \text{Law}(\widetilde{Y}_1, \widetilde{X}) = \mu_1$,  we have
\[
\begin{aligned}
\boldsymbol{K}_1( \widehat{\gamma}_{1,3}, \mu_1)  & \leq     \left( \frac{\eta}{\eta_0} \right)  \boldsymbol{K}_1(\gamma^{\eta_1, \delta}_{1,3}, \mu_1)  +    \left( 1- \frac{\eta}{\eta_0} \right)  \boldsymbol{K}_1(\widetilde{\gamma}_{1,3}, \mu_1)  \leq \eta. 
\end{aligned}
\]
Similarly, we have $\boldsymbol{K}_2( \widehat{\gamma}_{2,3}, \mu_2)  \leq   \delta$. As a result, we verify $\widehat{\gamma} \in \Sigma( \eta, \delta)$. Next, it is easy to see
\[
\begin{aligned}
& \mathbb{E}   \left[   c_1 \left(  ( \widehat{Y}_1, \widehat{X} ),  (Y_1,X )  \right) \right] 
  \leq   \left(1-\frac{\eta}{\eta_0}\right)  \mathbb{E}\left[c_1\left ( (\widetilde{Y}_1, \widetilde{X} ),(Y_1, X) \right)\right]    \\
\end{aligned} \leq (\eta - \eta_0)
\]

Since $\text{Law}(Y_2, X) = \gamma^{\eta_0, \delta}_2$, $\text{Law}(\widetilde{Y}_2, \widetilde{X} ) = \mu_2$ and  $\boldsymbol{K}_2\left(\gamma_{2,3}^{\eta_0, \delta}, \mu_2\right) \leq \delta$, i.e. $\boldsymbol{W}_{p_2} \left(\gamma_{2,3}^{\eta_0, \delta}, \mu_2\right) \leq \delta^{1/p_2}$, by triangle inequality, we have
\[
\boldsymbol{W}_{p_2} \left(\gamma_{2,3}^{\eta_1, \delta}, \delta_{s_2 } \right) \leq   \boldsymbol{W}_{p_2} \left(\gamma_{2,3}^{\eta_1, \delta}, \mu_2\right) +   \boldsymbol{W}_{p_2} \left(\mu_2 , \delta_{s_2 } \right)  \leq  \delta^{1/p_2} + \boldsymbol{W}_{p_2} \left(\mu_2 , \delta_{s_2 } \right),
\]
where $\delta_{s_2}$ denotes the dirac measure at $\{s_2\}$ and $s_2  \in \mathcal{S}_2$ is arbitrary. Further,  \Cref{assumption:quasi-metric-ii} implies $\rho_2(y^\prime_2, y_2) \leq  1 +  \boldsymbol{d}_{\mathcal{S}_2 } (s^\prime_2, s_2)^{p_2} $ for all $s_2 = (y_2,x)$  and $s^\prime_2 = (y^\prime_2,x^\prime)$,
\[
\begin{aligned}
\mathbb{E} \left[   \rho_2 (Y_2, y_2) \right]  - 1 &  \leq \mathbb{E} \left[  \boldsymbol{d}_{\mathcal{S}_2 } (S_2, s_2)^{p_2}  \right] = \boldsymbol{W}_{p_2} \left(\gamma_{2,3}^{\eta_1, \delta}, \delta_{s_2 } \right)^{p_2}  \leq  \left[ \delta^{1/p_2} + \boldsymbol{W}_{p_2} \left(\mu_2 , \delta_{s_2} \right) \right]^{p_2} ,\\
\end{aligned}
\]
and 
\[
\mathbb{E} \left[   \rho_2 ( \widetilde{Y}_2, y_2) \right] - 1 \leq   \mathbb{E} \left[  \boldsymbol{d}_{\mathcal{S}_2 } (\widetilde{S}_2, s_2)^{p_2}  \right] =   \boldsymbol{W}_{p_2} \left(\mu_2 , \delta_{s_2} \right)^{p_2}.
\]
As a result, by \Cref{assumption:quasi-metric-iii}, 
\[
\begin{aligned}
 \mathbb{E} \left[ \rho_2 (Y_2 , \widehat{Y}_2)   \right] & =  (\eta/ \eta_0)  \underbrace{\mathbb{E} \left[ \rho_2 (Y_2 , Y_2) \big  | \varepsilon = 0 \right]  }_{= 0} + ( 1- \eta/ \eta_0)  \mathbb{E} \left[ \rho_2 (Y_2 , \widetilde{Y}_2) \big  | \varepsilon = 1 \right] \\
 & \leq  (1- \eta/ \eta_0)   \mathbb{E}   \left[ \rho_2 (Y_2 , \widetilde{Y}_2)   \right]  \leq    (1- \eta/ \eta_0) N \left(  \mathbb{E}   \left[ \rho_2 (Y_2 ,y_2)   \right] + \mathbb{E}  \left[ \rho_2 (y_2 , \widetilde{Y}_2)   \right] \right) \\ 
 & \leq M  (1- \eta/ \eta_0),
 \end{aligned}
\]
where
\[
M =   N  \boldsymbol{W}_{p_2} \left(\mu_2 , \delta_{s_2} \right)^{p_2} +  N \left[ \delta^{1/p_2} + \boldsymbol{W}_{p_2} \left(\mu_2 , \delta_{s_2} \right) \right]^{p_2}  < \infty.
\]
Therefore, by \Cref{assumption:f-Psi-growth}, we have 
\[
\begin{aligned}
\mathcal{I}\left(\eta_0, \delta\right) - \mathcal{I}(\eta, \delta)
& \leq  \mathbb{E} \left[ f(Y_1, Y_2, X) \right] -   \mathbb{E} \left[ f(\widehat{Y}_1, \widehat{Y}_2, \widehat{X} ) \right]  \\
& \leq  \mathbb{E} \left[   \Psi\left(   c_1 \left( (Y_1, X),  (\widehat{Y}_1,  \widehat{X} )  \right),  \rho_2 (Y_2 , \widehat{Y}_2) \right)   \right] \\
& \leq \Psi\left(   \mathbb{E}  \left[ c_1 ( S_1 , \widehat{S}_1   ) \right],   \mathbb{E} \left[ \rho_2 (Y_2 , \widehat{Y}_2)   \right]    \right)   \\
& \leq \Psi\left( \eta_0- \eta   ,  M   ( 1- \eta/ \eta_0)   \right).
\end{aligned}
\]
The rest of the proof can be completed using the same reasoning.

    \end{proof}
    
Now, we give the proof of \Cref{thm:I-continuity}.
    \begin{proof}[{Proof of \Cref{thm:I-continuity}}]
If $\eta_0 >  \eta \geq 0$, \Cref{lemma:I-inequality} implies
\[
\begin{aligned}
0 \leq \mathcal{I}(\eta_0, \delta) - \mathcal{I}(\eta, 0) & =  \mathcal{I}(\eta_0, \delta) - \mathcal{I}(\eta, \delta) + \mathcal{I}(\eta, \delta)  - \mathcal{I}(\eta, 0)  \\
&  \leq  \Psi_1\left( \eta_0- \eta   ,  M   ( 1- \eta/ \eta_0)   \right) +    \Psi_2 \left(M\delta,   \delta \right).
\end{aligned}
\]
If $\eta \geq \eta_0$, by monotonicity of $\eta \mapsto  \mathcal{I}(\eta, 0)$ and \Cref{lemma:I-inequality},  we have
\[
 \mathcal{I}(\eta_0, \delta) - \mathcal{I}(\eta, 0) \leq  \mathcal{I}(\eta_0, \delta) - \mathcal{I}(\eta_0 , 0) \leq   \Psi_2\left( M \delta  ,  \delta  \right),
\]
and 
\[
 \mathcal{I}(\eta_0, \delta) - \mathcal{I}(\eta, 0) \geq   \mathcal{I}(\eta, \delta) - \mathcal{I}(\eta, 0) \geq 0 .
\]
As a result, we must have for all $\eta_0, \eta$ and $\delta$ in $[0, \infty)$, .
\[
 0 \leq \mathcal{I}(\eta_0, \delta) - \mathcal{I}(\eta, 0) \leq  \Psi_1\left( |\eta_0- \eta|   ,  M  | 1- \eta/ \eta_0|  \right)  + \Psi_2\left( M \delta  ,  \delta  \right).
\]
The continuity of $\mathcal{I}$ at $(\eta, 0)$ follows from the continuity of $\Psi_1$ and $\Psi_2$, and letting $(\eta_0, \delta) \rightarrow (\eta, 0)$. Using a similar argument, we can show  $\mathcal{I}$ is continuous at $(0,\eta)$.        \end{proof}

\subsection{Proofs in \texorpdfstring{\Cref{sec:examples-revisited}}{}}

\subsubsection{Proof of \texorpdfstring{\Cref{prop:I-ID-equivalence-separable}}{}}

By some simple algebra and \Cref{thm:ID-duality}, we have
\[
\begin{aligned}
\mathcal{I}_{\mathrm{D} }(\delta) & =   \inf_{\lambda \in \mathbb{R}^2_+}   \left\{ \langle \lambda, \delta \rangle + \sup_{\gamma \in \Pi(\mu_{1}, \mu_{2} )} \int_{\mathcal{S} } \left[   (f_1)_{\lambda_1}(y_1)  +  (f_2)_{\lambda_2}(y_2)  \right]  \, d\gamma(y_1, y_2)\right\} \\
& = \inf_{\lambda_1 \geq 0 } \left[ \lambda_1 \delta_1 +\int_{\mathcal{Y}_1}(f_1)_{\lambda_1}  \, d \mu_1  \right] + \inf_{\lambda_2 \geq 0 } \left[ \lambda_2 \delta_2 +\int_{\mathcal{Y}_2}  (f_2 )_{\lambda_2} \,  d \mu_2  \right],
\end{aligned}
\]
where the last step holds because $(f_{\ell})_{\lambda} \geq f_{\ell}$ and the right-hand side is well-defined since $f_{\ell} \in L^{1}(\mu_{\ell})$.  Next, we show $\mathcal{I}(\delta) = \mathcal{I}_{\mathrm{D} }(\delta)$.  \Cref{thm:I-duality} implies 
\[
\mathcal{I}(\delta) = \inf_{\lambda \in \mathbb{R}^2_+}   \left\{ \langle \lambda, \delta \rangle + \sup_{\pi \in \Pi(\mu_{13}, \mu_{23} )} \int_{\mathcal{S}_1 \times \mathcal{S}_2 }  (f_\mathcal{S} )_\lambda  \, d\pi \right\},
\]
where $\left(f_{\mathcal{S}}\right)_\lambda: \mathcal{S}_1 \times \mathcal{S}_2 \rightarrow \mathbb{R}$ is given by 
\[
\left(f_{\mathcal{S}}\right)_\lambda(s_1, s_2)  =   \sup_{ (y_1^\prime, y_2^\prime, x^\prime ) \in \mathcal{S} } \left\{ f_1(y_1^\prime) +  f_2(y_2^\prime) - \sum_{1 \leq \ell \leq 2} \lambda_\ell c_\ell\left(   (y_\ell, x_\ell), (y^\prime_\ell, x^\prime) \right)    \right\}.
\]
In fact,  \Cref{assumption: separable_cost} implies for all $s_\ell = (y_\ell, x_\ell) \in \mathcal{S}_\ell$ and $s_\ell^\prime = (y^\prime_\ell, x^\prime_\ell) \in \mathcal{S}_\ell$, one has
\[
c_{Y_{\ell}}\left(y_{\ell}, y_{\ell}^{\prime}\right)  = \inf_{x_\ell, x_\ell^\prime \in \mathcal{X} } c_\ell\left(   (y_\ell, x_\ell), (y^\prime_\ell, x^\prime_\ell) \right)    \leq c_\ell\left(   (y_\ell, x_\ell), (y^\prime_\ell, x^\prime_\ell) \right)    .
\]
Recall $(f_{\mathcal{S}})_\lambda: (s_1, s_2) \mapsto (f_{\mathcal{S}})_\lambda (s_1, s_2)$ is a function from $\mathcal{S}_1 \times \mathcal{S}_2 \rightarrow \mathbb{R}$ with $s_\ell = (y_\ell, x_\ell) \in \mathcal{S}_\ell$. As a result, for all $s_1 \in \mathcal{S}_1$ and $s_2 \in \mathcal{S}_2$
\[
\begin{aligned}
(f_{\mathcal{S}})_\lambda(s_1, s_2) &  \leq       \sup_{ (y_1^\prime, y_2^\prime, x^\prime ) \in \mathcal{S} } \left\{ f_1(y_1^\prime) +  f_2(y_2^\prime) - \sum_{1 \leq \ell \leq 2} \lambda_\ell c_{Y_\ell}  \left(  y_\ell, y_\ell^\prime \right)    \right\} \\
&= (f_1 )_{\lambda_1}(y_1) +    (f_2 )_{\lambda_2}(y_2).
\end{aligned}
\]
This shows for all $\lambda = (\lambda_1, \lambda_2) \in \mathbb{R}^2_+$, one has
\[
\sup _{\pi \in \Pi\left(\mu_{13}, \mu_{23}\right)} \int_{\mathcal{S}_1 \times \mathcal{S}_2} \left(f_{\mathcal{S}}\right)_\lambda \, d \pi \leq  \sup_{\gamma \in \Pi(\mu_{1}, \mu_{2} )} \int_{\mathcal{Y}_1 \times \mathcal{Y}_2 } \left[   (f_1)_{\lambda_1}(y_1)  +  (f_2)_{\lambda_2}(y_2)  \right]  \, d\gamma(y_1, y_2),
\]
and hence $\mathcal{I}(\delta) \leq   \mathcal{I}_{\mathrm{D} }(\delta)$. We end the proof by showing
\begin{align*}
\sup _{\pi \in \Pi\left(\mu_{13}, \mu_{23}\right)} \int_{\mathcal{S}_1 \times \mathcal{S}_2}  (f_{\mathcal{S}} )_\lambda d \pi \geq \int_{\mathcal{Y}_1}(f_1)_{\lambda_1} d \mu_1 +\int_{\mathcal{Y}_2} (f_2 )_{\lambda_2} d \mu_2.
\end{align*}
It suffices to show that there is $\pi \in \Pi(\mu_{13}, \mu_{23})$ such that $(f_{\mathcal{S}} )_\lambda ( s_1, s_2) \geq  (f_1)_{\lambda_1}(y_1) +  (f_2 )_{\lambda_2}(y_2)$, $\pi$-a.e. In fact, we note that if $x_1 = x_2$, then $(f_{\mathcal{S}} )_\lambda (  (y_1, x_1), (y_2, x_2) ) =  (f_1)_{\lambda_1}(y_1) +  (f_2 )_{\lambda_2}(y_2)$ under \Cref{assumption: separable_cost}. Consider a probability measure $\pi^\star = \mathrm{Law}(Y_1, X, Y_2,  X)$ where $\mu_{\ell, 3} = \mathrm{Law}(Y_\ell, X)$ for $\ell = 1,2$. As a result, 
\begin{align*}
\sup _{\pi \in \Pi\left(\mu_{13}, \mu_{23}\right)} \int_{\mathcal{S}_1 \times \mathcal{S}_1}\left(f_{\mathcal{S}}\right)_\lambda \, d \pi & \geq  \int_{\mathcal{S}_1 \times \mathcal{S}_2}\left(f_{\mathcal{S}}\right)_\lambda \, d \pi^\star = \int_{\mathcal{S}_1 \times \mathcal{S}_2} \left[ (f_1)_{\lambda_1} +  (f_2 )_{\lambda_2}  \right] \, d\pi^\star \\
& =   \int_{\mathcal{Y}_1}(f_1)_{\lambda_1} \, d \mu_1 +\int_{\mathcal{Y}_2} (f_2 )_{\lambda_2} \, d \mu_2.
\end{align*}

\subsubsection{Proof of \texorpdfstring{\Cref{prop:ATE-Mahalanobis}}{}}

Since $c_{Y_\ell}(y_\ell, y_\ell') =  \inf_{x_\ell, x_\ell' \in \mathcal{X}_\ell} c_{\ell} (s_\ell, s_\ell')$, the proof of \Cref{prop:I-ID-equivalence-separable} implies $\mathcal{I}(\delta) \leq \mathcal{I}_{\mathrm{D}}(\delta)$.

\subsubsection{Proof of \texorpdfstring{\Cref{prop:ATE-Mahalanobis-Dualform_i}}{}}

The proof consists of two steps. In Step 1, we derive the dual form of $\mathcal{I}_{\mathrm{D}}(\delta)$ and $\mathcal{I}(\delta)$ for $\delta \in \mathbb{R}^2_{++}$. In Step 2, we derive the dual reformulations of $\mathcal{I}_{\mathrm{D}}(\delta)$ and $\mathcal{I}(\delta)$ for $\delta \in \mathbb{R}^2_{+} \setminus \mathbb{R}^2_{++}$. 

\textbf{Step 1.} We derive the expressions of $\mathcal{I}_{\mathrm{D}}(\delta)$ and $\mathcal{I}(\delta)$ for $\delta \in \mathbb{R}^2_{++}$. First, recall $c_{Y_{\ell}} (y_{\ell}, y_{\ell}^{\prime}) =  V_{\ell, Y Y}^{-1} (y_{\ell}-y_{\ell}^{\prime})^2$.  \Cref{thm:ID-duality} implies 
\[
\mathcal{I}_{\mathrm{D} }(\delta)=\inf _{\lambda \in \mathbb{R}_{+}^2}\left[\langle\lambda, \delta\rangle+\sup _{\varpi \in \Pi\left(\mu_{Y_1}, \mu_{Y_2}\right)} \int_{\mathbb{R}^2} (f_{\mathcal{Y}})_{\lambda}(y_1, y_2) \, d \varpi(y_1, y_2)\right],
\]
where $(f_{\mathcal{Y}})_{\lambda}: (y_1, y_2)  \mapsto (f_{\mathcal{Y}})_{\lambda}(y_1, y_2)$ from $\mathbb{R}^2$ to $\mathbb{R}$ is given by
\[
(f_{\mathcal{Y}})_{\lambda}(y_1, y_2)  = y_2 - y_1 + \frac{V_{1, YY}}{4 \lambda_1} + \frac{V_{2, YY}}{4 \lambda_2}.
\]
Since $V_{\ell, YY} > 0$ for $\ell \in [2]$, by some simple algebra,  we have for all $\delta \in \mathbb{R}_{++}^2$
\begin{align*}
    \mathcal{I}_{\mathrm{D}}(\delta) = \mathbb{E}[Y_2] - \mathbb{E}[Y_1] + V_{1, YY}^{1/2} \; \delta_1^{1/2} + V_{2, YY}^{1/2}\;  \delta_2^{1/2} 
\end{align*}
Next, we derive the expression of $\mathcal{I}(\delta)$ for $\delta \in \mathbb{R}^2_{++}$. Let $Q_\ell \in \mathbb{R}^{(d+1) \times (d+1)}$ be the inverse of $V_\ell$, i.e.,
    \begin{align*}
    Q_{\ell} =    \begin{bmatrix}
            Q_{\ell, YY} & Q_{\ell, YX} \\
            Q_{\ell, XY} & Q_{\ell, XX}
       \end{bmatrix} 
       = 
       \begin{bmatrix}
        (V_{\ell} / V_{\ell, XX})^{-1} & - (V_{\ell} / V_{\ell, XX})^{-1} V_{\ell, YX} V_{\ell, XX}^{-1} \\
        - V_{\ell, XX}^{-1} V_{\ell, XY} (V_{\ell} / V_{\ell, XX})^{-1} &  (V_{\ell} / V_{\ell, YY})^{-1}
        \end{bmatrix},
    \end{align*}
    where $V_{\ell} / V_{\ell, XX} = V_{\ell, YY} - V_{\ell, YX} V_{\ell, XX}^{-1} V_{\ell, XY}$ and $V_{\ell} / V_{\ell, YY} = V_{\ell, XX} - V_{\ell, XY} V_{\ell, YY}^{-1} V_{\ell, YX}$.  Conversely,
	\begin{align*}
		\begin{bmatrix}
            V_{\ell, YY} & V_{\ell, YX} \\
            V_{\ell, XY} & V_{\ell, XX}
       \end{bmatrix} 
   & =
   \begin{bmatrix}
		(Q_{\ell} / Q_{\ell, XX})^{-1} & - Q_{\ell, YY}^{-1} Q_{\ell, YX} (Q_{\ell} / Q_{\ell, YY})^{-1} \\
		- (Q_{\ell} / Q_{\ell, YY})^{-1} Q_{\ell, XY} Q_{\ell, YY}^{-1} &  (Q_{\ell} / Q_{\ell, YY})^{-1}
   \end{bmatrix} ,
	\end{align*}
    where $Q_{\ell} / Q_{\ell, XX} = Q_{\ell, YY} - Q_{\ell, YX} Q_{\ell, XX}^{-1} Q_{\ell, XY}$ and $Q_{\ell} / Q_{\ell, YY} = Q_{\ell, XX} - Q_{\ell, XY} Q_{\ell, YY}^{-1} Q_{\ell, YX}$. Next, we evaluate the function $(f_{\mathcal{S}})_{\lambda}(s_1, s_2)$ that appears in the dual reformulation. For simplicity, we write $a_1 = -1$ and $a_2 = 1$. Consider the following derivation:
    \begin{align*}
       \left(f_{\mathcal{S}}\right)_\lambda (s_1, s_2) 
        & := 
        \sup_{y_1', y_2', x'}
        \left\{y_2' - y_1' - \sum_{\ell = 1, 2} \lambda_{\ell} c_{\ell}((y_{\ell}', x'), (y_\ell, x_{\ell}))\right\} \\
        & = 
        \sup_{y_1', y_2', x'}
        \left\{
        \sum_{1 \leq \ell \leq 2}
        \left( a_\ell y_\ell -   \lambda_\ell
        \begin{bmatrix}
            y_\ell' - y_\ell \\
            x' - x_\ell
        \end{bmatrix}^{\top}
  Q_{\ell} 
            \begin{bmatrix}
                y_\ell' - y_\ell \\
                x' - x_\ell
            \end{bmatrix} \right) \right\}
        \\
        & =_{(1)}
        y_2 - y_1
        +
        \sup_{z_1', z_2', x'}
        \left\{
        \sum_{1 \leq \ell \leq 2}
        \left( a_\ell z_\ell' -   \lambda_\ell 
        \begin{bmatrix}
            z_\ell' \\
            x' - x_\ell
        \end{bmatrix}^{\top}
           Q_{\ell} 
            \begin{bmatrix}
                z_\ell' \\
                x' - x_\ell
            \end{bmatrix} \right) \right\} \\
        & =
        y_2 - y_1
        +
        \sup_{x'\in \mathbb{R}^d }
        \left\{
        \sum_{1 \leq \ell \leq 2} 
        \sup_{z_\ell' \in \mathbb{R}}
        \left( a_\ell z_\ell' -   \lambda_\ell 
        \begin{bmatrix}
            z_\ell' \\
            x' - x_\ell
        \end{bmatrix}^{\top}
	Q_{\ell}
            \begin{bmatrix}
                z_\ell' \\
                x' - x_\ell
            \end{bmatrix} \right) \right\},
    \end{align*}
where equation (1) follows from the change of variables $z^\prime_\ell = y_\ell^\prime - y_\ell$. So, to evaluate $\left(f_{\mathcal{S}}\right)_\lambda (s_1, s_2)$, it suffices to maximize $(z_1^\prime, z_2^\prime, x^\prime) \mapsto \phi_{1}(z_1', x'; x_1) + \phi_{2}(z_2', x'; x_2)$ where 
    \begin{align*}
  \phi_{\ell}(z_\ell', x'; x_\ell)
        & =
        a_\ell z_\ell' -   \lambda_\ell \
        \begin{bmatrix}
            z_\ell' \\
            x' - x_\ell
        \end{bmatrix}^{\top}
	Q_{\ell} 
            \begin{bmatrix}
                z_\ell' \\
                x' - x_\ell
            \end{bmatrix}.
    \end{align*}
We first consider $\sup_{z_\ell' \in \mathbb{R} } \phi_{\ell}(z_\ell', x'; x_\ell)$. The first-order conditions imply that the optimal solution is 
	\begin{align*}
		z_\ell' = (\lambda_\ell Q_{\ell, YY})^{-1} \left[ \frac{a_\ell}{2}  - \lambda_\ell Q_{\ell, YX} (x' - x_\ell) \right].
	\end{align*}
By some simple algebra, $\sup_{z_\ell' \in \mathbb{R} } \phi_{\ell}(z_\ell', x', x_\ell) = 		  \varphi_\ell (x^\prime- x_\ell, \lambda)$ where $\varphi_\ell: \mathbb{R}^d \times \mathbb{R} \rightarrow \mathbb{R}$ is given by
\[
\varphi_\ell (x, \lambda_\ell) = \frac{Q_{\ell, YY}^{-1} }{4\lambda_{\ell}}   +    a_\ell   x^{\top} V_{\ell, XX}^{-1} V_{\ell, XY} - \lambda_{\ell} x^{\top} V_{\ell, XX}^{-1} x.
\]
As result, 
\[
(f_{\mathcal{S}})_\lambda (s_1, s_2 ) = \sup_{x^\prime \in \mathbb{R}^d}  \left[ \varphi_1 (x^\prime- x_1, \lambda_1)  + \varphi_2 (x^\prime- x_2, \lambda_2)  \right].
\]
Now, we consider the optimization above. The first-order conditions imply the optimal solution $x^\prime$ takes the form of $x^\prime - x_\ell = B_\ell (x_2 - x_1) + b_\ell$ for some $B_\ell \in \mathbb{R}^{d \times d}$ and $b_\ell \in \mathbb{R}^d$ that depend on $\lambda_\ell$. So,  we have
\[
\sup_{x^\prime \in \mathbb{R}^d}  \left[ \varphi_1 (x^\prime, x_1)  + \varphi_2 (x^\prime, x_2)  \right] = b + B(x_1- x_2) - (x_1-x_2)^\top W (x_1-x_2).
\]
for some positive definite matrix  $W \in \mathbb{R}^{d \times d}$ and $b \in \mathbb{R}$  that depend on $\lambda_1,\lambda_2, x_1$ and $x_2$.
Here, the constant $b$ will be determined below. For any $\pi \in \Pi(\mu_{13}, \mu_{23})$, we have 
    \begin{align*}
\int_{\mathbb{R}^{d+1} \times \mathbb{R}^{d+1} } (f_{\mathcal{S}})_\lambda \, d\pi &= \frac{1}{4 \lambda_1} Q_{1, Y Y}^{-1}+\frac{1}{4 \lambda_2} Q_{2, Y Y}^{-1}  + \underbrace{ \int_{\mathbb{R}^{d+1} \times \mathbb{R}^{d+1} }  B (x_1-x_2) \, d\pi }_{= 0} \\
& +  \int_{\mathbb{R}^{d+1} \times \mathbb{R}^{d+1} } (x_1-x_2)^{\top} W(x_1-x_2) \, d \pi(s_1, s_2) + b \\
& =  \frac{1}{4 \lambda_1} Q_{1, Y Y}^{-1}+\frac{1}{4 \lambda_2} Q_{2, Y Y}^{-1} -  \int (x_1-x_2)^{\top} W(x_1-x_2) \, d \pi + b.
\end{align*}
Now, let us consider $\sup_{\pi \in \Pi(\mu_{13}, \mu_{23})} \int (f_{\mathcal{S}})_\lambda \, d \pi$. To maximize $\int (f_{\mathcal{S}})_\lambda \,  d \pi$, it suffices to consider 
\[
\inf_{\pi \in  \Pi(\mu_{13}, \mu_{23})  } \int_{\mathbb{R}^{d+1} \times \mathbb{R}^{d+1} }  (x_1-x_2)^{\top} W (x_1-x_2) \, d \pi(s_1, s_2).
\]
Since $(x_1-x_2)^{\top} W (x_1-x_2)$ for all $x_1, x_2 \in \mathbb{R}^d$, the probability measure $\pi = \mathrm{Law}(Y_1,X, Y_2, X)$ with $\mathrm{Law}(Y_\ell,X) = \mu_{\ell,3}$ for $\ell=1,2$ is a solution and the optimal value is 0. We denote by $\Pi$ the set of all probability measures on $\mathcal{S}_1 \times \mathcal{S}_2$ that takes forms of 
$\pi = \mathrm{Law}(Y_1,X, Y_2, X)$. As a consequence,
\[
\sup_{\pi \in \Pi\left(\mu_{13}, \mu_{23}\right)} \int_{\mathbb{R}^{2d +2} }  (f_{\mathcal{S}} )_\lambda \, d \pi = \frac{1}{4 \lambda_1} Q_{1, Y Y}^{-1}+\frac{1}{4 \lambda_2} Q_{2, Y Y}^{-1} + b 
\]
where $b = \frac{1}{4} V_o^{\top}\left(\lambda_1 V_{1, X X}^{-1}+\lambda_2 V_{2, X X}^{-1}\right)^{-1} V_o$ with $V_o=V_{2, X X}^{-1} V_{2, X Y}-V_{1, X X}^{-1} V_{1, X Y}$. As a result, the dual reformulation of $\mathcal{I}_{\mathrm{D} }(\delta)$ is given by 
 \begin{align*}
\mathcal{I}(\delta) = \mathbb{E}\left[Y_2\right]-\mathbb{E}\left[Y_1\right]    +
        \inf_{\lambda \in \mathbb{R}_{+}^2}
        \Bigg\{
         \lambda_1 \delta_1 + \lambda_2 \delta_2 &  +
        \frac{1}{4\lambda_{1}} \left(V_{1} / V_{1, XX} \right) + \frac{1}{4\lambda_{2}}  \left(V_{2} / V_{2, XX} \right) \\
        &  + 
        \frac{1}{4}
        V_o^{\top} 
        \left(\lambda_1 V_{1, XX}^{-1} + \lambda_2 V_{2, XX}^{-1}\right)^{-1} V_o
        \Bigg\}.
\end{align*}

\textbf{Step 2.} We derive the dual reformulation of $\mathcal{I}_{\mathrm{D}}(\delta)$ and $\mathcal{I}(\delta)$ for $\delta\in \mathbb{R}^2_+ \setminus \mathbb{R}^2_{++}$. First, we note that $\mathcal{I}_{\mathrm{D}}(0) = \mathcal{I}(0) = \mathbb{E}[Y_2] - \mathbb{E}[Y_1]$. \Cref{thm:ID-duality} implies that 
\begin{align*}
    \mathcal{I}_{\mathrm{D}}(\delta_1, 0) & = 
   \inf _{\lambda \in \mathbb{R}_{+}^2}\left[\lambda_1 \delta_1 + \sup _{\varpi \in \Pi(\mu_{Y_1}, \mu_{Y_2})} \int_{\mathbb{R}^2} (f_{\mathcal{Y}})_{\lambda, 1}(y_1, y_2) \, d \varpi(y_1, y_2)\right], \\
    \mathcal{I}_{\mathrm{D}}(0, \delta_2) & = 
    \inf_{\lambda_2 \in \mathbb{R}_{+}^2}
   \left[\lambda_2 \delta_2  + \sup _{\varpi \in \Pi(\mu_{Y_1}, \mu_{Y_2})} \int_{\mathbb{R}^2} (f_{\mathcal{Y}})_{\lambda, 2}(y_1, y_2) \, d \varpi(y_1, y_2)\right],
\end{align*}
where $(f_{\mathcal{Y}})_{\lambda, \ell}$, for $\ell = 1, 2$, is given by $(f_{\mathcal{Y}})_{\lambda, \ell}  = y_2 - y_1 +  (4 \lambda_{\ell})^{-1} V_{\ell, YY} $. Since $V_{\ell, YY} > 0$, by simple algebra,  we have for all $\delta \in \mathbb{R}_{++}^2$,
\begin{align*}
    \mathcal{I}_{\mathrm{D}}(\delta_1, 0)  = \mathbb{E}[Y_2] - \mathbb{E}[Y_1] + V_{1, YY}^{1/2} \delta_1^2  \quad \text{and} \quad 
    \mathcal{I}_{\mathrm{D}}(0, \delta_2)  = \mathbb{E}[Y_2] - \mathbb{E}[Y_1] + V_{2, YY}^{1/2} \delta_2^2.
\end{align*}
\Cref{thm:I-duality} implies that 
\begin{align*}
    \mathcal{I}(\delta_1, 0) & = 
   \inf _{\lambda \in \mathbb{R}_{+}^2}\left[\langle\lambda, \delta\rangle+\sup _{\varpi \in \Pi\left(\mu_{13}, \mu_{23}\right)} \int_{\mathbb{R}^2} (f_{\mathcal{S}})_{\lambda, 1}(y_1, y_2) \, d \varpi(y_1, y_2)\right], \\
    \mathcal{I}(0, \delta_2) & = 
    \inf_{\lambda_1 \in \mathbb{R}_{+}^2}
   \left[\langle\lambda, \delta\rangle+\sup _{\varpi \in \Pi\left(\mu_{13}, \mu_{23}\right)} \int_{\mathbb{R}^2} (f_{\mathcal{S}})_{\lambda, 2}(y_1, y_2) \, d \varpi(y_1, y_2)\right],
\end{align*}
where $(f_{\mathcal{Y}})_{\lambda, \ell}$, for $\ell = 1, 2$, is given by
\begin{align*}
    (f_{\mathcal{Y}})_{\lambda, 1} 
    & =
    \sup_{y_1'} \left\{y_2 - y_1' - \lambda_1 \begin{bmatrix}
        y_1' - y_1 \\ x_2 - x_1 
    \end{bmatrix}^{\top} 
    Q_{1} \begin{bmatrix}
        y_1' - y_1 \\ x_2 - x_1 
    \end{bmatrix}\right\}, \\
     (f_{\mathcal{Y}})_{\lambda, 2} 
    & =
    \sup_{y_2'} \left\{y_2' - y_1 - \lambda_2 \begin{bmatrix}
        y_2' - y_2 \\ x_1 - x_2 
    \end{bmatrix}^{\top} 
    Q_{2} \begin{bmatrix}
        y_2' - y_2 \\ x_1 - x_2 
    \end{bmatrix}\right\}.
\end{align*}
With similar calculation as in Step 1, the functions $(f_{\mathcal{Y}})_{\lambda, 1}$ and $(f_{\mathcal{Y}})_{\lambda, 2}$ can be written as
\begin{align*}
    (f_{\mathcal{Y}})_{\lambda, 1} & = y_2 - y_1 + \frac{V_{1}/V_{1, XX}}{4 \lambda_1} - (x_2 - x_1)^{\top} V_{1, XX}^{-1} V_{1, XY} - \lambda_1 (x_2 - x_1)^{\top} V_{1, XX}^{-1} (x_2 - x_1), \\
    (f_{\mathcal{Y}})_{\lambda, 2} & = y_2 - y_1 + \frac{V_{2}/V_{2, XX}}{4 \lambda_2} + (x_1 - x_2)^{\top} V_{2 XX}^{-1} V_{2, XY} - \lambda_2 (x_1 - x_2)^{\top} V_{2, XX}^{-1} (x_1 - x_2).
\end{align*}
With the same reasoning as in Step 1, we have 
\begin{align*}
    \sup_{\varpi \in \Pi(\mu_{13}, \mu_{23})} \int (f_{\mathcal{S}})_{\lambda, \ell} \, d \varpi 
     =  \mathbb{E}[Y_2] - \mathbb{E}[Y_1] + \frac{V_{\ell}/V_{\ell, XX}}{4 \lambda_\ell},  \quad   \text{for }  \ell \in [2].
\end{align*}
Therefore,
\begin{align*}
     \mathcal{I}(\delta_1, 0) & = 
   \mathbb{E}[Y_2] - \mathbb{E}[Y_1] + (V_{1}/V_{1, XX})^{1/2} \delta_1^{1/2} = \mathcal{I}_{\mathrm{D}} (\delta_1, 0 ), \\
    \mathcal{I}(0, \delta_2) & =  
   \mathbb{E}[Y_2] - \mathbb{E}[Y_1] + (V_{2}/V_{2, XX})^{1/2} \delta_2^{1/2} = \mathcal{I}_{\mathrm{D}} (0, \delta_2 ).
\end{align*}

\subsubsection{Proof of \texorpdfstring{\Cref{prop:ATE-Mahalanobis-Dualform_ii}}{}}

Recall the proof of \texorpdfstring{\Cref{prop:ATE-Mahalanobis-Dualform_i}}{}, we have
\[
\mathcal{I}(\delta) =  \inf_{\lambda \in \mathbb{R}^2_+} \left\{  \langle \lambda, \delta \rangle +   \sup_{\pi \in \widetilde{\mathit{\Pi} }} \int_{ \mathbb{R}^{2d+2} } \left(f_{\mathcal{S}}\right)_\lambda \, d \pi \right\}
\]
where $\widetilde{\mathit{\Pi} }$ is the set of all probability measures such that their supports $\mathrm{Supp}(\pi)$ are in  $\left\{  (y_1, x_1, y_2, x_2)\in \mathbb{R}^{2d+2} : x_1 = x_2 \right\}$. By the definition of $\widetilde{\mathit{\Pi} }$, to evaluate $\mathcal{I}(\delta)$, it suffices to restrict the domain of  $\left(f_{\mathcal{S}}\right)_\lambda$ on $\mathrm{Supp}(\pi)$. 
For any $(s_1, s_2) \in \mathrm{Supp}(\pi)$, we have $x_1 = x_2$
\[
\begin{aligned}
(f_{\mathcal{S}})_\lambda(s_1, s_2) &=  (y_2-y_1) + \sup _{x^{\prime} \in \mathbb{R}^d}\left[\varphi_1 (x^{\prime}-x_1, \lambda_1 )+\varphi_2 (x^{\prime}-x_2, \lambda_2)\right] \\
& = (y_2-y_1) + \underbrace{ \sup_{x^\prime \in \mathbb{R}^d }  \left\{  \sum_{1\leq \ell \leq 2} \frac{ Q_{\ell, Y Y}^{-1} }{4 \lambda_{\ell}}  +{x^{\prime}}^{\top} V_{\ell, X X}^{-1} V_{\ell, X Y} a_{\ell}-\lambda_{\ell} {x^{\prime}}^{\top} V_{\ell, X X}^{-1} x^{\prime}  \right\} }_{= \mathcal{H}(\lambda, \delta) }
\end{aligned}
\]
As a consequence, $(f_{\mathcal{S}})_\lambda(s_1, s_2)$ is independent of $x_1$ and $x_2$ for all $(s_1, s_2)\in \mathrm{Supp}(\pi)$, and hence for all $\pi \in \widetilde{\mathit{\Pi} }$, we have
\[
\int_{\mathbb{R}^{2 d+2}}\left(f_{\mathcal{S}}\right)_\lambda d \pi = \mathbb{E}[Y_2] - \mathbb{E}[Y_1]  + \mathrm{R}(\lambda, \delta),
\]
where $\mathrm{R}(\lambda, \delta) = \mathcal{H}(\lambda, \delta) +\langle\lambda, \delta\rangle$ and $\mathrm{Law}(Y_\ell, X) = \mu_{\ell 3}$ for $\ell =1, 2$. So, $\mathcal{I}(\delta) = \mathbb{E}[Y_2] - \mathbb{E}[Y_1]  + \inf _{\lambda \in \mathbb{R}_{+}^2} \mathrm{R}(\lambda, \delta)$. Moreover, 
\[
\mathcal{I}_{\mathrm{D} }(\delta)=\mathbb{E}\left[Y_2\right]-\mathbb{E}\left[Y_1\right]+\inf _{\lambda \in \mathbb{R}_{+}^2} \mathrm{R}_{\mathrm{D}}(\lambda, \delta),
\]
where
\[
\mathrm{R}_{\mathrm{D}}(\lambda, \delta) =  \langle\lambda, \delta\rangle +\frac{V_{1, Y Y}}{4 \lambda_1}+\frac{V_{2, Y Y}}{4 \lambda_2}.
\]
The rest of the proof is divided into the following two steps.

{\bf Step 1.} We show that $\mathcal{I}_\mathrm{D}(\delta)=\mathcal{I}(\delta)$ implies
\begin{align}
    \delta_1^{1 / 2} V_{1, Y Y}^{-1 / 2} V_{1, X Y} + \delta_2^{1 / 2} V_{2, Y Y}^{-1 / 2} V_{2, X Y} = 0. \label{eq:ATE-I_D-I-equality-condition}
\end{align}
Since $Q_{\ell, Y Y} \geq  V_{\ell, Y Y}^{-1}$ by definition, then $Q^{-1}_{\ell, Y Y}\leq V_{\ell, Y Y}$ and $\mathrm{R}(\lambda, \delta) \leq \mathrm{R}_{\mathrm{D}}(\lambda, \delta)$. Let $\lambda_{\mathrm{D}}^\star = (\delta_1^{-1 / 2} V_{1, Y Y}^{1 / 2}, \delta_2^{-1 / 2} V_{2, Y Y}^{1 / 2} )$. It is easy to see $\inf _{\lambda \in \mathbb{R}_{+}^2} \mathrm{R}_{\mathrm{D}}(\lambda, \delta) =  \mathrm{R}_{\mathrm{D}}(\lambda_{\mathrm{D}}^\star, \delta) \geq  \mathrm{R} (\lambda_{\mathrm{D}}^\star, \delta)$ and hence
\[
\mathcal{I}(\delta) \leq  \mathbb{E}[Y_2] - \mathbb{E}[Y_1]  + \mathrm{R}(\lambda_{\mathrm{D} }^\star, \delta) \leq \mathbb{E}[Y_2] - \mathbb{E}[Y_1]  +   \mathrm{R}_{\mathrm{D}} (\lambda_{\mathrm{D}}^{\star}, \delta) = \mathcal{I}_{ \mathrm{D} }(\delta).
\]
Thus, $\mathcal{I}(\delta) = \mathcal{I}_{\mathrm{D} }(\delta)$ implies $\mathrm{R}_{\mathrm{D}}(\lambda_{\mathrm{D}}^\star, \delta) = \mathrm{R}(\lambda_{\mathrm{D}}^\star, \delta)$. In fact, we note that
\[
\mathrm{R}_{\mathrm{D}}(\lambda, \delta) =  \langle\lambda, \delta\rangle+   \sup _{x^{\prime} \in \mathbb{R}^d}  \left[ \sum_{1\leq \ell \leq 2} \varphi_\ell(x^{\prime}, \lambda_\ell)\right] \quad \text{and} \quad \mathrm{R}_{\mathrm{D}}(\lambda, \delta) =  \langle\lambda, \delta\rangle+ \sum_{1\leq \ell \leq 2}  \sup _{x^{\prime} \in \mathbb{R}^d}\varphi_\ell(x^{\prime}, \lambda_\ell).
\]
Since $x^\prime \mapsto \varphi_{\ell}(x^{\prime}, \lambda_{\ell})$ is strictly concave, it admits a unique maximizer and hence  $\mathrm{R}_{\mathrm{D}}(\lambda_{\mathrm{D}}^\star, \delta) = \mathrm{R}(\lambda_{\mathrm{D}}^\star, \delta)$ implies for $\ell=1,2$,
\[
\underset{x^\prime \in \mathbb{R}^d}{\arg \max } \left[ \sum_{1\leq \ell \leq 2} \varphi_\ell(x^{\prime}, \lambda_{\mathrm{D}, \ell }^{\star}  )\right] =  \underset{x^\prime \in \mathbb{R}^d}{\arg \max} \ \varphi_\ell(x^{\prime},\lambda_{\mathrm{D}, \ell }^{\star}).
\]
The first-order conditions imply
\[
\underset{x^\prime \in \mathbb{R}^d}{\arg \max } \left[ \sum_{1\leq \ell \leq 2} \varphi_\ell(x^{\prime}, \lambda_\ell)\right] = \left(\sum_{1 \leq \ell \leq 2} \lambda_{\ell}  V_{\ell, X X}^{-1}\right)^{-1}  \left(\sum_{1 \leq \ell \leq 2} a_{\ell} V_{\ell, X X}^{-1} V_{\ell, X Y}\right),
\]
and
\[
\underset{x^{\prime} \in \mathbb{R}^d}{\arg \max } \ \varphi_{\ell}(x^{\prime}, \lambda_{\ell}) = \frac{1}{2} a_\ell  {\lambda_{\mathrm{D}, \ell}^{\star}   }^{-1}    V_{2, X Y},  \quad \text{ for } \ell = 1,2.
\]
So, recall $\lambda_{\mathrm{D}, \ell}^{\star} = \delta_{\ell}^{-1 / 2} V_{\ell, Y Y}^{1 / 2}$, $a_1 = -1$ and $a_2 = 1$, we have 
\[
\delta_1^{1 / 2} V_{1, Y Y}^{-1 / 2} V_{1, X Y}+\delta_2^{1 / 2} V_{2, Y Y}^{-1 / 2} V_{2, X Y}=0.
\]

{\bf Step 2.} We show $\delta_1^{1 / 2} V_{1, Y Y}^{-1 / 2} V_{1, X Y}+\delta_2^{1 / 2} V_{2, Y Y}^{-1 / 2} V_{2, X Y}=0$ implies  $\mathcal{I}_\mathrm{D}(\delta)=\mathcal{I}(\delta)$.  We note $\lambda \mapsto \mathrm{R}_{\mathrm{D}} (\lambda, \delta )$ is convex since it is supremum of a set of affine functions. It can be written as
\[
\mathrm{R}_{\mathrm{D}} (\lambda, \delta ) = \langle\lambda, \delta\rangle + \sum_{1\leq \ell \leq 2 }  \frac{V_\ell / V_{\ell, X X} }{4 \lambda_\ell} + \frac{1}{4} V_o^{\top} { \underbrace{\left(\lambda_1 V_{1, X X}^{-1}+\lambda_2 V_{2, X X}^{-1}\right) }_{= \Lambda_\lambda} }^{-1} V_o
\]
Taking derivatives with respect to $\lambda_\ell$ yields
\[
 \frac{\partial  \mathrm{R}_{\mathrm{D}} (\lambda, \delta )  }{\partial \lambda_\ell}= \delta_{\ell}-   \frac{V_\ell / V_{\ell, X X} }{4 \lambda_\ell^2}  -  \frac{1}{4} V_o^{\top}   \Lambda_\lambda^{-1}  V_{\ell, X X}^{-1}  \Lambda_\lambda^{-1}  V_o.
 \]
By some algebra and under $\delta_1^{1 / 2} V_{1, Y Y}^{-1 / 2} V_{1, X Y}+\delta_2^{1 / 2} V_{2, Y Y}^{-1 / 2} V_{2, X Y}=0$, we can show
\[
\frac{\partial  \mathrm{R}_{\mathrm{D}} (\lambda_{\mathrm{D}}^\star, \delta )  }{\partial \lambda_\ell} = 0.
 \]
 As a result, $\mathrm{R}_{\mathrm{D}} (\lambda_{\mathrm{D}}^{\star}, \delta) = \inf_{\lambda \in \mathbb{R}^2_+} \mathrm{R}_{\mathrm{D}}(\lambda, \delta) = \mathrm{R}(\lambda_{\mathrm{D}}^{\star}, \delta) =  \inf_{\lambda \in \mathbb{R}^2_+} \mathrm{R}(\lambda, \delta)$ and 
 \[
 \mathcal{I}(\delta) = \mathbb{E}[Y_2 ]-\mathbb{E} [Y_1]+  \inf _{\lambda \in \mathbb{R}_{+}^2} \mathrm{R}_{\mathrm{D}}(\lambda, \delta) = \mathcal{I}_{\mathrm{D}}(\delta) .
 \]

	\textbf{Step 3. } We show that \Cref{eq:ATE-I_D-I-equality-condition} incorporates the case when $\delta_1 = 0$ or $\delta_2 = 0$.   From \Cref{prop:ATE-Mahalanobis-Dualform} (ii), we know the following statements hold.
	\begin{itemize}
		\item When  $\delta_1 > 0$ and $\delta_2 = 0$,  $\mathcal{I}_{\mathrm{D}}(\delta) = \mathcal{I}(\delta)$ if and only if $V_{1, XY} = 0$.
		
		\item When $\delta_1 = 0$ and $\delta_2 > 0$,  $\mathcal{I}_{\mathrm{D}}(\delta) = \mathcal{I}(\delta)$ if and only if $V_{2, XY} = 0$. 
		
		\item When $\delta_1 = \delta_2 = 0$,  $\mathcal{I}_{\mathrm{D}}(\delta) = \mathcal{I}(\delta) = \mathcal{I}_{ \mathrm{D}, 0}$.
	\end{itemize}
	We will see that \Cref{eq:ATE-I_D-I-equality-condition} incorporates all these cases.
	\begin{itemize}
		\item When $\delta_1 > 0$ and $\delta_2 = 0$, \Cref{eq:ATE-I_D-I-equality-condition} is equivalent to $V_{1, XY} = 0$. 
		
		\item When $\delta_1 = 0$ and $\delta_2 > 0$, \Cref{eq:ATE-I_D-I-equality-condition} is equivalent to $V_{2, XY} = 0$. 
		
		\item When $\delta_1 = \delta_2 = 0$, \Cref{eq:ATE-I_D-I-equality-condition} is satisfied always.
	\end{itemize}
This completes the proof.

\subsubsection{Proof of \texorpdfstring{\Cref{prop:ATE-Mahalanobis-Dualform_iii}}{}}

The continuity of $\mathcal{I}_{\mathrm{D}}$ can be seen from the \Cref{prop:ATE-Mahalanobis-Dualform_i} or \Cref{thm:ID-continuity}. Next, we show $\mathcal{I}$ is continuous on $\mathbb{R}_+^2$ by verifying the conditions of \Cref{thm:I-continuity}.
Obviously, $\boldsymbol{d}_{\mathcal{S}_\ell}\left(s_\ell, s_\ell^{\prime}\right) = \sqrt{ c_\ell (s_\ell, s_\ell^\prime ) } $ defines a norm on $\mathcal{S}_\ell = \mathbb{R}^{q+1} $. Define a function $\rho_\ell:\mathcal{Y}_{\ell} \times  \mathcal{Y}_{\ell} \rightarrow \mathbb{R}_+$ as
	\[
	\rho_\ell (y_\ell, y_\ell^\prime) = \left(y_{\ell}-y_{\ell}^{\prime}\right)^{\top} V_{\ell, Y Y}^{-1}\left(y_{\ell}-y_{\ell}^{\prime}\right).
	\]
	In fact, it is not difficult to see 
	\[
	\rho_\ell (y_\ell, y_\ell^\prime)  = \min_{(x_\ell, x^\prime_\ell) \in \mathcal{X}_\ell \times  \mathcal{X}_\ell } \left(s_{\ell}-s_{\ell}^{\prime}\right)^{\top} V_{\ell}^{-1}\left(s_{\ell}-s_{\ell}^{\prime}\right)  \leq c_\ell(s_\ell, s_\ell^\prime), \quad \forall s_\ell, s_\ell^\prime \in \mathcal{S}_\ell.
	\]
	Moreover, $\rho_\ell^{1/2}$ is a norm on $\mathcal{Y}_\ell$ and the triangle inequality implies
	\[
	\rho_\ell^{1/2}\left(y_\ell, y_\ell^{\prime}\right) \leq   \rho_\ell^{1/2} \left(y_\ell, y_\ell^{\star}\right)+ \rho_\ell^{1/2} \left(y_\ell^{\star}, y_\ell^{\prime}\right), \quad \forall   y_\ell, y_\ell^{\prime}, y_\ell^{\star} \in \mathcal{Y}_\ell.
	\]
	As a result, we must have 
	\[
	\rho \left(y_\ell, y_\ell^{\prime}\right) \leq 2 \left[  \rho_\ell \left(y_\ell, y_\ell^{\star}\right)+ \rho_\ell \left(y_\ell^{\star}, y_\ell^{\prime}\right) \right], \quad \forall   y_\ell, y_\ell^{\prime}, y_\ell^{\star} \in \mathcal{Y}_\ell.
	\]
	We verified the functions $\rho_1$ and $\rho_2$ satisfy \Cref{assumption:quasi-metric} with respect to Mahalanobis distances. Recall $f(y_1,y_2, x) = y_1 - y_2$ and define a concave function $\Psi: \mathbb{R}^2 \rightarrow \mathbb{R}_{+}$ as
	\[
	\Psi: (a_1, a_2) \mapsto    V_{1, YY}^{1/2} a^{1/2} + V_{2, YY}^{1/2} a_2^{1/2}.
	\]
	Since $\left|y_\ell- y_\ell^\prime \right|^2  = V_{\ell, YY} \rho_{\ell}\left(y_{\ell}, y_{\ell}^{\prime}\right)$ and $\rho_\ell \leq c_\ell$, then 
	\[
	\begin{aligned}
	f\left(y_1, y_2, x\right)-f\left(y_1^{\prime}, y_2^{\prime}, x^{\prime}\right)   & \leq  |y_1- y_1^\prime| +  |y_2- y_2^\prime|  \\
	 &  \leq \sum_{\ell=1}^2 V_{\ell, YY}^{1/2} \rho_{\ell}^{1/2 }\left(y_{\ell}, y_{\ell}^{\prime}\right) = \Psi(\rho_1(y_{1}, y_{1}^{\prime}) , \rho_2(y_{2}, y_{2}^\prime ) ) \\
	 & \leq   \Psi \left( c_1(s_{1}, s_{1}^{\prime}) , \rho_2(y_{2}, y_{2}^\prime ) \right).
	\end{aligned}
	\]
	Similarly, we can show 
	\[
	f(y_1, y_2, x) -  f(y_1^\prime, y_2^\prime, x^\prime)  \leq   \Psi \left ( \rho_1(y_1, y_1^\prime), c_2(s_2, s_2^\prime)  \right).
	\]
\Cref{thm:I-continuity} implies the continuity of $\mathcal{I}$ on $\mathbb{R}^2_+$.

\subsection{Proofs in \texorpdfstring{\Cref{sec:RobustWelfare-revisited}}{}}

\subsubsection{Proof of \texorpdfstring{\Cref{prop:RW-L1-with-X-L2}}{}} \label{appendix:RW-L1-with-X-L2}

We prove \Cref{prop:RW-L1-with-X-L2_i} using a technique similar to \citet{adjaho2022}.  For any $s_\ell = (y_\ell, x_\ell) \in \mathcal{S}_\ell$, we have
	\begin{align*}
		(f_{\mathcal{S}})_{\lambda}(s_1, s_2) 
		& =
		\sup_{ x^\prime \in 	\mathcal{X}} \sup_{ (y_1^\prime , y_2^\prime) \in 	\mathcal{Y}_1 \times \mathcal{Y}_2}  
		\left\{ - y_2' d(x') - y_1' \left[1 - d(x')\right] -  \sum_{1\leq \ell \leq2} \lambda_\ell \left[ | y_\ell - y_\ell' | + \| x_\ell - x' \|_2 \right]  \right\}   \\
		& =
		\sup_{ x^\prime \in \mathcal{X}}
		\Bigg\{  
		\left[  \sup_{y_2' \in \mathcal{Y}_2} \{ - y_2' d(x') -  \lambda_2 |y_2 - y_2'| \} + \sup_{y_1' \in \mathcal{Y}_1} \{ - y_1' (1 - d(x')) -  \lambda_1 |y_1 - y_1'| \}  \right] \\
		& \quad \quad  -  \sum_{1\leq \ell \leq2} \lambda_\ell  \| x_\ell - x' \| 
		\Bigg\}.
	\end{align*}
We note that
	\begin{align*}
		\sup_{y_2' \in \mathcal{Y}_2} \{ - y_2' d(x') -  \lambda_2 |y_2 - y_2'| \}
		=
		\begin{cases}
			\infty	            & \text{ if } 0 \le \lambda_2 < 1 \\
			- y_2 d(x')   &  \text{ if } \lambda_2 \ge 1
		\end{cases},
	\end{align*}
	and
	\begin{align*}
		\sup_{y_1' \in \mathcal{Y}_1} \{ - y_1' (1 - d(x')) -  \lambda_1 |y_1 - y_1'| \}
		=
		\begin{cases}
			\infty	            & \text{ if } 0 \le \lambda_1 < 1 \\
			- y_1 (1 - d(x'))   &  \text{ if } \lambda_1 \ge 1
		\end{cases}.
	\end{align*}
Therefore,  we have for $\lambda_1 \ge 1$ and $\lambda_2 \ge 1$
	\begin{align*}
		(f_{\mathcal{S}})_{\lambda}(s_1, s_2) 
		& = \sup_{ x^\prime \in \mathcal{X}}
		\left\{  - y_2 d(x') - y_1 (1 - d(x')) 
		-  \sum_{1 \leq \ell \leq 2} \lambda_\ell  \| x_\ell - x' \| 
		\right\} \\
		& = 
		- \min\{ y_2 + \varphi_{\lambda, 1} (x_1, x_2), y_1 +  \varphi_{\lambda, 0} (x_1, x_2)   \},
	\end{align*}
where 
\[
\varphi_{\lambda, d} (x_1, x_2)  = \min_{u \in \mathcal{X} : d(u) = d}   \sum_{1\leq \ell \leq 2} \lambda_\ell \| x_\ell - u \|_2,
\]
for $d \in \{0,1\}$. 	If $\lambda_1 < 1$ or $\lambda_2 < 1$, then $(f_{S})_{\lambda}(s_1, s_2) = \infty$. As a result, we have
	\begin{align*}
		\mathrm{RW}(d) & =
		\inf_{\gamma \in \Sigma(\delta)} \mathbb{E}[Y_2 d(X) + Y_1 (1 - d(X))]  = - \inf_{\lambda \in \mathbb{R}_{+}^2} \left[ \langle \lambda, \delta \rangle  + \sup_{\pi \in \Pi(\mu_{13}, \mu_{23})}  \int_{\mathcal{V} } (f_{S})_{\lambda} \,	d \pi \right] \\
		& =
		- \inf_{\lambda \in [1, \infty)^2} \left[ \langle \lambda, \delta \rangle  + \sup_{\pi \in \Pi(\mu_{13}, \mu_{23})}  \int_{\mathcal{V} } - \min\{ y_2 + \varphi_{\lambda, 1} (x_1, x_2), y_1 +  \varphi_{\lambda, 0} (x_1, x_2)   \} \, d \pi(v) \right] \\
		& =
		\sup_{\lambda \in [1, \infty)^2 }
		\left[
		\inf_{\pi \in \Pi(\mu_{13}, \mu_{23})}  \int_{\mathcal{V} } \min \left\{ y_2 + \varphi_{\lambda, 1} (x_1, x_2), y_1 +  \varphi_{\lambda, 0} (x_1, x_2) \right \} \, d \pi(v) - \langle \lambda, \delta \rangle
		\right].
	\end{align*}
Next, we show \Cref{prop:RW-L1-with-X-L2_ii}. Recall the set $\widetilde{\mathit{\Pi} }$ defined in the proof of \Cref{prop:ATE-Mahalanobis-Dualform_ii}. Here, $\widetilde{\mathit{\Pi} }$ is the set of all the probability measures concentrate on $\left\{ (y_1, x_1, y_2, x_2) \in \mathbb{R}^{2d +2} : x_1 = x_2  \right\}$. Consider the following derivation:
	\begin{align*}
 \mathrm{RW}(d)     & = 
		\sup_{\lambda_0 \ge 1, \lambda_2 \ge 1}
		\left[
		\inf_{\pi \in \Pi(\mu_{13}, \mu_{23})}  \int_{\mathcal{V}} \min\{ y_2 +  \varphi_{\lambda, 1} (x_1, x_2), y_1 +  \varphi_{\lambda, 0} (x_1, x_2) \} \, d \pi(v) -  (\lambda_1 + \lambda_2) \delta_0
		\right] \\
        &   \le \sup_{\lambda_1 \ge 1, \lambda_2 \ge 1}
		\left[
		\inf_{\pi \in \widetilde{\mathit{\Pi} } }   \int_{\mathcal{V}} \min\{ y_2 +  \varphi_{\lambda, 1} (x_1, x_2), y_1 +  \varphi_{\lambda, 0} (x_1, x_2) \} \, d \pi(v) -  (\lambda_1 + \lambda_2) \delta_0
		\right].
        \end{align*}
Recall the functions $h_0$ and $h_1$ defined in \Cref{prop:RW_AC}, we notice that for all $(y_1,x_1,y_2,x_2) \in \widetilde{\mathit{\Pi} }$,
\[
\varphi_{\lambda, \ell} (x_1, x_2) = (\lambda_1 + \lambda_2) h_\ell(x_1), \quad \forall \ell = 1,2.
\]
As a result, we have
  \begin{align*}
		\mathrm{RW}(d) &\leq 
		\sup_{\lambda_1 \ge 1, \lambda_2 \ge 1}
		\left[
		\inf_{\pi \in \mathcal{F}(\mu_{13}, \mu_{23})}
		\int_{\mathcal{S}} \min\{ y_2 +  \varphi_{\lambda, 1} (x), y_1 +  \varphi_{\lambda, 0} (x) \} \, d \pi(s) -  (\lambda_1 + \lambda_2) \delta_0
		\right] \\
		& =
		\sup_{\eta \ge 2}
		\left[ \inf_{\pi \in \mathcal{F}(\mu_{13}, \mu_{23})}
		\int_{\mathcal{S}} \min\{ y_2 +  \eta h_{1} (x), y_1 +  \eta h_{0} (x) \} \, d \pi(s) -  \eta \delta_0
		\right] \\
		& \le
		\sup_{\eta \ge 1}
		\left[
		\inf_{\pi \in \mathcal{F}(\mu_{13}, \mu_{23})}
		\int_{\mathcal{S}} \min\{ y_2 +  \eta h_{1} (x), y_1 +  \eta h_{0} (x) \} \, d \pi(s) -  \eta \delta_0
		\right] \\
         & =_{(1)}  \sup_{\eta \ge 1} \left[
		\inf_{\pi \in \mathcal{F}(\mu_{13}, \mu_{23})}
		\mathbb{E}_X\left[\mathbb{E}\left(\min\{ Y_2 - Y_1 +  \eta h_{1} (X),  \eta h_{0} (X) \} |X\right)\right] + \mathbb{E}(Y_1)-  \eta \delta_0
		\right] \\
           & = \sup_{\eta \ge 1}
		\left[
		\int_{\mathcal{S}} \min\{ y_2 +  \eta h_{1} (x), y_1 +  \eta h_{0} (x) \} \, d \pi^*(s) -  \eta \delta_0
		\right]\\
  & = \mathrm{RW}_0(d)
	\end{align*}
 where equation (1) follows from Proposition 2.17 in \citet{Santambrogio_2015} and the concavity of $y \mapsto \min\{y +  \eta h_{1} (x),  \eta h_{0} (x)\}$  (see also Section 4.3.1 in \citet{adjaho2022}).

\subsection{Proofs in \texorpdfstring{\Cref{sec:multi-marginals}}{}}   

We provide a brief sketch of proofs in  \Cref{sec:multi-marginals}.
\subsubsection{Proof of \texorpdfstring{\Cref{thm:ID-duality-multi-marginals}}{}}

Similarly to the proof of \Cref{thm:ID-duality}, it is sufficient to derive the dual reformulation of $\mathcal{I}_{\mathrm{D}}(\delta)$ for $\delta \in \mathbb{R}^L_{++}$.  Let $\mathcal{P}_{\mathrm{D}}$ denote the set of $\gamma \in \mathcal{P}(\mathcal{V})$ that satisfies $\boldsymbol{K}_\ell(\mu_\ell, \gamma_\ell) < \infty$ for all $\ell \in [L]$ and $\int_{\mathcal{V} } g d \gamma > -\infty$.  Taking the Legendre transform on $\mathcal{I}_{\mathrm{D}}$ yields that any $\lambda \in \mathbb{R}^2_+$,
\begin{align*}
\mathcal{I}_{\mathrm{D} }^\star (\lambda) :=&\sup_{\delta \in \mathbb{R}_{+}^L }  \left\{ \mathcal{I}_{\mathrm{D} }(\delta)  - \langle \lambda, \delta \rangle \right\}  =\sup_{\delta \in \mathbb{R}_{+}^L } \sup_{\gamma \in \Sigma_{\mathrm{D}}(\delta)  } \left\{ \int_\mathcal{V} g \, d\gamma - \langle \lambda, \delta \rangle \right\}\\
= &   \sup_{\gamma \in \mathcal{P}_{\mathrm{D}} }  \underbrace{\left\{ \int_\mathcal{V} g \, d\gamma -  \sum_{\ell \in [L]}\lambda_\ell \boldsymbol{K}_\ell(\mu_\ell,\gamma_\ell) \right\}}_{:= I_{\mathrm{D}, \lambda}[\gamma]} = \sup_{\gamma \in \mathcal{P}_{\mathrm{D}} } I_{\mathrm{D}, \lambda}[\gamma].
\end{align*}
Using \Cref{lemma:decomposability-V-multi-marginals} and the similar seasoning as the proof of \Cref{thm:ID-duality}, we can show
\begin{align*}
\mathcal{I}_{\mathrm{D} }^\star (\lambda) = \sup_{\gamma \in \mathcal{P}_{\mathrm{D}} } I_{\mathrm{D}, \lambda}[\gamma] = \sup_{\pi \in  \Gamma(\Pi, \varphi_\lambda ) } \int_{\mathcal{V} \times \mathcal{V} } \varphi_\lambda \, d\pi =   \sup_{\pi \in \Pi(\mu_1, \ldots, \mu_L)} \int_{\mathcal{V} } \, g_\lambda d\pi. 
\end{align*}
The desired result follows from \Cref{lemma: Legendre_transform}.

\subsubsection{Proof of \texorpdfstring{\Cref{thm:I-duality-multi-marginals}}{}}

Similarly to the proof of \Cref{thm:I-duality}, it is sufficient to derive the dual reformulation of $\mathcal{I}(\delta)$ for $\delta \in \mathbb{R}^L_{++}$.  Let $\bar{\mathcal{P}}$ denote the set of $\gamma \in \mathcal{P}(\mathcal{S})$ that satisfies $\boldsymbol{K}_\ell\left(\mu_{\ell,L}, \gamma_{\ell,L} \right) < \infty$ for all $\ell \in [L]$ and $\int_{\mathcal{S} } f d \gamma > -\infty$.  Taking the Legendre transform on $\mathcal{I}$ yields that any $\lambda \in \mathbb{R}^2_+$,
\begin{align*}
\mathcal{I}^\star (\lambda) :=&\sup_{\delta \in \mathbb{R}_{+}^L }  \left\{ \mathcal{I}(\delta)  - \langle \lambda, \delta \rangle \right\}  =\sup_{\delta \in \mathbb{R}_{+}^L } \sup_{\gamma \in \Sigma(\delta)  } \left\{ \int_\mathcal{V} f d\gamma - \langle \lambda, \delta \rangle \right\}\\
= &   \sup_{\gamma \in \bar{\mathcal{P} }  }  \underbrace{\left\{ \int_\mathcal{V} g d\gamma -  \sum_{\ell \in [L]}\lambda_\ell \boldsymbol{K}_\ell(\mu_\ell,\gamma_\ell) \right\}}_{:= I_{ \lambda}[\gamma]} = \sup_{\gamma \in \bar{\mathcal{P} }  } I_{ \lambda}[\gamma].
\end{align*}
For notational simplicity, we write $\Pi := \Pi\left(\mu_{1,L+1}, \ldots, \mu_{L, L+1} \right)$. Using \Cref{lemma:decomposability-S-multi-marginals} and the similar seasoning as in the proof of \Cref{thm:I-duality}, we can show
\begin{align*}
\mathcal{I}^\star (\lambda) = \sup_{\gamma \in \bar{\mathcal{P}} } I_{\lambda}[\gamma] = \sup_{\pi \in  \Gamma(\Pi, \phi_\lambda ) } \int_{\mathcal{V} \times \mathcal{V} } \varphi_\lambda d\pi =   \sup_{\pi \in \Pi} \int_{\mathcal{V} } f_\lambda d\pi. 
\end{align*}
The desired result follows from \Cref{lemma: Legendre_transform}.

\subsubsection{Proof of \texorpdfstring{\Cref{prop:RW-L1-with-X-L2-multi-marginals}}{}}

The proof is identical to that of \Cref{prop:RW-L1-with-X-L2}.

\end{document}